\documentclass[12pt,a4paper,reqno]{amsart}  
\usepackage{amsmath, amssymb, amsthm} 
\usepackage[english]{babel} 
\usepackage{fullpage}  
\usepackage{enumerate}
\usepackage[all]{xy}
\usepackage{graphicx}
\usepackage{mathtools}
\usepackage{mathrsfs} 
\usepackage{hyperref}
\numberwithin{equation}{section}

\title{The gauge group of a noncommutative principal bundle and twist deformations} 
\author{Paolo Aschieri, Giovanni Landi, Chiara Pagani}
\address[]{\textit{Paolo Aschieri}  \newline \indent 
Universit{\`a} del Piemonte Orientale, 
\newline \indent 
Dipartimento di Scienze e Innovazione Tecnologica
\newline \indent   viale T.~Michel~11,~15121~Alessandria,~Italy,
\newline \indent  and INFN Torino, via P.~Giuria~1, 10125~Torino,~Italy,
\newline \indent  and Arnold-Regge Centre, Torino, via P.~Giuria~1, 10125~Torino,~Italy}
\email{paolo.aschieri@uniupo.it}

\address[]{\textit{Giovanni Landi}  \newline \indent 
Matematica, Universit\`a di Trieste, \newline \indent
Via A. Valerio, 12/1, 34127  Trieste, Italy, \newline \indent 
Institute for Geometry and Physics (IGAP) Trieste, Italy \newline \indent 
and INFN, Trieste, Italy
}
\email{landi@units.it}

\address[]{\textit{Chiara Pagani} \newline \indent   
Universit{\`a} del Piemonte Orientale, \newline \indent 
Dipartimento di Scienze e Innovazione Tecnologica
\newline \indent   viale T.~Michel~11,~15121~Alessandria,~Italy
\newline \indent and INDAM-GNSAGA}
\email{chiara.pagani@uniupo.it}

\theoremstyle{plain}
\newtheorem{thm}{Theorem}[section]
\newtheorem{lem}[thm]{Lemma}
\newtheorem{prop}[thm]{Proposition}

\newtheorem{cor}[thm]{Corollary}

\theoremstyle{definition}
\newtheorem{defi}[thm]{Definition}

\theoremstyle{remark}
\newtheorem{ex}[thm]{Example}
\newtheorem{Q}{Question/Comment}
\newtheorem{rem}[thm]{Remark}

\newcommand{\nn}{\nonumber}
\newcommand{\ot}{\otimes}
\newcommand{\beq}{\begin{equation}}
\newcommand{\eeq}{\end{equation}}

\newcommand{\ra}{\rightarrow}
\newcommand{\tra}{\triangleright}
\renewcommand{\1}{1}
\newcommand{\bbK}{\mathbb{K}}
\newcommand{\id}{\mathrm{id}}

\newcommand{\A}{\mathcal{A}}
\newcommand{\G}{\mathcal{G}}
\renewcommand{\Q}{\mathcal{Q}}
\newcommand{\M}{\mathcal{M}}
\newcommand{\C}{\mathcal{C}}
\renewcommand{\O}{\mathcal{O}}

\newcommand{\IR}{\mathbb{R}}
\newcommand{\IC}{\mathbb{C}}
\newcommand{\IZ}{\mathbb{Z}}
\newcommand{\can}{\chi}

\newcommand{\hg}{H_\cot}

\renewcommand{\cot}{\gamma}
\newcommand{\co}[2]{\cot\left({#1}\ot{#2}\right)}
\newcommand{\coin}[2]{\bar\cot\left({#1}\ot{#2}\right)}

\newcommand{\mt}{\cdot_\cot}
\newcommand{\mtco}{\bullet_\cot}

\newcommand{\zero}[1]{{#1}_{\scriptscriptstyle{(0)}}}
\newcommand{\one}[1]{{#1}_{\scriptscriptstyle{(1)}}}
\newcommand{\two}[1]{{#1}_{\scriptscriptstyle{(2)}}}
\newcommand{\three}[1]{{#1}_{\scriptscriptstyle{(3)}}}
\newcommand{\four}[1]{{#1}_{\scriptscriptstyle{(4)}}}
\newcommand{\five}[1]{{#1}_{\scriptscriptstyle{(5)}}}
\newcommand{\six}[1]{{#1}_{\scriptscriptstyle{(6)}}}
\newcommand{\seven}[1]{{#1}_{\scriptscriptstyle{(7)}}}
\newcommand{\eight}[1]{{#1}_{\scriptscriptstyle{(8)}}}
\newcommand{\nine}[1]{{#1}_{\scriptscriptstyle{(9)}}}
\newcommand{\ten}[1]{{#1}_{\scriptscriptstyle{(10)}}}

\newcommand{\tuno}[1]{{#1}^{\scriptscriptstyle{<1>}}}
\newcommand{\tdue}[1]{{#1}^{\scriptscriptstyle{<2>}}}

\newcommand{\uzero}[1]{{#1}^{\scriptscriptstyle{(0)}}}
\newcommand{\uone}[1]{{#1}^{\scriptscriptstyle{(1)}}}
\newcommand{\utwo}[1]{{#1}^{\scriptscriptstyle{(2)}}}

\newcommand{\Aut}[1]{\mathrm{Aut}_B(#1)}

\newcommand{\F}{\mathsf{F}}
\newcommand{\f}{\mathsf{f}}
\newcommand{\g}{\mathsf{g}}
\newcommand{\Hom}{{\rm{Hom}}}

\newcommand{\Ad}{\mathrm{Ad}}
\newcommand{\R}{\mathcal{R}}

\renewcommand{\r}{R}
\newcommand{\qu}[2]{Q{\!} \left({#1}\ot{#2}\right)}
\newcommand{\ru}[2]{\r_{\!}\left({#1}\ot{#2}\right)}
\newcommand{\ruin}[2]{\bar{\r}_{\!}\left({#1}\ot{#2}\right)}
\newcommand{\rug}[2]{\r_\cot \left({#1}\ot{#2}\right)}

\newcommand{\tpr}{\,{\raisebox{.28ex}{\rule{.6ex}{.6ex}}}\,}
\renewcommand{\bot}{\boxtimes}
\newcommand{\bs}{\underline{S}}
\newcommand{\bh}{{\underline{H}}}
\newcommand{\bc}{{qc}}

\newcommand{\qcr}{{\A}^{(H,\r)}_\bc}
\newcommand{\qcrp}{{\A}^{(K,\r')}_\bc}
\newcommand{\qcrpp}{{\A}^{(H\otimes K,\r'')}_\bc}
\newcommand{\qcrc}{{\A}^{(\hg,\r_\gamma)}_\bc}

\newcommand{\hpr}{{\:\cdot_{^{_{{{\:\!\!\!\!\!-}}}}}}}

\newcommand{\RA}{\A^H}
\newcommand{\bdelta}{{\underline\delta}}

\begin{document}

\begin{abstract} 
We study noncommutative principal bundles (Hopf--Galois extensions) in
the context of coquasitriangular Hopf algebras and their monoidal
category of  comodule algebras. 
When the total space is quasi-commutative, and thus the base space subalgebra is central, we define the gauge group as
the group of vertical automorphisms or equivalently as the group of equivariant algebra maps.
We study Drinfeld twist (2-cocycle) deformations of
Hopf--Galois extensions and show that the gauge group of the twisted
extension is isomorphic to the gauge group of the initial extension.
In particular, noncommutative principal bundles arising via twist deformation of
commutative principal bundles have classical gauge group. We illustrate
the theory with a few examples.
\end{abstract}

\maketitle
\tableofcontents

\section{Introduction}

Noncommutative gauge theories have emerged in different contexts in
mathematics and physics.
The present study aims at a better understanding of the
geometric structures underlying these theories.
The relevant framework is that of noncommutative principal bundles that 
we approach from the algebraic perspective of Hopf--Galois extensions.
These first emerged as a generalization of classical Galois field
extensions and were later recognised to be suitable for a description
of principality of actions in algebraic and noncommutative geometry.
Aiming at the noncommutative differential geometry of Hopf--Galois
extensions, with a theory of connections and their moduli spaces,
in this paper we study the notion of group of
noncommutative gauge transformations.  \\

For Hopf--Galois extensions the group of gauge transformations was considered in \cite{brz-maj}
and further studied in \cite{brz-tr},  (see also \cite{strong}).
An unusual feature of these works is that the group 
there defined is bigger than one would expect. 
Classically, the group of gauge transformations of a principal $G$-bundle $P\to M$ is the group of vertical bundle  automorphisms $P\to
P$ or of  $G$-equivariant maps $P\to G$, 
for the adjoint action of $G$ onto itself. 
The pull-back of these maps to the algebra of functions
gives $O(G)$-equivariant algebra maps $\O(G)\to \O(P)$. 
However, with the definition of these papers, for the Hopf--Galois extension $\O(P)$
one would get the bigger group of $O(G)$-equivariant unital and convolution
 invertible linear maps $\O(G)\to \O(P)$, which are not necessarily algebra maps. This suggests for gauge transformations  to retain some algebra map property.

As a way of clarification, let us consider the simplest case of the
bundle $G\to \{*\}$ over a point. This is an elementary example of a Galois object, 
that is a Hopf-Galois extension of the ground field 
\cite[Def.7.11]{kassel-review}. 
Then, gauge transformations, as $O(G)$-equivariant algebra maps $\O(G)\to\O(G)$, are a copy of $G$ itself. Thus they make a much smaller group than 
that of all $O(G)$-equivariant unital and convolution invertible linear maps  $\O(G)\to \O(G)$.
Similarly, infinitesimal gauge transformations are left invariant
vector fields, giving then the Lie algebra $\frak{g}$ of $G$. In the dual picture
they act on $\O(G)$ as derivations, that is, as infinitesimal algebra maps.
Without requiring infinitesimal automorphisms of $\O(G)$ to be
derivations, one obtains the whole universal enveloping algebra
$U(\frak{g})$. With quantum groups one can consider their universal
enveloping algebra, or construct a quantum Lie algebra of left
invariant vector fields that are deformed derivations (\` a la
Woronowicz \cite{wor}). 

Quantum Lie algebras have been used for infinitesimal gauge
transformations for gauge field theories for example in \cite{Cast, Cast1}.
A further independent argument in favour of a theory of noncommutative gauge
groups that does not drastically depart from the 
classical one comes from the Seiberg-Witten map between 
commutative and noncommutative gauge theories  \cite{Seiberg:1999vs}. This map (initially
considered for noncommutative gauge theories in the
context of string theory and related flux compactification) 
establishes a one-to-one correspondence between commutative and
noncommutative gauge transformations and hence points to a 
noncommutative gauge group that is a deformation of the classical one.

Other studies suggesting a view on gauge transformations as (deformed) algebra maps
are those on noncommutative instanton moduli spaces, for example \cite{LS}
for instantons on the principal bundle on the noncommutative four sphere
$S^4_\theta$ \cite{LS0}.  There the dimension of the moduli space survives the $\theta$-deformation (see also \cite{BL}). 
\\

In the present paper
we study the group of  gauge transformations as the group of
equivariant algebra maps. By way of comparison, let us anticipate here our results for the case of the Galois object 
$(\O(G)_{\bullet_\cot}, \O(G)_\cot)$, with the quantum structure group $\O(G)_\cot$ coacting on the total space
algebra $\O(G)_{\bullet_\cot}$ (in general this is not a trivial Hopf--Galois extension, but only a cleft one). 
We change the multiplication in $\O(G)_\cot$ by considering the braided Hopf algebra
$\underline{\O(G)_\cot}\;\!$ so that we find gauge transformations as
$\O(G)_\cot$-equivariant algebra maps  $\underline{\O(G)_\cot}\to
\O(G)_{\bullet_\cot}$. These are deformed algebra maps with respect to
the initial product in $\O(G)_\cot$.

In \cite{ppca} noncommutative principal bundles 
were revisited and considered in a categorical perspective.
Since noncommutative principal bundles with
Hopf algebra (quantum structure group) $H$ are $H$-comodule algebras
$A$ with a canonically given $H$-equivariant map $\chi$, required to be
invertible,  the basic category where to study these objects is
that of $H$-(co)representations, that is that of $H$-comodules.
A second category where they can be studied is
that of $H$  and $K$-comodules, with $K$ an ``external Hopf
algebra of symmetries'', a Hopf algebra associated
with the automorphisms of a Hopf--Galois extension (classically a group acting via
equivariant maps differing from the identity on the base space). 
In \cite{ppca}  
Hopf Galois extensions were studied in these two categorical settings and it was 
shown that Drinfeld twists (Hopf algebra 2-cocycles) deform functorially
Hopf--Galois extensions to Hopf--Galois extensions.
Considering a twist on the Hopf algebra $H$ leads to a deformation
of the fibers of the principal bundle; considering a twist on the
external symmetry Hopf algebra $K$ leads to a
deformation of the base space. Combining twists on $H$ and on $K$ one
obtains deformations of both the fibers and the base space.
Many examples were provided starting from commutative principal bundles.
\\

In the present paper we work within the representation
category of an Hopf algebra $H$, with $A$ an $H$-comodule algebra.
We study gauge transformations of noncommutative
principal bundles $B=A^{co H}\subseteq A$ with quantum structure
group $H$, noncommutative total space $A$ and commutative base
$B$. Examples motivating the interest in this case  include also 
quantum group gauge theory on lattices, that is related
to models quantizing the algebra of observables of Chern-Simons theory
\cite{Meus-Wise}.

In a sequel paper 
we consider Hopf algebras $H$
and Hopf--Galois extensions in a category of $K$-comodules.
In this  richer context we study gauge transformations of
Hopf--Galois extensions with noncommutative bases (for instance noncommutative
tori and related manifolds). A further motivation
for these studies comes from the relevance of noncommutative gauge
field theories for string theory and related
compactifications. There $U(N)$ gauge theories on noncommutative tori
naturally emerge \cite{Connes:1997cr}. In that context already
considering simple Lie groups (like $SU(N)$ or $SO(N)$) is
problematic, one way out being 
the use of the Seiberg-Witten map between commutative and
noncommutative gauge theories \cite{Seiberg:1999vs, WESS}, another
approch possibly being the Hopf--Galois one we are pursuing. 
\\

Before considering gauge transformations as algebra
maps, we study {conditions for} the canonical map $\chi$ to be an algebra map. 
The natural  categorical setting for addressing this question is that of
coquasitriangular Hopf algebras. Indeed in this context the category of  $H$-comodule
algebras is a monoidal category.  We show that the
canonical map is a morphism in the category when the
multiplication in $A$ is a morphism as well
(we call such comodule algebras  quasi-commutative). 
This implies that the base $B$ is commutative.
Canonically associated with a coquasitriangular Hopf algebra $H$ we
have the braided Hopf algebra $\bh$. The gauge group is first defined
as the set of $H$-equivariant (unital) algebra maps $\bh\to A$ and then proven to be a group.  
A second approach is to define the gauge group as the set of $H$-equivariant
algebra maps $A\to A$ that restrict to the identity on $B$. This corresponds
to the classical picture of vertical authomorphisms of a principal bundle.
Here too we prove that these maps form a group. These two definitions
of gauge group are then shown to be equivalent, and the theory is
illustrated with examples. \\

We study next Drinfeld twist deformations of Hopf--Galois
extensions and of their gauge groups. 
We refine the results in \cite{ppca} to the case of coquasitriangular
and cotriangular  Hopf algebras.
A twist on $H$ induces an equivalence of the associated monoidal
categories, and braided Hopf algebras are twisted to braided Hopf algebras.
The equivalence of the possible different twisting procedures is
proven via a map $\Q$. This map is related to the natural isomorphism
that gives the equivalence of
the categories  of Hopf algebra modules and of twisted Hopf algebra 
modules as closed monoidal categories.
These results allow us to conclude that Hopf--Galois extensions
$B=A^{coH}\subseteq A$, with
canonical map $\chi$ that is an algebra map, are twisted to
Hopf--Galois extensions $B=A_\gamma^{co H_\gamma}\subseteq A_\gamma$ with 
canonical map $\chi_\gamma$ that is an algebra map.
The twist functor is then applied to the two equivalent
characterizations of the gauge group of a Hopf--Galois extension. By
using again the map $\Q$ we show that the initial gauge
group and the twisted one are isomophic. In particular
cleft (but not necessarily trivial) Hopf--Galois extensions obtained twisting
trivial Hopf--Galois extensions have isomorphic gauge groups. \\

Finally, we consider tensor products of noncommutative
principal bundles and study the
resulting gauge groups.
Combining the tensor product construction and the twisting procedure
we construct interesting examples.
 In particular we study the noncommutative
principal fibration of spheres $S^7\times_\gamma S^1\to S^4$ on the
commutative 4-sphere. The structure group is $U_q(2)$, a cotriangular
deformation of the unitary group, and the gauge group of this
Hopf--Galois extension is isomorphic to the direct product of the 
classical gauge group of the
instanton bundle on the 4-sphere $S^4$ with the group of $U(1)$-valued
functions on $S^4$. \\

We mention that 
the idea of using a braiding that renders the Hopf-Galois canonical map an algebra homomorphism, so as to enable one to define gauge transformations, was already put forward in \cite{du96}. There it is shown that the algebra structure on the left-hand side of the canonical map, that is induced from the tensor algebra on its right-hand side, is given by a braiding of Hopf algebras. Then, for cosemisimple commutative Hopf algebras this braiding was used to define gauge transformations as algebra homomorphisms. 
While the braiding in \cite{du96} is written using the inverse of the canonical map, in the present paper 
the braiding comes from the coquasitriangular structure of the class of Hopf algebras considered. 
 
\subsection{Background material} \label{sec:bg-mat}~\\[.5em]
We work in the category of $\bbK$-modules, for $\bbK$ a fixed commutative field with unit $\1_\bbK$ or the ring of formal power series in a variable $\hbar$ over a field.
Much of what follows can be generalised to $\bbK$ a commutative unital ring. 
We denote simply by $\otimes$ the tensor product over $\bbK$.  
All algebras will be over $\bbK$ and assumed to be unital and associative. 
The product in an algebra $A$ is denoted by $m_A: A \ot A \ra A$,  $a \ot b \mapsto ab$ and 
 the unit map by $\eta_A: \bbK \ra A$, with $\1_A:= \eta_A(\1_\bbK)$ the unit element.  
Morphisms of algebras will be assumed to be unital. 
Analogously all coalgebras will be over $\bbK$ and assumed to be counital and coassociative. 
The  coproduct and counit of a coalgebra $C$ are denoted by $\Delta_C: C \ra C \ot C$ and  $\varepsilon_C: C \ra \bbK$ respectively. 
We use the standard Sweedler notation for the coproduct:  $\Delta_C(c)= \one{c} \ot \two{c}$ (sum understood), for all $ c \in C$, 
and for its iterations: $\Delta_C^n=(\id \ot  \Delta_C) \circ\Delta_C^{n-1}: c \mapsto \one{c}\ot \two{c} \ot 
\cdots \ot c_{\scriptscriptstyle{(n+1)\;}}$, $n >1$.
We denote by $*$ the convolution product in the  dual $\bbK$-module
$C':=\mathrm{Hom}(C,\bbK)$, $(f * g) (c):=f(\one{c})g(\two{c})$, for all $c \in C$, $f,g \in C'$.
For a Hopf algebra $H$, we denote by $S_H: H \ra H$ its antipode. 
For all these maps we will omit the subscripts which refer to
the co/algebras involved when no risk of confusion can occur.  
We simply write $V \in \mathcal{C}$ for an object $V$ in a category $\mathcal{C}$, and $\Hom_\mathcal{C}(-,-)$ for morphisms between any two objects.
Finally, all monoidal categories in this paper will have a trivial associator, hence we can unambiguously write $V_1\otimes V_2\otimes \cdots \otimes V_n$ for the tensor product of $n$ objects.

Given a bialgebra (or a Hopf algebra) $H$, we denote by $\M^H$ the category of right $H$-comodules:
a right $H$-comodule is a $\bbK$-module $V$ with a $\bbK$-linear map
$\delta^V:V\to V\otimes H$ (a right $H$-coaction) such that 
\beq \label{eqn:Hcomodule}
(\id\otimes \Delta)\circ \delta^V = (\delta^V\otimes \id)\circ \delta^V~,\quad 
(\id\otimes \varepsilon) \circ \delta^V =\id~. 
\eeq
In Sweedler notation we write $\delta^V:V\to V\otimes H$, $v\mapsto \delta^V =
\zero{v}\otimes \one{v}$, and the right $H$-comodule properties
\eqref{eqn:Hcomodule}  read, for all $v\in V$,
$$
\zero{v} \otimes \one{(\one{v})}\otimes \two{(\one{v})} = \zero{(\zero{v})}\otimes \one{(\zero{v})} \otimes \one{v}=: \zero{v} \ot\one{v} \ot \two{v} ~,\qquad
\zero{v} \,\varepsilon (\one{v}) = v~.
$$
A morphism between  $V, W \in  \M^H$ is a $\bbK$-linear map $\psi:V\to W$ which is $H$-equivariant: 
$\delta^W\circ \psi=(\psi\otimes\id)\circ \delta^V$.  We equivalently
say that  $\psi:V\to W$ is an $H$-comodule map.
 
In fact, $\M^H$ is a monoidal category:  given
$V,W\in \M^H$, the tensor product
$V\otimes W$ of $\bbK$-modules is an object in $ \M^H$ with  the right $H$-coaction  
\begin{align}\label{deltaVW}
 \delta^{V\otimes W} :V\otimes W \longrightarrow  V\otimes W\otimes H~, \quad
 v\otimes w \longmapsto \zero{v}\otimes \zero{w} \otimes 
 \one{v}\one{w} ~.
\end{align}
The unit object in $\M^H$ is $\bbK$ with coaction $\delta^{\bbK}$
given by the unit map $\eta_H:\bbK\to \bbK\otimes H \simeq H$.
 
We denote by $\A^H$ the category of
right $H$-comodule algebras: 
a right $H$-\textbf{comodule algebra} is an algebra  $A$  which is 
a right $H$-comodule such that the multiplication and unit of $A$ are morphisms of $H$-comodules. 
This is equivalent to requiring the coaction $\delta^A: A\to A\otimes H$ to be
a morphism of unital algebras (where $A\otimes H$ has the usual tensor
product algebra structure):  for all $a,a^\prime\in A \, $,
$$
\delta^A(a\,a^\prime) =\delta^A(a)\,\delta^A(a')~~\;,\quad
\delta^A(\1_A) = \1_A\otimes \1_H~\; .
$$
Morphisms in $\A^H$  are $H$-comodule maps which are also algebra maps.
 
We denote by  ${\C}^H$ the category of right $H$-comodule coalgebras: a right $H$-\textbf{comodule coalgebra}  
is a coalgebra $C$ which is a right $H$-comodule and such that the
coproduct and the counit are morphisms of $H$-comodules 
that is, for each $c \in C$
\beq\label{com-co}
\zero{(\one{c})} \ot \zero{(\two{c})} \ot \one{(\one{c})} \one{(\two{c})} 
=
\one{(\zero{c})} \ot \two{(\zero{c})} \ot \one{c} \, , 
\quad  
\varepsilon(\zero{c}) \one{c}=\varepsilon(c) 1_H \, .
\eeq
 Morphisms in ${\C}^H$ 
are $H$-comodule maps which are also coalgebra maps.

Let $H$ be a bialgebra and let $A\in\A^H$.
An $(A,H)$-\textbf{relative Hopf module}
$V$ is a right $H$-comodule with a compatible left $A$-module structure,
that is  the left  $A$-action $\tra_V$ is a morphism of $H$-comodules such that 
the following diagram commutes
\beq\label{AMHdiag}
\xymatrix{
\ar[d]_-{\tra_V} A\otimes V \ar[rr]^-{\delta^{A \ot V}} &&\ar[d]^-{\tra_V \otimes \id} A \otimes V\otimes H\\
V \ar[rr]^-{\delta^V}&& V\otimes H
}
\eeq
Explicitly, for all $a\in A$ and $v\in V$,
\beq\label{eqn:modHcov} 
\zero{(a \tra_V v)} \ot \one{(a \tra_V v)} = \zero{a} \tra_V \zero{v} \ot \one{a}\one{v} ~. 
\eeq
A morphism of 
$(A,H)$-relative Hopf modules is a morphism of
right $H$-comodules  which is also an $A$-linear map, that is a morphism of left $A$-modules.
We denote by $ {}_{A}\M^H$ the category of 
$(A,H)$-relative Hopf modules.
In a similar way one defines the categories of relative Hopf modules ${{\mathcal M}_A}^{\!H}$ for $A$ acting on the right, 
 and ${{}_{E}{\mathcal M}_A}^{\!H}\,$ for right $A$ and left $E$
 compatible actions, with $E\in \A^H$.

\section{Hopf--Galois extensions for coquasitriangular Hopf algebras}

We consider noncommutative principal bundles as Hopf--Galois
extensions. These are $H$-comodule algebras
$A$ with a canonically defined map $\chi: A\ot_B A\to A\ot H$ which is required to be
invertible.
We first consider the 
category of $(A,H)$-relative Hopf modules  and understand  within this monoidal category  the
notion of Hopf--Galois extension, that is the bijectivity of the 
map $\chi$.
This is done in \S 2.1, where  we see that the monoidal structure forces $H$
in $A\ot H$ to be considered as an $H$-comodule with the adjoint
action $\Ad$, denoted $\bh$.  In \S 2.2 we consider the case of  $H$
coquasitriangular, where the
category of $H$-comodule algebras is monoidal. 
The braided Hopf algebra $\bh$ with the adjoint action $\Ad$ is an
$H$-comodule algebra so that both $ A\ot A$ and $A\ot H$ are $H$-comodule
algebras.
The canonical map is then proven to be an algebra map provided $A$ is
quasi-commutative.

 \subsection{Hopf--Galois extensions}\label{sec:hge}~\\[-1.3em]
\begin{defi} \label{def:hg}
Let $H$ be a Hopf algebra and let $A\in\A^H$ with coaction $\delta^A$.
Consider the subalgebra $B:= A^{coH}=\big\{b\in A ~|~ \delta^A (b) = b \ot \1_H \big\} \subseteq A$ of coinvariant elements 
(elements invariant under the $H$-coaction) and let
$A \ot_B A := A \ot A / \, {\!\langle a \ot b a' - a b \ot a'
  \rangle}_{a, a' \in A, \, b \in B}$
be the corresponding balanced tensor product.
 The extension $B\subseteq 
A$ is called an $H$-\textbf{Hopf--Galois extension} provided the
(so-called) {\bf canonical  map}
\begin{align}\label{canonical}  \can := (m \ot \id) \circ (\id \otimes_B \delta^A ) : A \otimes_B A  \longrightarrow A \ot H~  ,  \quad a' \ot_B a \longmapsto a' a_{\;(0)} \ot a_{\;(1)} 
\end{align} 
is bijective. 
\end{defi}

The canonical map $\can$ is a morphism in the category ${_A{\mathcal  M}_A}^{H}$ of relative Hopf modules \cite{ppca}.
Both $A \otimes_B A$ and $A \ot H$ are objects in  ${_A{\mathcal  M}_A}^{H}$. 
The left $A$-module structures are given by the left multiplication on the first factors while the 
right $A$-actions are given by
$$
(a \ot_B a')a'':=  a \ot_B a'a'' \quad \mbox{and} \quad (a \ot h) a' := a \zero{a'} \ot h \one{a'}~.
$$
As for the $H$-comodule structure, the tensor product $A\ot A$ has the natural 
right $H$-coaction induced by the monoidal structure of $\M^H$,
as in \eqref{deltaVW}, 
\beq\label{AAcoact}
\delta^{A\otimes A}: A\otimes A\to A\otimes A\otimes H, \quad a\otimes a'  \mapsto  \zero{a}\otimes \zero{a'} \otimes  \one{a}\one{a'} \, 
\eeq 
for all $a,a'\in A$. This descends to the quotient $A\ot_B A$ because $B\subseteq A$ is the
subalgebra of $H$-coinvariants.  
Similarly,  $A \ot H$ is endowed with the tensor product coaction, where   we  regard the Hopf algebra $H$ as a right $H$-comodule
with the right adjoint $H$-coaction
\beq\label{adj}
\mathrm{Ad} :  
h \longmapsto \two{h}\otimes S(\one{h})\,\three{h} ~.
\eeq
The right $H$-coaction on $A \ot H$ 
is then given again as in \eqref{deltaVW} by 
\beq\label{AHcoact}
\delta^{A\otimes H}(a\otimes h) = \zero{a}\otimes \two{h} \otimes \one{a}\,S(\one{h})\, \three{h} \in A \ot H \ot H~
\eeq 
for all $ a\in A,~ h \in H$. 
Both $A\ot_B A$ and  $A\otimes H$ are shown to be objects in  ${_A{\mathcal M}_A}^{H}$ with respect to these structures and 
$\can$ to be a morphism in the category ${_A{\mathcal  M}_A}^{H}$ of
relative Hopf modules (see \cite{ppca} for details, and see Appendix \ref{app:hge} for a comparison with other descriptions of the map $\can$ as a morphism of relative Hopf modules).

Since the canonical map $\can$ is left $A$-linear, its inverse is
determined by the restriction $\tau:=\chi^{-1}_{|_{1 \ot H}}$, named \textbf{translation map},
\begin{eqnarray*}
\tau= {\chi^{-1}}{|_{_{1 \ot {H}}}}: H\to A\otimes_B A , 
\quad h \mapsto \tuno{h} \ot_B \tdue{h}~.
\end{eqnarray*}
We recall for later use the following properties of the translation map (see Appendix \ref{app:t}).
$$
(\id \ot_B \delta^A)\circ \tau =(\tau \ot \id) \Delta \quad, \quad
 (  \tau  \ot  S) \circ \mbox{flip} \circ \Delta= (\id \ot \mbox{flip}) \circ (\delta^A \ot_B \id) \circ \tau  ~,
$$
that  on the generic element $h \in H$ respectively read
\beq\label{p4} 
\tuno{h} \ot_B \zero{\tdue{h}} \ot \one{\tdue{h}} = \,\tuno{\one{h}} \ot_B \tdue{\one{h}} \ot 
\two{h} ~,
\eeq
\beq\label{p1} 
~~ \tuno{\two{h}}  \ot_B \tdue{\two{h}} \ot S(\one{h}) = \zero{\tuno{h}} \ot_B {\tdue{h}}  \ot \one{\tuno{h}}
~.
\eeq

A Hopf--Galois extension is {\bf cleft} if there exists a convolution
invertible morphism of $H$-comodules $j : H \to A$ (the {\bf cleaving
map}), where $H$ has coaction $\Delta$. This is equivalent to an isomorphism $A \simeq B \otimes H$ of left
$B$-modules and right $H$-comodules, where $B \otimes H$ is a left
$B$-module via multiplication on the left and a right H-comodule via $\id \otimes \Delta$.
A Hopf--Galois extension is a trivial extension if the cleaving map is
also an algebra map.

Commutative Hopf--Galois extensions typically arise when considering
principal $G$-bundles. Twisted versions will be described in Section \ref{sec:HGtwist} below.

\begin{ex}\label{affinecase}
Let $G$ be a semisimple affine algebraic group and 
let $\pi: P \ra P/G$ be a principal $G$-bundle with $P$ and $P/G$ affine varieties. 
Let $H=\O(G)$ be the dual coordinate Hopf algebra and $A=\O(P)$, $B=\O(P/G)$ the corresponding coordinate algebras. 
Let  $B\subseteq A$ be the subalgebra of functions
constant on the fibers, we then have $B=A^{co H}$ and $\O(P\times_{P/G}P)
\simeq A\otimes_B A$. The bijectivity of the map $P\times G\to P\times_{P/G}
P$, $(p,g)\mapsto (p,pg)$, characterizing principal bundles in
this context, corresponds to the bijectivity of the canonical
map $\chi: A\otimes_B A\to A\otimes H$, thus showing that $B=A^{co
  H}\subseteq A$ is a Hopf--Galois extension (see e.g. \cite[\S 8.5]{mont} and \cite[Thm.3.1.5]{Dasc}). An important notion is that of the classical translation map $t: P\times_{P/G} P\to G$,
$(p,q)\mapsto t (p,q)$ where $q=p_{\,}t(p,q)$. Properties \eqref{p4} and \eqref{p1}
then read: $t(p,qg)=t(p,q)g$ and $t(pg,q)=g^{-1}t(p,q)$.
\qed\end{ex}

\subsection{Coquasitriangular Hopf algebras}~\\[.6em]
We begin by recalling basic properties of coquasitriangular Hopf algebras; for
proofs we refer e.g. to \cite[Ch.10]{KS} or \cite[Ch.2]{Majid}. We
then study the monoidal category of comodule algebras
$(\RA,\bot)$ and the braided Hopf algebra $\underline H\in \RA$.

\begin{defi}\label{def:cqt}
A bialgebra $H$ is called \textbf{coquasitriangular} (or dual
quasitriangular) if it is endowed with a linear form $\r:H \ot H \ra \bbK$ such that 
\begin{enumerate}[(i)]
\item $\r$ is  invertible for the convolution product, with inverse denoted by
 $\bar \r$;  
\item $m_{op} = \r * m * \bar \r$, that is, for all $h,k \in H$,
\beq\label{gira-H}
kh= \ru{\one{h}}{\one{k}} {\two{h}}{\two{k}} \ruin{\three{h}}{\three{k}} ;
\eeq
\item\label{iiiR} $\r\circ (m \ot \id)=\r_{13}* \r_{23}$ \quad and \quad $\r\circ (\id \ot m)=\r_{13}* \r_{12}$, \\
where $\r_{12}(h \ot k \ot l)=\ru{h}{k}\varepsilon(l)$  and similarly
 for $\r_{13}$ and $\r_{23}$;\\
in components, for all $h,k,l \in H$, these conditions read 
\beq\label{eq-iii}
\ru{hk}{l}=\ru{h}{\one{l}}\ru{k}{\two{l}} \quad \mbox{and} \quad \ru{h}{kl}= \ru{\two{h}}{k}\ru{\one{h}}{l} .
\eeq
 \end{enumerate}
The linear form $\r$ is called a \textbf{universal $R$-form} of $H$. 
If $(H,R)$ is coquasitriangular then so is $(H,\bar{R}_{21})$ where $R_{21}(h\otimes k):=R(k\otimes h)$ for all $h,k\in H$.
A coquasitriangular bialgebra $(H,\r)$ is called \textbf{cotriangular} if $\r= \bar{\r}_{21}$.

A Hopf algebra $H$ is called co(quasi)triangular if it is such as a bialgebra.
\end{defi}

\begin{ex}\label{trivialex}
Any commutative bialgebra $H$ 
is cotriangular with (trivial) universal $R$-form $R=\varepsilon \ot \varepsilon$.
\qed\end{ex}

Note that if a  coquasitriangular bialgebra $(H,R)$ is cocommutative,
then it is commutative. Nonetheless, this does not imply that $R$ is  trivial:
\begin{ex} \label{ex:CZzero} Let $H=\mathbb{C}\mathbb{Z}$ be the group
  Hopf-algebra of the abelian group $\mathbb{Z}$. It is the
  commutative and cocommutative Hopf algebra generated by an
  invertible element $g$, $\mathbb{C}\mathbb{Z}=\mathbb{C}[g,g^{-1}]$,
  with $\Delta(g)=g\otimes g$, $\varepsilon(g)=1$, $S(g)=g^{-1}$.  For
  every complex number $q\neq 1$, this Hopf algebra is
  coquasitriangular with $R$-form $R_q(g^n,g^m)=q^{-nm}$. 
\qed\end{ex}

\begin{ex}
The FRT bialgebras $\O(G_q)$, noncommutative deformations of the coordinate algebra 
on the Lie groups $G$ of the $A,B,C,D$ series, are
coquasitriangular  \cite{FRT}.
\qed\end{ex}
\begin{ex}\label{ex:rug}
If $(H,\r)$ is a coquasitriangular Hopf algebra and $\cot: H \ot H \ra
\bbK$ is a $2$-cocycle on $H$, then the Hopf algebra $\hg$  with 
twisted product and antipode (see \S \ref{sec:twists-hopf}) is also coquasitriangular with universal $R$-form
\beq\label{rug}
\r_\cot :=\cot_{21} *\r* \bar{\cot}: h \ot k \longmapsto
\co{\one{k}}{\one{h}}\ru{\two{h}}{\two{k}} \coin{\three{h}}{\three{k}} ~,
\eeq
where $\bar\cot:H\ot H\to \bbK$ is the convolution inverse of
  $\cot$. The (convolution) inverse of $R$ is $\bar\r_\cot :=\cot *\bar\r * \bar{\cot}_{21}$. If  $(H,\r)$ is  cotriangular, then  $(\hg,\r_\cot)$ is cotriangular.
\qed\end{ex}

From its definition, it follows that the $R$-form of a coquasitriangular bialgebra $(H,R)$ is
normalized, that is, for all $h\in H$,
\beq\ru{1}{h } = \varepsilon(h) = \ru{h}{1},\label{Rnormalized}
\eeq 
and that it satisfies the Yang--Baxter-Equation
$\r_{12}* \r_{13}* \r_{23}= \r_{23} *\r_{13}*\r_{12}$,
 that is, for all $h,k,l \in H$,
\beq\label{YB}
\ru{\one{h}}{\one{k}} \ru{\two{h}}{\one{l}} \ru{\two{k}}{\two{l}}
= 
 \ru{\one{k}}{\one{l}} \ru{\one{h}}{\two{l}} \ru{\two{h}}{\two{k}}~.
\eeq
If in addition $H$ is a Hopf algebra, then for all $h,k\in H$ we have
\beq\label{prop-R}
\ru{S(h)}{k}= \ruin{h}{k} \; ; \quad
\ruin{h}{S(k)}= \ru{h}{k} \; , 
\eeq
from which it also follows  $\ru{S(h)}{S(k)}= \ru{h}{k}$. Furthermore, 
the antipode $S$ of $H$ is invertible with inverse $S^{-1}=u_\r * S * \bar{u}_\r$, where
\beq
u_\r:h \longmapsto \ru{\one{h}}{S(\two{h})} \quad ;\quad 
\bar{u}_\r:h \longmapsto \ruin{S(\one{h})}{\two{h}} \; .
\eeq
Finally, when $(H,\r)$  is coquasitriangular, the monoidal category of right
$H$-comodules $\M^H$ is braided monoidal with 
braiding given by the $H$-comodule isomorphisms
\beq\label{braiding}
\Psi^{R}_{V,W}: V \ot W  \longrightarrow W \ot V ~, \quad v\ot w \longmapsto  
\zero{w}\ot \zero{v} ~\ru{\one{v}}{\one{w}}~.
\eeq
  
We can now recall a key feature of coquasitriangular Hopf algebras:
tensor products of comodule algebras are comodule algebras and 
tensor products of comodule algebra maps are again comodule
algebra maps.
\begin{prop}\label{prop:tpr0}
Let $(H,\r)$ be a coquasitriangular bialgebra. Let  $(A,\delta^A), (C,\delta^C) \in \A^H$
be right $H$-comodule algebras. 
Then the $H$-comodule $A \ot C$ (with tensor product coaction $\delta^{A \ot C}: a \ot c \mapsto \zero{a} \ot \zero{c} \ot \one{a}\one{c}$
as in \eqref{deltaVW}) is a right $H$-comodule algebra when endowed with the product
\beq\label{tpr}
(a \ot c) \tpr (a' \ot c'):= a~  \Psi^R_{C,A} (c \ot a') c'=
a\zero{a'} \ot \zero{c} c' ~\ru{\one{c}}{\one{a'}} \, .
\eeq
Moreover, when $\phi: A\to E$ and $\psi: C\to F$ 
are $H$-equivariant algebra maps, that is morphisms of $H$-comodule algebras, then so is 
the map $\phi\otimes\psi : A \ot C \to E \ot F$, $a\otimes c\mapsto
\phi(a)\otimes \psi(c)$,
where $A\otimes C$ and $E\otimes F$ are endowed with the $\tpr$-products in \eqref{tpr}.
\end{prop}
\begin{proof}
Associativity of the product in $A\otimes C$ is straighforward. The
coaction $\delta^{A \ot C}$ is also easily seen to be an algebra map because of
\eqref{gira-H}, (an explicit proof can be found in \cite{maj-braid}, or 
in \cite[Lem.31 \S 10.3]{KS}).
 The statement about morphisms follows by writing $\phi\otimes\psi = (\phi\otimes
\id_F)\circ  (\id_A\otimes \psi)$ and showing that $\id_A\otimes \phi$ and $\psi\otimes
\id_F$ are both algebra maps (this is due to $H$-equivariance of $\phi$ and $\psi$).
\end{proof}
The $H$-comodule algebra $(A \ot C, \tpr)$ is called the
{\bf{braided tensor product algebra}} of $A$ and $C$; we
denote it by $A \bot  C$, and write $a \bot c\in A \bot
 C$ for $a\in A$, $c\in C$. 
Similarly, we denote by
 $\phi\bot\psi:=\phi\otimes\psi: A\bot C\to E\bot F$ the
 $H$-equivariant algebra map resulting from the tensor
 product of the  $H$-equivariant algebra maps  $\phi: A\to E$ and $\psi: C\to F$.

\begin{prop}\label{prop:tpr}
Let $(H,\r)$ be a coquasitriangular bialgebra.  The category $\RA$ of
$H$-comodule algebras endowed with the above defined tensor product $\bot$ becomes a
 monoidal category, denoted $(\RA,\bot)$.
\end{prop}
\begin{proof}
Let $(A,\delta^A),
(C,\delta^C),  (E,\delta^E)\in \RA$. If we forget the algebra structure  the tensor product $\bot$
becomes the associative tensor product of $H$-comodules
of the monoidal category $({\M}^H,\otimes)$, where $(A\ot C)\ot E \simeq A\ot (C\ot E)$ (as $\bbK$-modules). 
We only need to show that this isomorphism is compatible with the algebra structure, so that it is an isomorphism in ${\A}^H$. 
The equality
$$
(( a \ot c) \ot e ) \tpr (( a' \ot c') \ot e' ) = ( a \ot ( c\ot e )) \tpr ( a' \ot ( c' \ot e' ) )
$$ 
follows from the explicit expression \eqref{tpr} for the product and the property \eqref{eq-iii} of the $R$-form.  The units in $(A\bot C)\bot E$ and in $A\bot (C\bot E)$ trivially coincide.
The unit object in $(\A^H,\bot)$
is $\bbK$, seen as an $H$-comodule algebra 
(since $\delta^\bbK =\eta_H$ is an algebra map).
\end{proof}

\begin{rem}
The braiding \eqref{braiding} of $\M^H$ defines a braiding 
\beq\label{braidingA}
\Psi^R_{A,C}: A \ot C  \longrightarrow C \ot A ~, \quad a\ot c \longmapsto 
\zero{c}\ot \zero{a} ~\ru{\one{a}}{\one{c}} 
\eeq 
for the monoidal subcategory $(\A^H,\bot)$ if and only if $R$ is cotriangular. Indeed the 
$H$-comodule isomorphisms \eqref{braidingA} are algebra maps if and only if $R$ is cotriangular. 
Hence requiring the monoidal category  $(\A^H,\bot)$ to be braided
with braidings \eqref{braidingA} is more specifically requiring it to
be a  symmetric monoidal category, that is $({\Psi^R_{A,C}})^{\!-1}=\Psi^R_{C,A}$.
\end{rem}

An important role in the following is played by the right $H$-comodule $\bh:=(H,\Ad)$, with the right adjoint coaction 
$\Ad: \bh \ra \bh \ot H$, $h \mapsto \two{h} \ot S(\one{h})\three{h}$
as defined in \S \ref{def:hg}.
The notation $\bh$ is used when considering $H$ as an $H$-comodule
rather than a Hopf algebra.
If $H$ is coquasitriangular, one can  endow 
$\bh$ with a product that makes $\underline H$ an $H$-comodule
algebra and a braided Hopf algebra (see e.g. \cite[\S 10.3.2]{KS}): 

\begin{prop}\label{prop:hpr} Let $(H,\r)$ be a coquasitriangular Hopf algebra.
The right $H$-comodule $\underbar{H}=(H,\Ad)$ becomes an $H$-comodule algebra when endowed with the product
\beq\label{hpr}
h \hpr k := \two{h}\two{k} 
\ru{S(\one{h})\three{h}}{S(\one{k})} 
\eeq
and unit $\eta: \bbK\ra \bh$ given, as linear map, by the unit
$\eta_H$ of $H$. 
\end{prop} 
Vice versa, the product in the Hopf algebra $H$ is recovered from that in $\bh$ as
\beq\label{hpr-inv}
hk= \two{h} \hpr \two{k} ~\ru{S(\one{h})\three{h}}{\one{k}}.
\eeq
Given an $H$-comodule $V\in\M^H$, we denote
 by $\bdelta^V: V\to V\ot \bh$ the coaction $\delta^V: V\to V\otimes H$
  thought as a linear  map from $V$ to $V\ot \bh$. It is easy to
  show that $\bdelta^V$ is an $H$-comodule map, that
  is, the commutativity
  of the diagram
\beq\label{dHcommap}
\xymatrix{
\ar[d]_-{\delta^V} V \ar[rr]^-{\bdelta^{V}} &&\ar[d]^-{\delta^{V \otimes \bh}} V \otimes \bh\\
V\ot H \ar[rr]^-{{~~~~~}{\bdelta^V\ot \,\id}{~~~~~}} &&\;V\otimes \bh\ot H \; .
}\eeq
Furthermore, given an  $H$-comodule algebra $A\in \A^H$, with $(H,R)$ coquasitriangular,
we denote by $\bdelta^A: A\to A\bot \bh$  the $H$-comodule map 
$\bdelta^A: A\to A\ot \bh$.

\begin{prop}\label{prop:deltaund} 
Let  $(H,R)$ be a coquasitriangular Hopf algebra, $A$ an algebra in
$\M^H$ with coaction $\delta^A: A\to A\ot H$. The map
$\bdelta^A: A\to A\bot \bh$ is an algebra map if and only if $\delta^A: A\to A\ot H$ is an algebra map.
\end{prop}
\begin{proof}
Let $\tpr$ be the product in the braided tensor 
product algebra $A\bot\bh$ given in \eqref{tpr}. 
Then, for all $a,c\in A$, 
$\bdelta^A(a)\tpr 
\bdelta^A(c)=\bdelta^A(ac)\Leftrightarrow
\delta^A(a)\delta^A(c)=\delta^A(ac)$; indeed, 
\begin{align*}
\bdelta^A(a) \tpr \bdelta^A(c) &= (\zero{a} \ot \one{a})
\tpr  (\zero{c} \ot \one{c})= \zero{a}  \zero{c} \ot \two{a} \hpr
                                  \two{c}
                                  \ru{S(\one{a})\three{a}}{\one{c}}
\\
&=\zero{a} \zero{c} \ot \one{a}\one{c} =    (\zero{a} \ot \one{a})
  (\zero{c} \ot \one{c})\\
& = \delta^A(a) \delta^A(c)~.
 \end{align*}
Furthermore, unitality of $\bdelta^A$ is equivalent to unitality of
$\delta^A$ since the two maps are the same as linear maps.
\end{proof}

\begin{defi} \label{def:bbH}
Let  $(H,R)$ be a coquasitriangular Hopf algebra. An $H$-comodule algebra 
$(L, m_L, \eta_L, \delta^L)$ and $H$-comodule coalgebra $(L, \Delta_L,\varepsilon_L, \delta^L)$ is called 
a \textbf{braided bialgebra associated with $H$} if it is a bialgebra
in the braided monoidal category $(\M^H,\otimes, \Psi^R)$ of $H$-comodules. That is,  
$\varepsilon_L: L\to \bbK$ is an algebra map, $\eta_L:\bbK\to L$ a coalgebra map and moreover $\Delta_L$ is an algebra map 
with respect to the product $m_L$ in $L$ and 
the product $m_{L \bot L} = (m_L \otimes m_L)  \circ (\id_L\otimes \Psi^R_{L,L} \otimes \id_L)$ in $L\bot L$ (as given in \eqref{tpr}), that is  
\beq\label{cop-braid}
\Delta_L\circ m_L  = m_{L \bot L} \circ (\Delta_L\otimes\Delta_L) \, .
\eeq
The braided biagebra  $L$ is a braided Hopf algebra if there is a map
$S_L : L\to L $, called antipode or braided antipode, that satisfies the antipode property
 (of being the convolution inverse of the identity
$\id: L\to L$):
\beq 
m_L\circ (\id_L\otimes S_L)\circ \Delta_L=
\eta_L\circ \varepsilon_L =
m_L\circ (S_L\otimes\id_L)\circ \Delta_L ~, \label{bantipode}
\eeq
and that in addition is an $H$-comodule map. 
\end{defi}
For later use we recall that the antipode $S_L: L\to L$ of a braided Hopf algebra $L$ is a  braided anti-algebra map
and a braided anti-coalgebra map
\beq\label{SmmRSS}
S_L\circ m_L=m_L\circ  \Psi^R_{L,L}\circ (S_L\ot S_L)~~,~~~\Delta_L\circ
S_L=(S_L\ot S_L)\circ \Psi^R_{L,L}\circ \Delta_L~~.
\eeq

\begin{ex}\label{ex:bh}
\textit{The braided Hopf algebra $\underline{H}$ of a coquasitriangular Hopf
  algebra $(H,R)$.} 
Recall that for any Hopf algebra $H$, the data  $(H, \Delta,
\varepsilon, \Ad)$ is an $H$-comodule coalgebra and that for $H$
coquasitriangular  $(\bh, \hpr,
\, \Ad)$ is an $H$-comodule algebra. These two structures define the braided Hopf algebra 
$$
(\bh, \hpr, \eta, \Delta, \varepsilon, \bs, \Ad). 
$$  
Here, as $H$-comodule maps, both $\eta:\bh\ra \bbK$ and $\Delta:
\bh \ra \bh \bot \bh$ are the same as the counit and coproduct in $H$,
with now $\Delta$ considered as an algebra map for the product $\hpr$ in $\bh$ and the $\tpr$-product $m_{\bh \bot \bh}$ in $\bh \bot \bh$.  
The antipode $\bs:=S_{\bh}:
\bh \ra \bh$ can be shown to be given, for all $h\in H$, by
\beq\label{bs} 
\bs(h):= S(\two{h}) \ru{S^2(\three{h})S(\one{h})}{\four{h}} ~.
\eeq 
\qed\end{ex} 
\begin{lem}
The braided Hopf algebra $\underline H$ is {\bf{braided commutative},}
that is, for all $h,k \in \bh$, its product satisfies
\beq  \label{H-braid-comm} 
\two{k} \hpr \two{h} \,\ru{S(\one{k})\three{k}}{\one{h}} =
\one{h} \hpr \two{k} \,\ruin{\two{h}}{S(\one{k})\three{k}}~,
\eeq
this equation being equivalent to \eqref{gira-H}.
\end{lem}
\begin{proof}
By substituting \eqref{hpr-inv} in \eqref{gira-H} 
and using the basic properties of  the $R$-form $\r$, one obtains:
\begin{align*}
kh\,\,= \,\,&\ru{\one{h}}{\one{k}} {\two{h}}{\two{k}} \ruin{\three{h}}{\three{k}} 
\\
\iff&
\two{k} \hpr \two{h} ~\ru{S(\one{k})\three{k}}{\one{h}}
\\ & = \ru{\one{h}}{\one{k}} 
\three{h} \hpr \three{k} ~\ru{S(\two{h})\four{h}}{\two{k}}
\ruin{\five{h}}{\five{k}} 
\\
\iff&
\two{k} \hpr \two{h} ~\ru{S(\one{k})\three{k}}{\one{h}}
= \ru{\two{h}}{\one{k}} 
\one{h} \hpr \two{k} ~
\ruin{\three{h}}{\three{k}} 
\\
\iff&
\two{k} \hpr \two{h} ~\ru{S(\one{k})\three{k}}{\one{h}}
= 
\one{h} \hpr \two{k} ~
\ruin{\two{h}}{S(\one{k})\three{k}}
\end{align*}
where in the last passage we used the analogous properties in \eqref{eq-iii} for $\bar{R}$.
Thus $\bh$ is braided commutative.
\end{proof}

\begin{lem}
The following conditions are equivalent to \eqref{H-braid-comm} 
\begin{align}
& h \hpr k =
\three{k} \hpr \two{h}\,
\ru{S(\two{k})\four{k}}{\one{h}}\ru{\three{h}}{S(\one{k})\five{k}} \label{H-braid-commQ} ~,\\
& h \hpr k =
\three{k} \hpr \three{h}
\,\ru{S(\two{h})\four{h}}{S(\one{k})\five{k}}\qu{\one{h}}{S(\two{k})\four{k}} ~ , \label{H-braid-commQQ}
\end{align}
where $Q$ is the convolution product $Q=R_{21} \ast R$.
\end{lem}
\begin{proof}
The implication  \eqref{H-braid-commQ} $\,\Rightarrow$  \eqref{H-braid-comm} is
proven by substituting in the right hand side of \eqref{H-braid-comm} the expression for 
$\one{h}\hpr\two{k}$ given by relation \eqref{H-braid-commQ}. 

For the converse  implication \eqref{H-braid-comm}  $\,\Rightarrow$  \eqref{H-braid-commQ}  
we compute
\begin{align*}
h\hpr k\, & = \one{h} \hpr \two{k} \, \varepsilon(\two{h}) \varepsilon(S(\one{k})\three{k}) \\
& = \one{h} \hpr \three{k} \, \ruin{\two{h}}{S(\two{k})\four{k}} \ru{\three{h}}{S(\one{k})\five{k}} \\
& =\three{k}\hpr \two{h}\,
    \,\ru{S(\two{k})\four{k}}{\one{h}}\ru{\three{h}}{S(\one{k})\five{k}}~,
\end{align*}
where in the last equality we used  \eqref{H-braid-comm}. 

Equivalence of \eqref{H-braid-commQQ} and
\eqref{H-braid-commQ} is shown by using the explicit convolution product $Q$:
\begin{align*}
h \hpr k 
& = \three{k} \hpr \three{h} \,\ru{S(\two{h})\four{h}}{S(\one{k})\five{k}}\qu{\one{h}}{S(\two{k})\four{k}} \\
& = \four{k} \hpr \four{h} \,\ru{S(\three{h})\five{h}}{S(\one{k})\seven{k}}
\ru{S(\three{k})\five{k}}{\one{h}} \ru{\two{h}}{S(\two{k})\six{k}} \\
& = \three{k} \hpr \four{h} \,\ru{\two{h} S(\three{h})\five{h}}{S(\one{k})\five{k}}
\ru{S(\two{k})\four{k}}{\one{h}} \\
& = \three{k} \hpr \two{h}\,
\ru{S(\two{k})\four{k}}{\one{h}}\ru{\three{h}}{S(\one{k})\five{k}}, 
\end{align*}
where we used the basic property \eqref{eq-iii} of the $R$-form in the third equality.
\end{proof}
 
We next introduce the notion of quasi-commutative algebra $A\in \A^H$ and provide a few examples.
\begin{defi}\label{def:qcomm}
Let $(H, \r)$ be a coquasitriangular Hopf algebra. A right $H$-comodule algebra $A
\in \A^H$ is {\bf{quasi-commutative}} (for the coquasitriangular structure $\r$ of $H$) if  
\beq\label{qcomm1}
 m_A = m_A \circ {({\Psi}^R_{A,A})}^{\!-1} \, , \qquad a c = \zero{c} \zero{a} ~ \ruin{\one{c}}{\one{a}}
\eeq
or equivalently 
\beq\label{qcomm2}
m_A =m_A \circ \Psi^R_{A,A} \, , \, \qquad a c= \zero{c} \zero{a}~ \ru{\one{a}}{\one{c}}
\eeq
for all $a,c \in A$. 
We denote by  $\qcr$  the full subcategory of $\A^H$ of
quasi-commutative comodule algebras (for the coquasitriangular structure $\r$), where morphisms are  $H$-comodule algebra morphisms. 
\end{defi}
\noindent
The first expression \eqref{qcomm1} implies the second \eqref{qcomm2}:
$$
a c = \zero{a} \zero{c} \varepsilon(\one{a} \one{c}) = \zero{a} \zero{c} \ruin{\one{a}}{\one{c}} \ru{\two{a}}{\two{c}} =
\zero{c} \zero{a} \ru{\one{a}}{\one{c}}~.
$$
Similarly the second expression implies the first one. For future use (see Theorem \ref{prop:canam})
we also prove a third equivalent expression:
\beq\label{qcomm3}
\zero{c}\zero{a} \ot \two{c} \ru{S(\one{c})\three{c}}{\one{a}}
= \zero{a}\zero{c} \ot \one{c} \ru{\two{c}}{\one{a}} \, .
\eeq
Indeed, \eqref{qcomm1} implies \eqref{qcomm3}:
\begin{align*}
\zero{a}\zero{c} \ot \one{c} \ru{\two{c}}{\one{a}} & = \zero{c} \zero{a} \ruin{\one{c}}{\one{a}} \ot \two{c} \ru{\three{c}}{\two{a}} \\
& = \zero{c} \zero{a} \ru { S(\one{c})}{\one{a}} \ot \two{c} \ru{\three{c}}{\two{a}} \\
& = \zero{c}\zero{a} \ot \two{c} \ru{S(\one{c})\three{c}}{\one{a}}. 
\end{align*}
On the other hand, $\id \ot \varepsilon$ applied to \eqref{qcomm3} and
the normalization $\ru{1}{h } = \varepsilon(h)$ give \eqref{qcomm2}. 

The quasi-commutativity property of $A\in \A^H$ can be equivalently
characterized as the compatibility of the
multiplication in $A$ with that in the braided tensor product $A\bot A$:
\begin{prop} \label{qclemma}
Let $(H, \r)$ be a coquasitriangular Hopf algebra. An $H$-comodule
algebra  $(A, m_A, \delta^A)$ is quasi-commutative if and only if the multiplication
$m_A:A\bot A\to A,$ $ a\bot c\mapsto ac$ is an algebra map. Thus $m_A:
A\bot A\to A$ is an $H$-comodule algebra map if  $A\in \qcr$.
\end{prop}
\begin{proof} 
On the one hand
$$
m_A\big((a\bot c)\tpr (a'\bot c')\big)=m_A\big(a\zero{a'}\bot
\zero{c}c' \r({\one{c}}\ot {\one{a'}})\big)
=a\zero{a'}\zero{c}c' \r({\one{c}}\ot
  {\one{a'}})~;$$
on the other hand
$$\big(m_A(a\bot c)\big)\big(m_A(a'\bot  c')\big)=aca'c'~.$$
Hence the two expressions coincide if and only if $A$ is quasi-commutative.
Moreover, by definition of $H$-comodule algebra, the multiplication map is an $H$-comodule map.
\end{proof}
\begin{rem}\label{remBinZ}
The subalgebra  $A^{coH}\subseteq A$ of a quasi-commutative
$H$-comodule algebra $A$ is contained in the centre $Z(A)$ of $A$. 
This follows from 
\eqref{qcomm1} and the
normalization property
\eqref{Rnormalized} of the $R$-form.
\end{rem} 
\begin{ex}\label{trivexH}
Every commutative  algebra $A \in \A^H$, with commutative Hopf
  algebra $H$ and trivial coquasitriangular structure $R=\varepsilon \ot \varepsilon$, is
quasi-commutative. Indeed quasi-commutativity with  $R=\varepsilon
\ot \varepsilon$ is equivalent to commutativity.
\qed\end{ex}
\begin{ex}\label{ex:usualm_H}
Let $(H,R)$ be a coquasitriangular Hopf algebra, the $H$-comodule algebra
$(H, \cdot, \Delta)$ is quasi-commutative if and only if 
$R=\varepsilon\otimes \varepsilon$ is the trivial $R$-form, and hence $H$ is commutative.
The proof is straighforward, comparing the cotriangularity and
quasi-commutativity conditions \eqref{gira-H} and \eqref{qcomm1}
we obtain, for all $h,k\in H$,
${h}{k}=\ru{\one{h}}{\one{k}}\two{h}\two{k}$. Applying the counit
$\varepsilon$ gives $R=\varepsilon\otimes\varepsilon$, and hence 
commutativity of $H$.
\qed\end{ex}

Many examples of quasi-commutative  algebras arise as
twist deformations (see \S \ref{sec:twists-hopf})
of commutative  algebras $A \in \A^H$. More in general, twist deformations of quasi-commutative algebras are quasi-commutative algebras:
\begin{ex}\label{ex:qc}
Let $A\in \qcr$ and  $\cot: H\otimes H\to \bbK$ a
2-cocycle on $H$. Consider the Hopf algebra
$H_\cot$ with coquasitriangular structure $R_\cot=\cot_{21}\ast R\ast \bar{\cot}$ as in Example \ref{ex:rug}. Let
$A_\cot \in \A^{\hg}$ be the twisted $H_\cot$-comodule algebra of $A$:
this is the $\bbK$-module $A$
with new product 
$
a \mtco a':= \zero{a} \zero{a'} \,\coin{\one{a}}{\one{a'}}$
and unchanged coaction $a\mapsto\zero{a}\ot \one{a}$
(see \S \ref{sec:twists-hopf}).
Then, 
\begin{align*}
a \mtco a' &= \zero{a'}   \zero{a}  \, \ru{\one{a}}{\one{a'}} \coin{\two{a}}{\two{a'}}
\\
&=\zero{a'}  \mtco \zero{a}  \, \co{\one{a'}}{\one{a}} \ru{\two{a}}{\two{a'}}\coin{\three{a}}{\three{a'}}
\\
&= \zero{a'} \mtco \zero{a}~ \r_\cot({\one{a}}\ot {\one{a'}} )
\end{align*}
showing that $A_\cot\in \qcrc$.
\qed\end{ex}
\begin{ex}
Let $H$ be commutative with trivial $R$-form $R=\varepsilon\ot \varepsilon$, so that the 
$H$-comodule algebra $(H,\cdot,\Delta)$ is quasi-commutative
(cf. Example \ref{trivexH}).  
The twist deformation of $(H,\cdot,\Delta)\in
\A_{qc}^{(H,\varepsilon\ot\varepsilon)}$, as in Example \ref{ex:qc} just above, gives the quasi-commutative $H_\gamma$-comodule algebra
$(H,\mtco,\Delta)$ $\in \qcrc$, with $R_\cot=\cot_{21}\ast\bar{\cot}$.
\qed\end{ex}

\begin{ex}\label{mainexHcot}
A main  example of quasi-commutative comodule algebra is the $H$-comodule algebra
$(\bh, \hpr, \Ad)$ associated with a cotriangular Hopf algebra
$(H,R)$. Indeed cotriangularity reads $Q=\varepsilon\otimes\varepsilon$
and then  the braided commutativity property \eqref{H-braid-commQQ} 
reduces to the quasi-commutativity property 
\[
h\hpr k=\two{k}\hpr\two{h}\ru{S(\one{h})\three{h}}{S(\one{k})\three{k}}~. 
\]
\qed\end{ex}
 
Quasi-commutativity of $\underline H$ does not imply cotriangularity of $H$  as this example shows:
\begin{ex} \label{ex:CZ} Let $H=\mathbb{C}\mathbb{Z}=\mathbb{C}[g,g^{-1}]$ be the group
  Hopf-algebra of the group $\mathbb{Z}$ considered in Example \ref{ex:CZzero}, with $R$-form $R_q(g^n,g^m)=q^{-nm}$
 for a complex number $q\neq 1$. It is 
coquasitriangular but not  cotriangular. Since the adjoint coaction is trivial it is immediate
that $(\bh=\mathbb{C}\mathbb{Z},\cdot, \Ad)$ is
  quasi-commutative with respect to $R_q$.  More generally, if $R$ is
  a coquasitriangular structure on a commutative and cocommutative
  algebra $H$, then $(\bh,\hpr, \Ad) =(H,\cdot,\Ad)$ is
  quasi-commutative since the adjoint coaction is trivial.
\qed\end{ex}

Another example of quasi-commutative algebra $A\in \A^H$
with coquasitriangular and {\sl not} cotriangular Hopf algebra $H$ is the following one:
\begin{ex} 
The FRT bialgebra $\O(M_q(2))$ is generated for $j,k=1,2$, by elements $u_{jk}$, satisfying
$\R^{ji}_{kl} u_{km}u_{ln}= u_{ik} u_{jl} \R^{lk}_{mn}$, with the only non zero components of the matrix $\R$
\[
\R^{11}_{11} = \R^{22}_{22} =q \quad , \quad \R^{12}_{12} =\R^{21}_{21} = 1 \quad , \quad  \R^{21}_{12} = q-q^{-1} \; 
\]
for $q\in \IC $, $q \neq 0$.  
Let $H=\O(GL_q(2))$ be the Hopf algebra of coordinate functions of the
quantum group $GL_q(2)$ which is obtained by extending  $\O(M_q(2))$
by a central element $D^{-1}$, defined to be the inverse of the quantum determinant $D:=u_{11}u_{22}-q u_{12}u_{21}$.
The Hopf algebra $H$ is coquasitriangular  with (not cotriangular) universal $R$-form
\beq\label{above}
\ru{u_{ij}}{u_{kl}}= q^{-1} \R^{ik}_{jl} \; , \quad   \ru{D^{-1}}{u_{ij}}=\ru{u_{ij}}{D^{-1}}= q\, \delta_{ij} \; , \quad
\eeq
see e.g. \cite[\S 10.1]{KS}. The convolution inverse is $\ruin{u_{ij}}{u_{kl}}= q (\R^{-1})^{ik}_{jl}$.
Let $A=\O(\IC_q^2)$ be the algebra of the quantum plane, that is, the algebra generated by two elements $x_1,x_2$ with commutation relations $x_1 x_2=q~x_2x_1$. It is well known that $A$ is a $\O(GL_q(2))$-comodule algebra  with coaction $\delta(x_i)= x_j \ot u_{ji}$;
 it is easily verified that 
$A$ is quasi-commutative with respect to the coquasitriangular 
structure $R$  defined in \eqref{above}:
$$x_i x_j = x_l x_p \ru{u_{pi}}{u_{lj}}
= 
q^{-1} \R^{pl}_{ij} x_l x_p$$
for each $i,j=1,2{}$.
Note that the Hopf algebra $\O(GL_q(2))$ admits the one parameter family
  of coquasitriangular structures
$
R_\lambda(u_{ij} \ot u_{kl})= \lambda \R^{ik}_{jl}
$,
with nonvanishing $\lambda \in \IC$. For $\lambda$ a square root of
$q^{-1}$, $R_\lambda$ is also a coquasitriangular structure on the
quotient Hopf algebra $\O(SL_q(2))$. Nevertheless, the comodule algebra
$A=\O(\IC_q^2)$  is not quasi-commutative with respect to it.
\qed\end{ex}

\subsection{Hopf--Galois extensions for coquasitriangular Hopf
  algebras}~\\[.5em]
As mentioned in \S \ref{sec:hge}, for a generic 
noncommutative algebra extension, in contrast with the commutative
case, the canonical map $\can=(m_A \ot \id)\circ (\id \ot \delta)$ is
just a morphism of relative Hopf modules. The domain 
$A \ot_B A$ of $\can$ itself does not inherit an algebra structure from $A \ot
A$ and  the multiplication $m_A: A\ot_B A\to A$ is not an algebra map. 
In this subsection we find when 
the canonical map of an Hopf--Galois extension with coquasitriangular
Hopf algebra is an algebra map, and a morphism in the category $(\A^H,\bot)$.

\begin{lem}\label{lem:B}
Let $(H,\r)$ be a coquasitriangular Hopf algebra, and let $A\in \A^H$ with subalgebra of coinvariants $B=A^{coH}
\subseteq A$. The $\tpr$-product
\eqref{tpr} on $A \ot A$ descends to a well-defined product on $A \ot_B A$ if and
only if $B$ is in the centre of $A$.
\end{lem}
\begin{proof} 
The balanced tensor product $A\ot_B A$ is by definition the quotient
of $A\ot A\in  {_A{\mathcal  M}_A}^{H}$ by the
$A$-sub-bimodule and $H$-subcomodule $J=\{a( b\ot 1-1\ot b) a' \, ,
a,a' \in A, ~b \in B \}$.  We prove the lemma by showing that $J$ is an ideal in $A\bot A$ if and only if the
subalgebra of coinvariants $B$ is central in $A$. 
If $J$ is
an ideal in $A\bot A$ then, for all $a\in A, b\in B$,  $(b\ot 1-1\ot
b)\tpr(a\ot 1)\in J$; since $B=A^{coH}$, and thus  $(b\ot 1-1\ot
b)\tpr(a\ot 1)= - [a,b] \ot 1 + a(b \ot 1 -1 \ot b)$, this implies $[a,b]\otimes 1=0$ and hence $[a,b]=0$.
 Vice versa if $B$ is central in $A$ then for all $a,a' \in A, b \in B$
$$a(b\ot 1-1\ot b)a'=(a\ot a') (b\ot 1-1\ot b)=(a\ot a') \tpr (b\ot 1-1\ot b)$$
where the last equality holds because $B=A^{coH}$. In a similar way, $a(b\ot 1-1\ot b)a'= (b\ot 1-1\ot b) \tpr (a\ot a')$. This proves that
$J$ is the two-sided ideal in $A\bot A$ generated
by $b\ot 1-1\ot b, ~b \in B$. 
\end{proof}

Since Remark \ref{remBinZ} shows that the subalgebra of coinvariants
of a quasi-commutative algebra $A\in \qcr$
is in the centre of $A$, for such an algebra $A$ we have that $A \ot_B A$ inherits an algebra
structure from $A\bot A$;  we denote it  by $A \bot_B A$. 
We
correspondingly denote by $\delta^{A \,\bot_BA}: A \bot_B A\to A
\bot_B A\bot H$ the $H$-coaction $\delta^{A \,\ot_BA}:  A \ot_B A\to A
\ot_B A\ot H$.

\begin{prop}
Let $(H,\r)$ be a coquasitriangular Hopf algebra, and let $A\in \qcr$
with subalgebra of coinvariants $B=A^{coH}$. Then
$(A\bot_B A,\tpr,\delta^{A \,\bot_BA})$ is an $H$-comodule
algebra. 
\end{prop}
\begin{proof}
The triple $(A\bot A,\tpr,\delta^{A \,\bot A})$ is an $H$-comodule
algebra because $(\A^H,\bot)$ is a monoidal category for $H$
coquasitriangular (cf. Proposition \ref{prop:tpr}). The balanced
tensor product $A\bot_B A$  is the quotient of $A\bot A$ via the ideal
and $H$-subcomodule $J$ generated by $b\ot 1-1\ot b, ~b \in B$. 
The $H$-comodule
algebra structure on the quotient $A\bot_B A$ is therefore induced from that of $A\bot A$.
\end{proof}
The results on $H$-comodule algebras and morphisms 
  established so far  are profitably applied to the study of the
  canonical map. Recalling the map $\bdelta^A: A\to A\bot \bh$
  associated with an $H$-comodule algebra $A$ (cf. Proposition
  \ref{prop:deltaund}), we have:
\begin{thm}\label{prop:canam}
Let $(H,\r)$ be a coquasitriangular Hopf algebra and $A \in \qcr$ a
quasi-commutative $H$-comodule algebra. Let $B=A^{coH}$ be the corresponding subalgebra of coinvariants. 
Then the canonical map 
\begin{align}\label{can}  \can = (m \bot \id) \circ (\id \bot_B \bdelta^A ) : A \bot_B A \longrightarrow A \bot \bh~  , 
 \quad  a' \bot_B a
  \longmapsto
a' \, \zero{a} \bot \one{a} \nn \end{align} 
is a composition of  (well-defined) $H$-comodule algebra maps and thus a morphism in $\A^H$.
\end{thm}
\begin{proof} 
The map $(\id \bot \bdelta^A ) : A\bot A\to A\bot A\bot \bh$ is an
$H$-comodule algebra map because tensor product of $H$-comodule
algebra maps (cf. Proposition \ref{prop:deltaund} and Proposition
\ref{prop:tpr0} or \ref{prop:tpr}). The quotient $A\bot_B A$ is well
defined because $B$ is central in $A\in \qcr$ (cf. Remark \ref{remBinZ}).
The induced map on the quotient $(\id \bot_B \bdelta^A ) : A\bot_B A\to A\bot_B A\bot \bh$
is well-defined because of $B$-linearity of $\bdelta^A$.

From Proposition \ref{qclemma} we know that $m_A: A\bot A\to A$ is an
$H$-comodule algebra map when  $A$ is quasi-commutative. It induces a well-defined
$H$-comodule algebra map on the quotient $m: A\bot_B A\to A$. Then $m\bot \id:
A\bot_B A\bot \bh\to A\bot \bh$ is an $H$-comodule algebra map because 
tensor product of comodule algebra maps.
\end{proof}
As a corollary of the above proposition, when $\chi$ is invertible, the translation map
$$\tau =\chi^{-1}{|_{_{1 \bot \underline{H}}}}: \underline{H} \longrightarrow A \bot_B A $$
is an algebra map as well, $\tau (h \hpr k)=\tau (h)\tpr \tau (k)$, and hence an $H$-comodule algebra map.  

Let us record some additional properties of the translation map. Being $\tau={\chi^{-1}}{|_{_{1 \ot {\bh}}}}$, 
one has $ \tuno{h}  \tdue{h}= \varepsilon(h) 1_A$, for any $h \in \bh$. In addition, 
by combining properties \eqref{p1} and \eqref{p4} one also has
\beq\label{p14}
\one{\tuno{h}} \ot \zero{\tuno{h}}  \ot_B \zero{\tdue{h}} \ot \one{\tdue{h}} = S(\one{h})  \ot \tuno{\two{h}} \ot_B \tdue{\two{h}} \ot 
\three{h} \; ,
\eeq
for any $h \in \bh$.

For the particular case of an Hopf--Galois extension $A^{coH} \subseteq A$ with $H$ coquasitriangular  and   $A \in \qcr$, using the property \eqref{p14} and $ \tuno{h}  \tdue{h}= \varepsilon(h) 1_A$, the quasi-commutativity of $A$ leads to  
 \beq\label{tau21H}
 \tdue{h} \tuno{h} =  \ruin{S(\one{h})}{\two{h}} 1_A = \ru{\two{h}}{S(\one{h})}1_A \; .
\eeq
For later use in Proposition \ref{prop:Finv},  we prove the following key additional properties of the
  translation map. The first one  \eqref{tauS} just concerns $\tau$  as a linear map, the second one \eqref{tauS2} uses that $\tau$ is an algebra map.
\begin{lem}
Let $(A,\delta^A) \in \qcr$ be a quasi-commutative $H$-comodule algebra, with the extension $A^{coH} \subseteq A$ a Hopf-Galois one.
The translation map satisfies the identity
\beq\label{tauS}
 \tau \circ \bs = \Psi^R_{A,A} \circ \tau ~.
\eeq
Moreover,
\beq\label{tauS2}
\tpr \circ ((\Psi^R_{A,A} \circ \tau) \bot  \tau) \circ \Delta_{\underline{H}}= \eta_{A \bot_B A} \circ \varepsilon_{\underline{H}}
\eeq
that is, for each $h \in \underline{H}$
\beq\label{tauSexp2}
\zero{\tdue{\one{h}}}
\zero{\tuno{ \two{h} } } \bot_B 
\zero{\tuno{\one{h}}} \tdue{\two{h} }
\ru{\one{\tuno{\one{h}}}}{\one{\tdue{\one{h}}}\one{\tuno{ \two{h}}}} 
= \varepsilon(h) \, 1_A \bot_B 1_A \, .
\eeq
\end{lem}

\begin{proof} 
Identity \eqref{tauS} holds, indeed for each $h \in \underline{H}$ one has
\begin{align*} 
\can(\Psi^R_{A,A} \circ \tau(h)) &= \can \big( 
\zero{\tdue{h}}\bot_B \zero{\tuno{h}}\big) ~\ru{\one{\tuno{h}}}{\one{\tdue{h}}}
\\
&= \zero{\tdue{h}}\zero{\tuno{h}} \bot \one{\tuno{h}}~\ru{\two{\tuno{h}}}{\one{\tdue{h}}}
\\
&= \tdue{\three{h}}\tuno{\three{h}} \bot S(\two{h}) ~\ru{S(\one{h})}{\four{h}}
\\
&= 1_A \bot S(\two{h}) ~\ruin{S(\three{h})}{\four{h}}~\ru{S(\one{h})}{\five{h}}
\\
&= 1_A \bot S(\two{h}) ~\ru{S^2(\three{h})}{\four{h}}~\ru{S(\one{h})}{\five{h}}
\\
&=
1_A \bot S(\two{h}) ~\ru{S^2(\three{h})S(\one{h})}{\four{h}}
\\
&=
1_A \bot \bs(h)
\end{align*}
where for the third equality we used  \eqref{p14}
and for the fourth one property \eqref{tau21H} of the translation map.
Then, the identity \eqref{tauS} directly implies the
  second part of the lemma; indeed
\begin{align*}
\tpr \circ ((\Psi^R_{A,A} \circ \tau) \bot  \tau) \circ \Delta_{\underline{H}} 
\,=\,
\tpr \circ (( \tau \circ \bs) \bot  \tau) \circ \Delta_{\underline{H}} 
=
\tpr \circ ( \tau  \bot  \tau) \circ ( \bs \bot  \id) \circ \Delta_{\underline{H}} \\[.3em]
\,=\,\tau \circ \hpr \circ ( \bs \bot  \id) \circ \Delta_{\underline{H}} 
=
\tau \circ \eta_{\underline{H}}\circ \varepsilon_{\underline{H}}
= \eta_{A \bot_B A} \circ \varepsilon_{\underline{H}} \, ,
\end{align*}
using in the second line that $\tau$ is an algebra map and
the antipode property \eqref{bantipode}.
\end{proof}
\begin{rem}\label{rem:class} 
If $H$ is commutative with trivial $\r$-form, $\r = \varepsilon
\ot \varepsilon$, then $A\in \qcr$ is commutative and property \eqref{tauS} simply reads 
$\tau (S(h))= \mbox{flip} \circ \tau(h)\in A\ot_B A$. 
In particular, let $B=A^{co H}\subseteq A$ be the
Hopf--Galois extension of the principal $G$-bundle $\pi: P \ra P/G$ considered in Example \ref{affinecase}.
Then the property $\tau (S(h))= \mbox{flip} \circ \tau(h)$ of the
pull-back $\tau=t^*$ of the classical translation map corresponds by duality to the property
$t(q,p)=t(p,q)^{-1}$, $p,q \in P$ .
Similarly, property \eqref{tauS2} corresponds by duality to
  $t(q,p)t(p,q)=e$, the neutral element of $G$.
\end{rem}

When $(H,\r)$ is a cotriangular bialgebra,
the category $\qcr$ of quasi-commutative $H$-comodule algebras 
with the braided tensor product $\bot$ becomes a braided monoidal category.
Moreover, the canonical map is a morphism in $\qcr$.

\begin{prop}\label{prop:tprqc}
Let $(H,\r)$ be a cotriangular bialgebra.
The braided tensor product of quasi-commutative H-comodule algebras is a
quasi-commutative $H$-comodule algebra. 
\end{prop}
\begin{proof}
Let $A,C \in \A^H$ be quasi-commutative, then for all $a, a'\in A$ and $ c, c'\in C$,
\begin{align*}
(a \ot c) \tpr (a' \ot c') & = a\zero{a'} \ot \zero{c} c' ~\ru{\one{c}}{\one{a'}} \\
& = \zero{a'}\zero{a}\ot\zero{c'}\zero{c}\,
\ru{\one{a}}{\one{a'}}
\ru{\one{c}}{\one{c'}}
\ru{\two{c}}{\two{a'}}
\end{align*}
where we used the definition of the $\tpr$-product in \eqref{tpr} and the quasi commutativity of $A$ and $C$. On the other hand, 
\begin{align*}
& \zero{({a'}\ot {c'})}\tpr\zero{(a\ot c)}\ru{\one{(a \ot c})}{\one{(a' \ot c')}} = \\ 
& \;= (\zero{a'}\ot \zero{c'})\tpr(\zero{{a}}\ot \zero{{c}})\ru{\one{a}\one{c}}{\one{a'}\one{c'}}\\ 
& \;=
\zero{a'}\zero{{a}}\ot \zero{c'}\zero{{c}}\,
\ru{\one{c'}}{\one{a}} \ru{\three{a}}{\one{a'}}
\ru{\two{a}}{\two{c'}}\ru{\one{c}}{\three{c'}}
\ru{\two{c}}{\two{a'}}.
\end{align*}
This coincides with the previous expression since, using the cotriangularity of $H$, one can simplify
$ \ru{\one{c'}}{\one{a}}\ru{\two{a}}{\two{c'}}=
\ruin{\one{a}}{\one{c'}}\ru{\two{a}}{\two{c'}}
=\varepsilon({\one{a}}) \varepsilon({\one{c'}}) $.
\end{proof}

As a direct consequence of this proposition 
we have: 
\begin{cor}\label{catqc}
Let $(H,\r)$ be a cotriangular Hopf algebra. The category $\qcr $ endowed with the braided tensor product $\bot$ is a
 full sub-monoidal category of $(\A^H,\bot)$.
\end{cor}

From Theorem \ref{prop:canam} we then have: 
\begin{cor}
Let $(H,\r)$ be a cotriangular Hopf algebra, $A\in \qcr$ and $B=A^{coH}\subseteq A$ a Hopf--Galois extension. 
Then $\bh\in \qcr$ and the corresponding canonical map $\chi : A \bot_B A \longrightarrow  A \bot \bh$ 
is an isomorphism in the category ($\qcr, \bot)$.
\end{cor}

\section{The gauge group}\label{GG}
In the classical (commutative) case one way to define the group
$\G_P$ of gauge transformations of a principal $G$-bundle $\pi:P \to P/G$ is as the group of $G$-equivariant maps, 
\begin{equation}\label{GPclass}
\G_P:=\{\sigma:P\to G ; \,\, \sigma (pg)=g^{-1}\sigma(p)g\}~,
\end{equation}
where $G$ is a right $G$-space with respect to the right $G$-adjoint action.
The group structure is by point-wise product:
$(\sigma\tilde\sigma)(p)=\sigma(p)\tilde\sigma(p)$, for all 
$\sigma,\tilde\sigma\in \G_P$ and $p\in P$. 
The gauge group can be equivalently defined as the subgroup of principal bundle
automorphisms  which are vertical (project to the identity on the base space):
\begin{equation}\label{AVclass}
\mathrm{Aut}_{P/G}(P) :=\{\varphi:P\to P ; \,\, \varphi (pg) = \varphi(p)g ~,\pi( \varphi(p)) = \pi(p) \},
\end{equation}
with group law given by map composition.  The equivalence of these two definitions is well known \cite[\S 7.1]{hus}. 

These definitions can be dualised for algebras rather than spaces. For instance, in the context of the affine varieties case treated in Example \ref{affinecase}, where
$A=\O(P)$, $B=\O(P/G)$, $H=\O(G)$, the gauge group $\G_P$ in
\eqref{GPclass} of
$G$-equivariant maps
corresponds to that of $H$-equivariant maps (or $H$-comodule maps)
that are also algebra maps
\begin{equation}\label{GAcomm}\G_A:=\{ \f: H\to A ;\,\, \delta^A(\f)=(\f\otimes \id)\circ \Ad\,,~ \f
{\mbox{ algebra map}}\}~.
\end{equation} 
The group structure is the convolution product.
The algebra map property for the pull-back
$\f=\sigma^\ast: H\to A$ comes 
from the point-wise product in $H=\O(G)$ and
$A=\O(P)$: $\f(hk)(p)=(hk)(\sigma(p))=
h(\sigma(p)k(\sigma(p))=(\f(h)\f(k))(p)$, for all $h,k\in H, p\in P$. 

Similarly, the vertical
automorphisms description \eqref{AVclass} of the gauge group corresponds to that of $H$-equivariant maps 
\begin{equation}\label{AVcomm}\Aut{A}=\{\F: A\to A ; \,\, \delta^A \F= (\F \otimes \id) \delta^A\,,
~\F|_B=\id: B\to B\,,~\F
\mbox{ algebra map}\}~.
\end{equation}

The dual definitions can be given for a general Hopf--Galois extension
$B=A^{co H}\subseteq A$, with $A$ and $H$
commutative algebras. However, for a noncommutative Hopf--Galois extension the algebra map condition
in these definitions is in general very restrictive. This does not
come as a surprise: for noncommutative
algebras already algebra automorphisms are very constrained with respect to the commutative
case.

In \cite[\S 5]{brz-tr} this issue was faced by weakening the notion
of gauge group: gauge transformations
are no longer algebra maps; 
they are defined to be
comodule maps that are invertible and unital.
In this ``no algebra maps" context the isomorphism
$\G_A\simeq \Aut{A}$ still holds.
A drawback of this approach, besides the extra requirement of 
 invertibility of the maps, is that the resulting gauge groups are
very big, even in the classical case. For example the gauge group of the
$G$-bundle on a point $G\to \{*\}$ is much bigger than the structure
group $G$ as the following simple example shows.
\begin{ex}
Consider the group $\IZ_2:=\{e, u\}$ of integers modulo $2$: $e+e=
u+u=e, e+u=u+e=u$.  Let $H$ be its coordinate Hopf algebra; this is
the commutative complex algebra generated by the two orthogonal
projections $p_e$ and $p_u$ (where $p_a(b)=\delta_{a,b\,}$, for $a,b\in
\IZ_2$) with unit $1_H = p_e + p_u$ the constant function $1$. 
It has cocommutative coproduct   
$$
\Delta p_e = p_e \ot p_e + p_u \ot p_u \, , \qquad
\Delta p_u = p_e \ot p_u + p_u \ot p_e \, 
$$
and counit $\varepsilon(p_e)=1$, $\varepsilon(p_u)=0$.
The trivial $\IZ_2$-bundle over a point is dually described
as the Hopf--Galois extension $B\subseteq A$, where $A=H$ 
with coaction $\Delta$ and 
resulting algebra of coinvariants $B=\IC$. Since $A=H$ is the
linear span of   $p_e$ and $p_u$ and the condition $F|_B=\id: B\to B$
is just that of $\IC$-linearity of the map $\F:A\to A$, we have
linear maps $\F (xp_e + y p_u) = x'p_e + y' p_u$ from $\IC^2$
to $\IC^2$, that is, complex $2 \times 2$ matrices
$$
 \begin{pmatrix}
x \\ y 
\end{pmatrix}
\mapsto 
\begin{pmatrix}
x' \\ y' 
\end{pmatrix}
=
\begin{pmatrix}
a & b  \\  c & d 
\end{pmatrix}
 \begin{pmatrix}
x \\ y 
\end{pmatrix} \, .
$$
Unitality of $\F$ requires $b=1-a$ and $c=1-d$;  invertibility of $\F$
requires $a+d\neq  1$. 
Finally, $H$ equivariance, that is $\Delta \F= (\F \otimes \id)
\Delta$, leads to $a=d$. Summing up, the group of these maps is the $GL(2,\IC)$-subgroup  
$$
\left\{\begin{pmatrix}
a & 1-a  \\  1-a & a 
\end{pmatrix}  \, , \quad \mbox{with} \quad 2a \neq  1 
\right\} \, .
$$
If one imposes the additional condition that the maps $\F$ are algebra maps this group collapses to a much smaller one. 
Indeed, the requirement
$$
\F \big( (x p_e + y p_u) (x' p_e + y' p_u) \big) = \F(x p_e + y p_u) \F(x' p_e + y' p_u) 
$$
for all $(x,y)$ and $(x',y')$ in $\IC^2$, forces $a=1$ or $a=0$. Thus the resulting group is
$$
\left\{\begin{pmatrix}
1 & 0  \\  0 & 1 
\end{pmatrix}  \, , \, 
\begin{pmatrix}
0 & 1 \\  1 & 0 
\end{pmatrix} 
\right\} \simeq \IZ_2 \, ,
$$
that is the expected group of gauge transformations.
\qed\end{ex}

We shall work in the noncommutative setting of the monoidal
category $(\A^H,\bot)$. In this context we show that it is natural
to define the gauge group of vertical automorphisms as in
\eqref{AVcomm}, that is, to  require vertical
automorphisms $\F$ to be algebra maps. Similarly, $\G_A$ is defined as
the group of $H$-equivariant algebra maps $\f: \underline H\to A$. The issue
of the lack of algebra maps is therefore in this case overcome by
properly choosing the algebra structure on $H$, namely the
multiplication $\hpr$ of the braided
Hopf algebra $\underline H$ rather than that of the Hopf algebra $H$. 

We begin by studying this latter space $\G_A$ of
$H$-equivariant algebra maps. 
We then consider
the gauge group $\Aut{A}$ of vertical automorphisms and prove its equivalence
with $\G_A$.  We present a few examples; while they are mainly commutative ones, they serve as a way of illustration of the notions involved. They will be deformed to noncommutative examples later on in the paper.

\subsection{The gauge group of equivariant algebra maps}~\\[-1.3em]
\begin{prop}\label{prop:GA}
Let $(H,R)$ be a coquasitriangular Hopf algebra, $(\bh, \hpr, \eta, \Delta, \varepsilon, \bs, \Ad)$ the associated
braided Hopf algebra, $A\in \qcr$ and $B=A^{coH}\subseteq A$ a Hopf--Galois extension.
The $\bbK$-module  
\beq 
\G_A:=\Hom_{\A ^H}(\bh, A) 
\eeq
of $H$-equivariant algebra maps $\underline {H}\to A$ is a group with
respect to the convolution product. The inverse of  $\f \in \G_A$ is given by
$\bar{\f}:=\f\circ \bs$.
\end{prop}

\begin{proof}
Given $\f, \g\in \Hom_{\A ^H}(\bh, A)$, the product $\f*\g$ is an
$H$-comodule map; for all $h\in \underline {H}$,
\begin{align*}
\delta^A(\f * \g)(h) &
= \delta^A \big(\f(\one{h})\g(\two{h}) \big)
= \delta^A \big(\f(\one{h}) \big)\delta^A \big(\g(\two{h}) \big)\\
& 
= \big(\f(\two{h}) \ot S(\one{h})\three{h} \big)\big(\g(\five{h}) \ot S(\four{h})\six{h} \big)
= \f(\two{h}) \g(\three{h})\ot S(\one{h})\four{h} 
\end{align*}
where we used that $\delta^A$ is an algebra morphism and that both $\f$ and $\g$ are $H$-comodule morphisms.
Then
\begin{align*}
\delta^A(\f * \g)(h) 
= (\f * \g)(\two{h}) \ot S(\one{h})\three{h} 
= \big((\f * \g) \ot \id_H \big) \Ad (h) \, . 
\end{align*}
The product $\f * \g$ is also an algebra map. Recall from
\eqref{cop-braid} that $\Delta_\bh$ is an algebra map with respect to
the products $\hpr$ in $\bh$ and  $\tpr$ in the braided tensor product algebra $\bh \bot \bh$, that is
$
\Delta_\bh (h \hpr k)  = \one{h} \hpr \two{k} \bot \three{h} \hpr \four{k} 
\ru{S(\two{h})\four{h}}{S(\one{k})\three{k}}\, 
$
for all $h,k \in \bh$. Then we compute 
\begin{align*}
(\f * \g) (h \hpr k)
&
= \f(\one{h} \hpr \two{k} ) \g( \three{h} \hpr \four{k} ) \ru{S(\two{h})\four{h}}{S(\one{k})\three{k}}
\\
&
= \f(\one{h}) \f(\two{k} ) \g( \three{h}) \g( \four{k} ) \ru{S(\two{h})\four{h}}{S(\one{k})\three{k}}
\\
&
= \f(\one{h}) \zero{(\f(\one{k} ))}
 \zero{(\g( \two{h}))} \g( \two{k} ) \ru{\one{(\g( \two{h}))} }{\one{(\f(\one{k} ))}}
\\
&
= \f(\one{h}) \g( \two{h}) \f(\one{k} )\g( \two{k} ) 
\\
&
=(\f *\g)(h) (\f *\g)(k) \; ,
\end{align*}
where the second equality uses that $\f,\g  $ are algebra maps, and the third equality that they are $H$-comodule maps. The fourth one follows from the quasi-commutativity of $A$.  

Any $\f:\underline H\to A$ is convolution invertible, with inverse $\bar{\f}:=\f\circ \bs$; 
indeed  (recalling the antipode defining property \eqref{bantipode})  
$$
(\f * \bar\f) (h)= \f (\one{h}) \f (\bs(\two{h})) = \f(\one{h} \hpr \bs(\two{h})) = \varepsilon(h)1_A~,
$$
and similarly  $\bar\f *\f = 1_A\, \varepsilon $. 
The map  $\bar{\f}$ is an $H$-comodule map being composition of
$H$-comodule maps.
In order to prove that $\bar{\f}\in \G_\A=\Hom_{\A
  ^H}(\bh, A) $ we are left to show that $\bar{\f}$ is an algebra map.
This is immediate if the Hopf algebra $H$ is cotriangular, because in this
case the braided antipode $\bs$ is an algebra map. In the coquasitriangular case,
with $A$ quasi-commutative, few more passages are needed. 
We evaluate the algebra map $\f$ on
\begin{equation}\label{Sbam}
\bs(h\hpr k)=\hpr\circ \Psi^R_{\bh,\bh}(\bs{(h)}\otimes \bs(k)) 
= \zero{\bs{(k)}} \hpr \zero{\bs{(h)}} \ru{\one{\bs{(h)}}}{\one{\bs{(k)}}}~,
\end{equation}
$h,k\in \bh$, this being the braided anti-algebra map property
\eqref{SmmRSS} of the braided antipode.  We compute
\begin{align*}
\f (\bs(h\hpr k)) 
&= \f \left(\zero{\bs{(k)}} \hpr \zero{\bs{(h)}} \right)\ru{\one{\bs{(h)}}}{\one{\bs{(k)}}} 
\\
&= \f (\zero{\bs{(k)}}) \f (\zero{\bs{(h)}})\ru{\one{\bs{(h)}}}{\one{\bs{(k)}}} 
\\
&= \zero{\f (\bs{(k)})} \zero{\f (\bs{(h)})} \ru{\one{\f (\bs{(h)})}}{\one{\f (\bs{(k)})}} 
\\
& = \f (\bs(h)) \f (\bs(k))
\end{align*}
where for the last but one equality we  used that $\f$ is a
morphism of comodules and for the last equality we  used
quasi-commutativity of the algebra $A$, as defined in \eqref{qcomm2}.
Therefore $\bar\f (h\hpr k) =\bar\f (h)\bar\f (k)$ as claimed.
\end{proof}

In the commutative case, for a principal $G$-bundle $\pi: P \ra M$ which is trivial, 
the gauge group is isomorphic to the group of maps from $M$ to
$G$ (see e.g. \cite[\S 7.1, Prop.1.7]{hus}). 
For Hopf--Galois extensions we have:
\begin{lem}\label{lem:ban} 
Let $B \subseteq A$ be a trivial Hopf--Galois extension, with $(H,R)$
coquasitriangular  and $A\in \qcr$.
Then $R$ is trivial, $H$ and $A$ are commutative, and 
the gauge group $\G_A$ coincides with the group (with convolution product $*$) of algebra maps from $H$ to $B$:
$$
\G_A\simeq  \left( \{ \alpha: H \ra B \mbox{ algebra maps} \} , *\right) .
$$
\end{lem}
\begin{proof}
For a trivial extension with $B$ in the centre of $A$, the cleaving map
gives the isomorphism $A\simeq B \ot H$ in $\qcr$, with coaction
$\delta = \id \ot \Delta$. 
This implies that $H$ is quasi-commutative, and therefore, cf. Example
\ref{ex:usualm_H},  $H$ is commutative with trivial $R$-form, so that $A$ is commutative as well.
As for the gauge group 
$\G_A=\Hom_{\A ^H}(\bh,  B \ot H)$, observe first that  the  braided Hopf algebra $\bh$ is isomorphic as a Hopf algebra to $H$. Indeed, since $R$ is trivial,   the product in $\bh$
equals that in $H$  and the braiding $\Psi^R_{\bh,\bh}$ is
trivial. Next, 
each $\f: H \ra B \ot H$ in $\G_A$
determines an algebra  map $\alpha_\f:= (\id \ot \varepsilon)\circ \f:
H \ra B$. Conversely, with any algebra map  $\alpha: H \ra B$,
one has a map $\f_\alpha:= (\alpha \ot \id)\circ \Ad :H \ra B \ot H$ (cf. \cite[Thm.5.4]{brz-tr}). 
It is easy to verify that $\f_\alpha$ is a morphism of $H$-comodules:
$$
(\f_\alpha \ot \id) \Ad(h)= \f_\alpha(\two{h}) \ot S(\one{h})\three{h}= 
\alpha(\three{h}) \ot S(\two{h})\four{h} \ot S(\one{h})\five{h}=
(\id \ot \Delta)\f_\alpha(h).
$$
It is also an algebra map being a composition of such maps:
$$
\f_\alpha : H \xrightarrow{~\Ad~ } H \ot H  \xrightarrow{~ \alpha \ot \id ~ } B  \ot H \, .
$$
One easily sees that $\alpha_{\f_\alpha}=\alpha$ and  
$\f_{\alpha_\f}= \f$, being $\f$ a comodule map, so that $(\f \ot \id) \Ad= (\id \ot \Delta)\f$.
\end{proof} 

We note that while the algebras in a trivial Hopf--Galois extension  $B\subseteq A\in \qcr$ are
  commutative, for cleft Hopf--Galois extensions this need not
  be the case, and their gauge group is in general not given by $\left( \{
   \alpha: \bh \ra B \mbox{ algebra maps} \} , *\right)$.  See Remark \ref{rem:sotwisted} later on.\\

Also for a principal $G$-bundle $\pi: P \ra M$ with $G$ abelian
the gauge group is isomorphic to the group of maps from $M$ to $G$. 
For Hopf--Galois extensions we have a similar result if
the Hopf-algebra is cocommutative. A coquasitriangular Hopf
algebra $(H,R)$ which is cocommutative is also commutative (cf. \eqref{gira-H} and comments after Example
\ref{trivialex}). Nevertheless, since the $R$-form
can be nontrivial, the algebra $A$ in the inclusion  
$B  \subseteq A$ can  be noncommutative. However the
gauge group does not  depend on $A$:
\begin{lem}\label{lem:Gab} 
Let $(H,R)$ be a coquasitriangular and cocommutative Hopf algebra, and
let $B=A^{coH} \subseteq A\in \qcr$ be a Hopf--Galois extension.  Then its gauge group $\G_A$ coincides with the group of algebra maps from $H$ to $B$ with the convolution product $*$:
$$
\G_A  \simeq  \left( \{ \f: H \ra B \mbox{ algebra maps} \} , *\right) .
$$
\end{lem}
\begin{proof}
Since $H$ is cocommutative, the adjoint coaction $\Ad$ is trivial, so the product in $\bh$
equals that in $H$ (which is commutative due to 
coquasitriangularity) and the braiding $\Psi^R_{\bh,\bh}$ is
trivial. Thus, the associated braided Hopf algebra $\bh$ is isomorphic to $H$ as a Hopf algebra.
Triviality of the adjoint coaction implies that each
$H$-equivariant $\bbK$-linear map $\f : \bh\to A$ satisfies
$\delta \f (h) = \f(h) \ot 1$, that is the image of $\f$ is contained
in the subalgebra $B$ of coinvariants. In particular
$H$-equivariant algebra maps $\f \in \G_A=\Hom_{\A ^H}(\bh, A)$
are algebra maps $\f : \bh\to B$, then algebra maps $\f : H\to B$. 
\end{proof}

\begin{ex} \textit{The Hopf bundle.} 
Consider the $\O(U(1))$-Hopf-Galois extension $\O(S^2) \subset \O(SU(2))$.
By Lemma \ref{lem:Gab} its gauge group is given by
$$
\G_A= (\{ \f:  \O(U(1)) \ra \O(S^2) \mbox{ algebra maps} \}, *)~.
$$
Since  $\O (U(1))$ is
   linearly spanned by  group-like elements, the convolution product
equals the point-wise product and we obtain 
$
\G_A \simeq (\{ \f:  S^2 \ra U(1)  \}, \cdot)$, as expected.
\qed\end{ex}

\subsection{The gauge group of vertical automorphisms}~\\[.5em]
For any Hopf--Galois extension $B\subseteq A\in \qcr$ we show that 
$\Aut{A}$, defined as in the commutative case in \eqref{AVcomm},  is a group.
This uses properties of the (dual of the classical) translation map $\tau=t^*$ (cf. Example \ref{affinecase}) leading to the following: 
\begin{prop}\label{prop:Finv} Let $B=A^{coH}\subseteq A$ be an $H$-Hopf--Galois extension with $(H,R)$ coquasitriangular and $A\in \qcr$.
The $\bbK$-module 
$$ 
\Aut{A}:=\Hom_{_B\A^H}(A, A) =\{ \F \in \Hom_{\A^H}(A, A) \, |  \,\, \F_{|_B}=\id\}
$$
of left $B$-module, right $H$-comodule algebra morphisms is a group
with respect to the composition of maps 
$$
\F \cdot \mathsf{G}:=  \mathsf{G} \circ \F 
$$
for all  $\F , \mathsf{G} \in \Aut{A}$. For $\F\in \Aut{A}$ its 
inverse $\F^{-1} \in
\Aut{A}$  is given by
\begin{eqnarray}\label{inv-F}
\F^{-1}:=m\circ (\id\bot m )\circ (\id\bot F\bot_B\id)\circ (\id\bot
\tau)\circ \bdelta^A\: :  A&\!\longrightarrow \!& A\\ a &\!\longmapsto\!& \zero{a} \F (\tuno{\one{a}}) \tdue{\one{a}} 
~,~~~\nn
\end{eqnarray}
where $\tau={\chi^{-1}}{|_{_{1 \bot \underline{H}}}}$ is the translation map.
\end{prop}
\begin{proof} 
The reversed composition  order $\F \cdot \mathsf{G}= \mathsf{G} \circ \F$ stems from 
the contravariant property of the pull-back
$\varphi\mapsto\F=\varphi^*$
used in the commutative case $A=\O(P)$. The expression for the inverse
map $F^{-1}={\varphi^*}^{-1}$ is induced from that of $\varphi^{-1}$.
The map $\F^{-1}$ is well-defined because $F\bot_B \id$ is well-defined due to the $B$-linearity of $\F$.
We show  $\F^{-1} \in \Aut{A}$. Clearly $\F^{-1}_{|_B}=\id$ since $\F$ and $\tau$ are unital; 
$\F^{-1}$ is an $H$-comodule algebra map because composition of
$H$-comodule algebra maps (for the product $m_A:A\bot A\to A$ see
Proposition \ref{qclemma},  for $m: A\bot_B A\to A$ recall the proof
concerning the canonical map $\chi$ in Theorem {\ref{prop:canam}}, for
$\bdelta^A: A\to A\bot \bh$ see Lemma \ref{prop:deltaund}).

We recall the identity $\zero{a}\tau(\one{a}) = \can^{-1} \circ \can(1 \ot_B a)=1
\ot_B a$, for all $a\in A$. 
To show that $\F^{-1} \circ \F= \id$ we evaluate $\F^{-1}$, as from
definition \eqref{inv-F}, on $F(a)$ and use that $F$ is
$H$-equivariant and that it is an algebra map:
$$
\F^{-1}(\F(a))=\F(\zero{a}) \F (\tuno{\one{a}}) \tdue{\one{a}}=
\F(\zero{a}  \tuno{\one{a}}) \tdue{\one{a}}=\F(1) a=a~.
$$
To show that $\F \circ \F^{-1}= \id$ requires property  \eqref{tauS2} of the
translation map.
Firstly,  by applying $m\circ (\id \ot_B \F)$ to the identity $\zero{a}
\tau(\one{a}) =1 \ot_B a$ we obtain
$$
\F(a)= \zero{a} \tuno{\one{a}} \F(\tdue{\one{a}}) ~.
$$
Then we replace $a$ with
$\F^{-1}(a)$, use that $\F^{-1}$ is a comodule map and obtain
$$
(\F \circ \F^{-1})(a) = 
\F (\F^{-1}(a))
= \zero{a} \F (\tuno{\one{a}}) \tdue{\one{a}}
\tuno{ \two{a} } 
\F \big(\tdue{\two{a} }\big)~.
$$
Next, the quasi-commutativity of  $A$ and the $H$-comodule algebra map property of  $\F$ give
\begin{align*}
\F (\F^{-1}(a))
&= \zero{a}  \zero{\tdue{\one{a}}}
\zero{\tuno{ \two{a} } }\F (\zero{\tuno{\one{a}}})
\F \big(\tdue{\two{a} }\big)
\ru{\one{\tuno{\one{a}}}}{\one{\tdue{\one{a}}}\one{\tuno{ \two{a}}}} 
\\
&= \zero{a}  \zero{\tdue{\one{a}}}
\zero{\tuno{ \two{a} } }\F \big(\zero{\tuno{\one{a}}}\tdue{\two{a} }\big)
\ru{\one{\tuno{\one{a}}}}{\one{\tdue{\one{a}}}\one{\tuno{ \two{a}}}} ~.
\end{align*}
Finally, property \eqref{tauSexp2} of the translation map  implies
\begin{align*} 
\F (\F^{-1}(a))
= \zero{a} \varepsilon(\one{a}) 
=a \, .
\end{align*}
This ends the proof. \end{proof}

\begin{rem}\label{homotopy}
By definition, two Hopf--Galois extensions $A,A' \in \A^H$ of a fixed
algebra $B$ are isomorphic provided there exists an isomorphism of
$H$-comodule algebras $A \ra A'$. This is the algebraic
  counterpart for noncommutative principal bundles of the geometric
  notion of isomorphism of principal $G$-bundles with fixed base
  space.
As in the geometric case this notion is relevant in the homotopy
classification of noncommutative principal bundles,
see e.g. \cite[\S 7.2]{kassel-review}. 
In the coquasitriangular and quasi-commutative context of the present
paper, if $A,A' \in \qcr$ are isomorphic  
via $\omega: A \ra A'$, then
the groups $\Aut{A}$ and $\Aut{A'}$ are isomorphic via
$$
\Aut{A} \ra \Aut{A'} \; , \quad \F \mapsto \omega \circ \F \circ \omega^{-1} \, .
$$
Indeed, even if in general $\omega$ is not the identity on $B$, $\omega (B) \subseteq B$ being $\omega$ a morphism of $H$-comodules. Thus $\omega \circ \F \circ \omega^{-1}|_B= \id$ and
$\omega \circ \F \circ \omega^{-1} \in \Aut{A'}$ as claimed.
Therefore the gauge group $\G_A$ of an $H$-Hopf--Galois extension depends
only on the isomorphism class of the extension, rather than on the single representative.
\end{rem}

\begin{ex} \textit{Galois field extensions.}  
Let $\mathbb{E}$ be a field, $\bbK \subseteq \mathbb{E}$ and $G=\{g_i\}$
a finite group  acting on $\mathbb{E}$ as 
automorphisms of $\mathbb{E}$. Let 
$\mathbb{F}\supseteq \bbK$ be the fixed field of the $G$ action. By
Artin's theorem if the $G$-action is faithful, $\mathbb{E}$ is a Galois
extension of $\mathbb{F}$ and $G$ is its Galois group (the group of authomorphisms of $\mathbb{E}$ that leave $\mathbb{F}$ invariant). The $G$ action $a\mapsto g_i(a)$, $a\in\mathbb{E}$, induces a coaction
of the  dual $(\bbK G)^*$ of the group algebra $\bbK G$, 
$\delta: \mathbb{E} \ra \mathbb{E} \ot (\bbK G)^*$, $a \mapsto \sum_i
g_i (a) \ot \beta_i$, where
$\{ \beta_i\}$ is the basis of  $(\bbK G)^*$ dual to
the basis $\{ g_i\}$ of $\bbK G$. In \cite[\S 8.1.2]{mont} it is proven 
that  $\mathbb{E}$ is a  Galois field extension of $\mathbb{F}$  with Galois group $G$ if and only if the $\bbK$-algebra $\mathbb{E}$
 is a Hopf--Galois extension of $\mathbb{F}= \mathbb{E}^{co(\bbK
   G)^{\!*}}$. In this case consider the trivial coquasitriangular structure on  $(\bbK G)^*$.
The gauge group $\mathrm{Aut}_{\mathbb{F}}(\mathbb{E})$ consists of
maps $\F \in G$ which are morphisms of  $(\bbK G)^*$-comodules, $\delta \F = (\F \ot \id) \delta$. This is equivalent to requiring  
$\F g_i=g_i \F$ for each $i$. Thus 
  $\mathrm{Aut}_{\mathbb{F}}(\mathbb{E})= Z(G)$,
 the center of the Galois group.
\qed\end{ex}

\begin{ex}\label{ex:graded}
\textit{Graded algebras.} Let $G$ be a  group, with neutral element $e$, and let $H=\bbK G$ be its group algebra. 
An algebra $A$ 
is $G$-graded, that is $A=\oplus_{g \in G} A_g$ and $A_g A_{h} \subseteq A_{gh}$ for all $g,h \in G$, if and only if $A$ is a right $\bbK G$-comodule  algebra with coaction 
$\delta: A \ra A \ot \bbK G$, $a \mapsto \sum a_g \ot g $ for $a= \sum a_g$, $a_g \in A_g$.
Moreover,
the algebra $A$ is strongly $G$-graded, that is $A_g A_h =A_{gh}$,  
if and only if $A_e=A^{co(\bbK G)} \subseteq A$ is Hopf--Galois (see e.g. \cite[Thm.8.1.7]{mont}).
One can easily see that 
$$\Hom_{_{A_e} \A^{\bbK G}}(A, A) =\{ \F:A \ra A  \mbox{ algebra maps} ~ |~ \F|_{A_e} = \id \, ,\; \F(A_g) \subseteq A_g ~\} ~.$$
Let now $H=\bbK G$ be coquasitriangular and $A$ be quasi-commutative.
Then Proposition \ref{prop:Finv} shows that $\Hom_{_{A_e} \A^{\bbK G}}(A, A)$ is a group,
the gauge group $\mathrm{Aut}_{A_e}(A)$ of the Hopf--Galois extension
$A_e \subseteq A$. Notice that $H=\bbK G$ coquasitriangular implies
$H$ commutative and hence 
$G$  abelian (cf. remark after Example \ref{trivialex}).
For $G=\IZ$, with $H=\IC \IZ= \O(U(1))$, the Hopf--Galois extension $A_e\subseteq A$ is a noncommutative principal $U(1)$-bundle. 
Examples with  $G=\IZ^n$, $H=\IC \IZ^n= \O(\mathbb{T}^n)$, and $A_e=\IC$ include Example \ref{ex:toroc}
and Example \ref{ex:torotwisted} (noncommutative principal $U(1)^n$ bundles).
\qed\end{ex}
\begin{ex}\textit{Torus bundle over a point.}\label{ex:toroc} 
(Its noncommutative deformation is in Example \ref{ex:torotwisted}.)
Let
$\O (\mathbb{T}^n)$ be  the  commutative algebra of polynomial
functions on the $n$-torus  with generators $t_j, \, t_j^*$
satisfying $t_j t_j^*=1= t_j^* t_j$ (no sum on $j$) for
$j=1,\dots, n$. It is a $*$-Hopf algebra with costructures 
\beq
\Delta(t_{j})=t_{j}\ot t_{j}\, ,\quad \varepsilon(t_{j}) = 1 \, , \quad S(t_{j})=t^{*}_{j} \, .
\eeq
Hence $A=\O (\mathbb{T}^n)$, with coaction $\Delta: A \ra A \ot \O (\mathbb{T}^n)$, is a Hopf--Galois extension of $B=\IC$.
Vertical automorphisms $\F \in \Aut{A}=\Hom_{\A^{\O (\mathbb{T}^n)}}(A, A)$ are determined by their action on the generators since they are algebra maps. In turn, $\O (\mathbb{T}^n)$-equivariance gives
$\F(t_i)=(\varepsilon\ot \id)\Delta(\F(t_i))=\varepsilon(\F(t_i))\,t_i 
$, and similarly for $t^*_i$, so that $\F\in \Hom_{\A
  ^{\O (\mathbb{T}^n)}}(A, A) $ is
determined by $\lambda_i=\varepsilon(F(t_i))$ and $\lambda^*_i=\varepsilon(F(t^*_i))$ with $\lambda_i \lambda_i^*=1$. Thus the gauge group is $\Aut{A}\simeq\mathbb{T}^n$. The result is in agreement with Lemma \ref{lem:Gab} 
which would give:
$$
\G_A= (\{ \f: \O (\mathbb{T}^n) \ra \IC \mbox{ algebra maps} \}, *)
$$
that is, $\G_A$ as the set of characters of the algebra $\O (\mathbb{T}^n)$, hence 
$\G_A = \mathbb{T}^n$.

This example can be directly generalised to a $G$-bundle over a point, with $G$
any affine algebraic group, that is a subgroup of $GL(n,\IC)$.
\qed
\end{ex}

\subsection{Equivalence of the gauge groups}~\\[0.5em]
We show the equivalence $\G_A\simeq \Aut{A}$. 

\begin{prop}\label{prop:theta} Let $(H,A)$ be a coquasitriangular Hopf
 algebra,  and $B=A^{coH}\subseteq A$ a Hopf--Galois extension, where
  $A\in \qcr$ is a quasi-commutative $H$-comodule algebra.
The groups $(\G_A, *)$  and $(\Aut{A}, \cdot)$ are isomorphic via the map
\begin{align}\label{isoGAVA}
{\underline{\theta}}_{\,A} : \G_A  & \longrightarrow  \Aut{A}  \\
\f & \mapsto  \F_{\f} := m_A \circ (\id_A \bot \f) \circ \bdelta^A  \quad : a   \mapsto \zero{a} \f(\one{a}) ~,\nn 
 \end{align}
with inverse
$$~~ \qquad ~~~~\F \mapsto \f_{\F} := m_A \circ (\id_A \bot_B \F) \circ \tau \quad : h \mapsto  \tuno{h} \F(\tdue{h})~   .
$$
\end{prop}

\begin{proof}
As mentioned after equation \eqref{AVcomm}, without the requirement that elements of 
$\G_A$  and $\Aut{A}$ are
algebra maps, the group isomorphism was proven in
\cite[\S 5]{brz-tr}  using the linear map $\theta_A: \f\mapsto \F_\f=m_A \circ (\id_A \ot \f) \circ \delta^A$. 
When $A\in \qcr$ we restrict to gauge transformations that are algebra
maps. Since $\theta_A$ restricts as a linear map to
${\underline{\theta}}_{\,A}$ in \eqref{isoGAVA} we  just have to
show that when $\f$ is an algebra map, the corresponding $\F_\f$ is an
algebra map and vice versa. This is so because $\F_\f$ is the composition of
the algebra maps $m_A$, $\id_A\bot \f$ and $\bdelta^A$, and similarly,
for $\f_F$.
\end{proof}
\begin{rem}

When the Hopf algebra $H$ and the $H$-comodule algebra $A$ are both  equipped with  compatible $*$-structures,  that is such that the coaction is a $*$-algebra map, the morphisms
 which constitute the gauge group $\G_A \simeq \Aut{A}$ will also be required to be  compatible with the $*$-structures.
\end{rem}

\section{Deformations by 2-cocycles}
A general theory of Drinfeld-twist deformation of Hopf--Galois extensions was developed in \cite{ppca}.
We specialise this theory to  coquasitriangular Hopf algebras (so that the canonical map is an algebra map) and 
study the corresponding gauge groups in the context of the theory presented in the previous section.

\subsection{Twisting comodule algebras and coalgebras by  2-cocycles}\label{sec:twists-hopf}~\\[.5em]
We first recall some relevant results from the general theory of 2-cocycle
deformations of algebras and comodules  \cite{drin1,drin2,doi}; we follow \cite[\S 2.2]{ppca}.

Let $H=(H,m,1_H,\Delta,\varepsilon,S)$ be a Hopf algebra.
\begin{defi}
A unital convolution invertible 2-cocycle, or simply a \textbf{2-cocycle}, on $H$ is 
a $\bbK$-linear map $\cot:H \otimes H \ra \bbK$ which is unital, that is  $\co{h}{\1}= \varepsilon(h) = \co{\1}{h}$, for all  $ h\in H$, 
invertible for the convolution product and satisfies the
 $2$-cocycle condition
\beq\label{lcocycle}
\co{\one{g}}{\one{h}} \co{\two{g}\two{h}}{k} =  \co{\one{h}}{\one{k}} \co{g}{\two{h}\two{k}}~,
\eeq
for all $g,h, k \in H$. 
\end{defi}
For $\cot$ a 2-cocycle, we denote by   $\bar\gamma : H\otimes H\to\bbK$ its  convolution inverse. The condition \eqref{lcocycle} 
can be equivalently written in terms of $\bar\gamma$ as 
\beq\label{ii} 
\coin{\one{g}\one{h}}{k} \coin{\two{g}}{\two{h}}  =  \coin{g}{\one{h}\one{k}} \coin{\two{h}}{\two{k}}\,, 
\eeq
for all $g,h, k \in H$.

Given a  2-cocycle $\cot$ on $H$, the map $m_\cot :=\cot * m * \bar{\cot}$, 
\beq\label{hopf-twist}
m_{\cot} (h \ot k):= h \mt k:= \co{\one{h}}{\one{k}} \,\two{h}\two{k}\, \coin{\three{h}}{\three{k}}~, 
\eeq
for $h,k \in H,$ defines an associative product on (the $\bbK$-module underlying) $H$. The resulting algebra 
$\hg:=(H,m_\cot,\1_H)$ is a Hopf algebra with coproduct $\Delta$ and
counit $\varepsilon$ that are those of $H$, 
and with antipode  $S_\cot:= u_\cot *S*\bar{u}_\cot$, where 
\begin{flalign}\label{uxS}
u_\cot:  H\longrightarrow \bbK ~, ~~ h\longmapsto \co{\one{h}}{S(\two{h})}  ~,\\
\bar{u}_\cot: H\longrightarrow \bbK~,~~ h \longmapsto \coin{S(\one{h})}{\two{h}} ~, \nn
\end{flalign}
(one  the convolution inverse of the other). 

The passage from $H$ to $\hg$ affects also the category $\M^H$ of (right) $H$-comodules. 
Since the comodule condition \eqref{eqn:Hcomodule} only involves the coalgebra structure of $H$, and  
$\hg$ coincides with $H$ as a coalgebra, any $H$-comodule $V\in\M^H$ with coaction $\delta^V$ 
is a right $\hg$-comodule when $\delta^V$ is thought of as a map $\delta^V:V \ra V \ot \hg$.  
When thinking of $V$ as an object in $\M^{\hg}$ we denote it by
$V_\cot$ and the coaction by $\delta^{V_\cot} : V_\cot \to V_\cot \otimes \hg$. 
For the same reason, any morphism $\psi : V\to W $ in $ \M^H$
can be thought as a morphism $\psi  : V_\cot\to W_\cot$ in $\M^{\hg}$.
Indeed the (identity) functor 
\begin{equation}\label{functGamma}
\Gamma : \M^H \to\M^{\hg}~,
\end{equation}
defined on objects by $\Gamma(V):=V_\cot$ and on morphisms by $\Gamma(\psi):=\psi$, is an equivalence of categories. The convolution inverse $\bar\cot$ twists back $\hg$ to $(\hg)_{\bar{\cot}}=H$ and $V_\cot$ to $(V_\cot)_{\bar{\cot}}=V$.

We denote by 
$(\M^{\hg},\ot^\cot)$ the monoidal category of comodules for 
the Hopf algebra $\hg$. 
Explicitly, for all objects  $V_\cot,W_\cot\in  \M^{\hg}$
(with coactions  $\delta^{V_\cot}$ and $\delta^{W_\cot}$), the right 
$H^\cot$-coaction on $V_\cot\ot^\cot W_\cot $ is
given,  following  \eqref{deltaVW}, by 
\begin{align}\label{deltaVWcot}
\delta^{V_\cot \ot^\cot W_\cot} :V_\cot\otimes^\cot W_\cot 
                                                        \longrightarrow  V_\cot \otimes^\cot W_\cot\otimes \hg~, 
                                                        \quad
 v\otimes^\cot w \longmapsto  \zero{v}\otimes^\cot \zero{w} \otimes \one{v}\mt\one{w} ~.
\end{align}
\begin{prop}\label{pro:funct}
The functor $\Gamma : \M^H \to\M^{\hg}$ together with the natural isomorphism   $\varphi : \otimes^\cot \circ
(\Gamma\times\Gamma)\Rightarrow \Gamma\circ \otimes$ given for objects $V,W\in \M^H$ by the isomorphism of ${\hg}$-comodules
\begin{align}\label{nt}
\varphi_{V,W}: V_\cot \ot^\cot W_\cot \longrightarrow  (V \ot W)_\cot  \, ,
\quad
v \ot^\cot w \longmapsto \zero{v} \ot \zero{w} ~\coin{\one{v}}{\one{w}} ~,
\end{align}
is an equivalence between the  monoidal categories $(\M^H, \ot)$ and 
$(\M^{\hg}, \ot^\cot)$.  
\end{prop}

The functor $\Gamma$ induces an equivalence of categories of comodule algebras
\beq\label{functGammaA}
\Gamma : \A^H\to \A^{\hg} \, , \quad {(A,m_A=\cdot\,,\eta_A, \delta^A)
  \mapsto (A_\cot ,m_{A_\cot}\!=\mtco\,,\eta_{A_\cot}, \delta^{A_\cot})}
\eeq
which is not the identity on objects any longer. 
Given an object $A\in \A^H$ with multiplication $m_A$ and unit $\eta_A$, in order for the coaction $\delta^{A_\cot}$ to be an algebra map one has to define a new product on 
$A_\cot=\Gamma(A)$. The new algebra
structure $m_{A_\cot} , \eta_{A_\cot}$ on
$A_\cot\in \A^{\hg}$ is defined by using the components
$\varphi_{\text{--},\text{--}}$ in \eqref{nt} of the natural
isomorphism $\varphi$, 
and by requiring the commutativity of the diagrams
\begin{flalign*}
\xymatrix{
\ar[d]_-{\varphi_{A,A}} A_\cot \otimes^\cot A_\cot \ar[rr]^-{m_{A_\cot}} && A_\cot && \bbK \ar[d]_-{\simeq} \ar[rr]^-{\eta_{A_\cot}} && A_\cot\\
 (A\otimes A)_\cot\ar[rru]_-{\Gamma(m_A)} &&  && \Gamma(\bbK)\ar[rru]_-{\Gamma(\eta_A)} &&
}
\end{flalign*}
in the category $\M^{\hg}$. 
Explicitly we have $\eta_{A_\cot} = \eta_A$ and the deformed product reads as
\begin{align}\label{rmod-twist} 
 m_{A_\cot} : A_\cot \otimes^\cot A_\cot \,\longrightarrow A_\cot \, , \quad
a\otimes^\cot a^\prime \,\longmapsto a \mtco a':=\zero{a} \zero{a'} \,\coin{\one{a}}{\one{a'}} ~.
 \end{align}
Moreover, for any $\A^H$-morphism $\psi : A\to A^\prime$ one checks that 
 $\Gamma(\psi)=\psi : A_\cot \to A^\prime_\cot$ is a morphism in $\A^{\hg}$. 
With similar constructions, for $A,C\in\A^H $, one obtains equivalences 
\beq\label{iso-rel}
\Gamma : {}_A\M^H\to {}_{A_\cot}\M^{\hg}~,~~
\Gamma : {\mathcal M}_{C}{}^H\to {\mathcal  M}{}_{C_\cot}{}^{\hg}~,~~
\Gamma : {}_A{\mathcal M}_{C}{}^H\to {}_{A_\cot}{\mathcal  M}_{C_\cot}{}^{\hg}
\eeq
for the categories of relative Hopf-modules.
 
The functor $\Gamma$ also induces an equivalence of categories of comodule coalgebras
\beq\label{functGammaC}
\Gamma : \C^H\to\C^{\hg} \, , \quad {(C,\Delta_C,\varepsilon_C, \delta^C) \mapsto (C_\cot ,\Delta_{C_\cot},\varepsilon_{C_\cot},\delta^{C_\cot})} \; .
\eeq
Each $H$-comodule coalgebra $C$ with co-structures ($\Delta_C,\varepsilon_C)$
is mapped to the $\hg$-comodule coalgebra $C_\cot=\Gamma(C)$ with co-structures 
$(\Delta_{C_\cot},\varepsilon_{C_\cot})$ 
defined by the commutativity of the diagrams
\begin{flalign*}
\xymatrix{
\ar[drr]_-{\Gamma(\Delta_C)}C_\cot \ar[rr]^-{{\Delta_C}_{\cot}} && C_\cot \otimes^\cot C_\cot\ar[d]^-{\varphi_{C,C}} && \ar[drr]_-{\Gamma(\varepsilon_C)} C_\cot \ar[rr]^-{{\varepsilon_C}_\cot} && \bbK\ar[d]^-{\simeq}\\
&& (C\otimes C)_\cot && && \Gamma(\bbK)
} 
\end{flalign*} 
in the category $\M^{\hg}$.
The
deformed coproduct explicitly reads 
\beq\label{cc-twist} 
 {\Delta_C}_\cot : C_\cot\longrightarrow C_\cot\otimes^\cot C_\cot ~,~~c \longmapsto  \zero{(\one{c})} \ot^\cot \zero{(\two{c})} \, \cot \left(\one{(\one{c})} \ot \one{(\two{c})}\right) ~,
\eeq 
while ${\varepsilon_C}_\cot = \varepsilon_C$. 
As before, $\Gamma$ acts as the identity on morphisms.

\begin{ex}\label{Hcomco}
The right $H$-comodule $\bh=(H,\Ad)$ is a comodule coalgebra with
coproduct and counit those of the Hopf algebra $H$, 
$\Delta_{\bh}=\Delta_H$ and
$\varepsilon_{\bh}=\varepsilon_H$.  Its twist deformation
$\bh_\cot:=(\Gamma(\bh), \Delta_{\bh_\cot}, \varepsilon_{\bh_\cot}, \Ad)$ is an $\hg$-comodule coalgebra.
Explicitly the coproduct \eqref{cc-twist} of an element $h \in \bh_\cot$ reads
$
\Delta_{\bh_\cot}(h)= \two{h} \ot^\cot  \five{h}
\co{S(\one{h})\three{h}}{S(\four{h})\six{h}} $.

On the other hand, given a twist $\cot$ on $H$, 
we have a second 
$\hg$-comodule coalgebra. It is given by the 
right $\hg$-comodule $\underline{\hg}=(\hg,  \Ad_\cot)$ with coaction
$$
\mathrm{Ad}_{\cot}: \underline{\hg} \longrightarrow \underline{\hg}\otimes \hg~,~~
h \longmapsto \two{h} \ot^\cot S_\cot (\one{h}) \mt \three{h} ~ 
$$
and 
coproduct and counit those of the twisted Hopf algebra $\hg$, that is, those of $H$:
$\Delta_{\underline{\hg}}=\Delta_H$ and
$\varepsilon_{\underline{\hg}}=\varepsilon_H$.  
\qed\end{ex}

We
recall from \cite[Thm.3.4]{ppca} that the comodule coalgebras  $\underline{\hg}$ and $\bh_\cot$ 
are isomorphic:
\begin{thm}\label{thm:Qmap}
The $\bbK$-linear map 
\begin{align}\label{mapQ}
\Q : \underline{\hg} \longrightarrow \bh_\cot ~,~~
h \longmapsto \three{h} \, u_\cot(\one{h}) \, \coin{S(\two{h})}{\four{h}}
\end{align}
is an isomorphism of right $\hg$-comodule coalgebras, with inverse
\beq\label{mapQinv}
\Q^{-1} : \bh_\cot \longrightarrow \underline{\hg} ~,~~
h \longmapsto \three{h} \, \bar{u}_\cot (\two{h}) \, \co{S(\one{h})}{\four{h}}~.
\eeq
\end{thm}

\subsection{The coquasitriangular  case}~\\[.5em]
In this section we consider $2$-cocycles on coquasitriangular Hopf algebras and study  twisted associated bialgebras.

Recall that if $H$ is coquasitriangular,  the category $\A^H$ is
monoidal (see Proposition \ref{prop:tpr}). Also, as mentioned in Example \ref{ex:rug}, if $R$ is the
universal $R$-form of $H$, the twisted Hopf algebra $H_\gamma$ is coquasitriangular with universal 
$R$-form $\r_\cot =\cot_{21} *\r* \bar{\cot}$.
 
\begin{prop}\label{prop:funct-alg} 
Let $(H,R)$ be a coquasitriangular Hopf algebra and $\gamma$ a $2$-cocycle on $H$. 
There is an equivalence of monoidal categories between $(\A^H, \bot)$ and $(\A^{\hg}, \bot^\cot)$ 
given by the functor $\Gamma : \A^H \to\A^{\hg}$ {in \eqref{functGammaA}} and the isomorphisms in $\A^{\hg}$
\begin{align} \label{nt-alg}
\varphi_{A,C}: A_\cot \bot^\cot C_\cot \longrightarrow  (A \bot C)_\cot  ~,
\quad
a \bot^\cot c \longmapsto   \zero{a} \bot \zero{c} ~\coin{\one{a}}{\one{c}} ~,\nn
\end{align}
with 
$A_\cot \bot^\cot C_\cot$  the braided tensor product of the algebras $A_\cot$ and $ C_\cot$, and 
$(A \bot C)_\cot$ the image via $\Gamma$ of the braided tensor product of the algebras $A$ and $C$. (Cf. Proposition \ref{pro:funct}.)
\end{prop}
\begin{proof}
Due to Proposition \ref{pro:funct} we just need to prove that the
isomorphisms $\varphi_{A,C}$ in ${\M}^{\hg}$  are also algebra maps:
$\varphi_{A,C}\circ m_{A_\cot\bot^\cot C_\cot}\,=\,m_{(A \bot
  C)_\cot}\circ (\varphi_{A,C}\ot \varphi_{A,C})$,
that is,
\beq\label{al-map}
\varphi_{A,C}\big((a \bot^\cot c) \tpr (a' \bot^\cot c')\big) = 
\varphi_{A,C} (a \bot^\cot c) \tpr_\cot \,\varphi_{A,C} (a' \bot^\cot c')~,
\eeq
 for all $a,a',\in A_\cot, c,c'\in C_\cot$.  
 Here the $\tpr$-product on the l.h.s. is the product in the braided
 tensor product algebra $A_\cot \bot^\cot C_\cot$ (defined in \eqref{tpr}),  while
 the $\tpr_\cot$-product on the r.h.s. is the twist deformation  (as in \eqref{rmod-twist}) of the $\tpr$-product in the tensor
product algebra $A \bot C$.

We prove \eqref{al-map} by first evaluating it
on specific products and then using the
associativity of the  multiplications $\tpr$ and $\tpr_\cot$.
Firstly we show that 
\beq\label{leftAlin}
\varphi_{A,C}\big((a \bot^\cot 1_C) \tpr (a' \bot^\cot c')\big) = 
\varphi_{A,C} (a \bot^\cot 1_C) \tpr_\cot \,\varphi_{A,C} (a' \bot^\cot
c')~.
\eeq
Explicitly 
\begin{align*}
\varphi_{A,C}\big((a \bot^\cot 1_C) \tpr (a' \bot^\cot c')\big) 
&=  
\varphi_{A,C} \left( (a \mtco a') \bot^\cot c') \right) 
\\
&= \nn \varphi_{A,C}   \left( (\zero{a}  \zero{a'})   \bot^\cot {c'}\right) ~\coin{\one{a}}{\one{a'}}
\\
&= \nn \zero{a}  \zero{a'} \bot \zero{c'}~ \coin{\one{a}\one{a'}}{\one{c}} \coin{\two{a}}{\two{a'}}
\\ &= \nn
\zero{a} \zero{a'} \bot \zero{c'}~ \coin{\one{a}}{\one{a'}\one{c'}} \coin{\two{a'}}{\two{c'}}
\\ \nn
&=(a\bot 1_C) \tpr_\gamma   (\zero{a'} \bot \zero{c'}) \;\coin{\one{a'}}{\one{c'}}  
\\
&= (a\bot 1_C) \tpr_\gamma  \,\varphi_{A,C} (a' \bot^\cot c') 
\\
&= \varphi_{A,C}(a\bot^\cot 1_C) \tpr_\gamma  \,\varphi_{A,C} (a' \bot^\cot c') 
\end{align*}
having also used the 2-cocycle condition \eqref{ii} for the fourth and fifth equalities.
Since $\varphi_{A,C}(a \bot^\cot 1_C)= a \bot 1_C $, the identity \eqref{leftAlin}  just expresses the fact that 
$\varphi_{A,C}$ are isomorphisms in ${}_{A_\cot}{\mathcal
  M}$
  for the obvious left action of $A$ and $A_\cot$ on
$A\bot C$ and $A_\cot\bot^\cot C_\cot$ respectively, see \eqref{iso-rel}. 
Similarly, since the $\varphi_{A,C}$ are isomorphisms in ${\mathcal
  M}{}_{C_\cot}$, we have 
$$
\varphi_{A,C}\big((a \bot^\cot c) \tpr (1_A \bot^\cot c')\big) = 
\varphi_{A,C} (a \bot^\cot c) \tpr_\cot \,\varphi_{A,C} (1_A \bot^\cot
c')~.
$$
Finally we have
\begin{align*}
\varphi_{A,C}\big((1_A \bot^\cot c) \tpr (a' \bot^\cot 1_C)\big) \nn
&= \varphi_{A,C}(\zero{a'}\bot^\cot \zero{c})\,\rug{\one{c}}{\one{a'}}\\
&=\nn
\zero{a'}\bot \zero{c}\,\bar\cot(\one{a'}\ot\one{c})\rug{\two{c}}{\two{a'}}\\
&=\nn
\zero{a'}\bot \zero{c}\,\ru{\one{c}}{\one{a'}}\bar\cot(\two{c}\ot\two{a})\\
&=\nn
\zero{(1_A\bot c)}\tpr\zero{(a'\bot  1_C)}\,\bar\cot(\one{c}\ot\one{a})\\
&=
\varphi_{A,C} (1_A \bot^\cot c) \tpr_\cot \,\varphi_{A,C} (a' \bot^\cot 1_C) 
\end{align*}
where in the third line we used that $\r_\cot =\cot_{21} *\r*
\bar{\cot}$ (cf. \eqref{rug}).
Thus on the generic product of two elements 
$ (a \bot^\cot c) = (a \bot^\cot 1) \tpr (1 \bot^\cot c) $
and 
$(a' \bot^\cot c') = (a' \bot^\cot 1) \tpr (1 \bot^\cot c')
$
the map $\varphi_{A,C}$  is an  algebra map.
\end{proof}
{This result and Corollary \ref{catqc} lead to}
\begin{cor}\label{moncatqc}
Let $(H,\r)$ be a cotriangular Hopf algebra. 
With the notations of Proposition \ref{prop:funct-alg}, the restriction
of the functor
$(\Gamma, \varphi) :(\A^H,\bot)\rightarrow (\A^{\hg},\bot^\cot)$ to
the subcategory $(\qcr,\bot)$ of quasi-commutative  comodule algebras induces an equivalence of monoidal categories 
$(\qcr,\bot)\simeq(\qcrc,\bot^\cot)$. 
\end{cor}

In the context of coquasitriangular  Hopf algebras in addition to the twist deformation of comodule algebras and comodule coalgebras  (considered in \S \ref{sec:twists-hopf}) one next deforms  braided bialgebras associated with $H$ (see Definition \ref{def:bbH}). 

\begin{prop}\label{prop:td-bb}
Let $(L, m_L, \eta_L, \Delta_L,\varepsilon_L, \delta^L)$ be a braided
bialgebra associated with a coquasitriangular  Hopf algebra $(H,R)$,
and $\cot$ a 2-cocycle on $H$.
The twist deformation of  $(L, m_L, \eta_L, \delta^L)$ 
as an $H$-comodule algebra and of  
$(L, \Delta_L,\varepsilon_L, \delta^L)$  as an $H$-comodule coalgebra
gives the braided bialgebra 
$({L_\cot}, m_{L_\cot}, \eta_{L_\cot}, \Delta_{L_\cot},\varepsilon_{L_\cot}, \delta^{L_\cot})$
associated with the twisted Hopf algebra $\hg$. Thus, ${L_\cot}$ is a bialgebra in the  braided monoidal category $(\M^{H_\cot}\!, \ot^\cot\!,
\Psi^{R_\cot})$ of $H_\cot$-comodules.
 Moreover, if $L$ is a braided Hopf algebra, then $L_\cot$ is a braided Hopf algebra with
antipode $S_{L_\cot}=\Gamma(S_L)$, $\ell\mapsto  S_{L_\cot}(\ell)=S_L(\ell)$.
\end{prop}
\begin{proof}
By the general theory, 
the  $\hg$-comodule 
$L_\cot =\Gamma(L)$
  is an $\hg$-comodule algebra with  unit $1_{L_\cot}=1_L$ (or $\eta_{L_\cot}=\Gamma(\eta_L)$) 
 and deformed product 
$m_{L_\cot}= \Gamma(m_L) \circ \varphi_{L,L}$ given by  \eqref{rmod-twist}.
Moreover $L_\cot$ is an $\hg$-comodule coalgebra with
counit $\varepsilon_{L_\cot}=\Gamma(\varepsilon_L)$, $\ell\mapsto  \varepsilon_{L_\cot}(\ell)=\varepsilon_L(\ell)$, 
and deformed coproduct $\Delta_{L_\cot} = \varphi_{L,L}^{-1} \circ \Gamma(\Delta_L)$ given by \eqref{cc-twist}.
In order for 
$({L_\cot}, m_{L_\cot}, \eta_{L_\cot},
\Delta_{L_\cot},\varepsilon_{L_\cot}, \delta^{L_\cot})$ to be a braided
bialgebra associated with the twisted Hopf algebra $\hg$ it suffices 
to show (cf. Definition \ref{def:bbH}) that
$\Delta_{L_\cot}$ is an algebra map for the product
  $m_\cot$ in $L_\cot$ and the product $m_{L_\cot \bot^\cot L_\cot} =
  (m_{L_\cot} \otimes^\cot m_{L_\cot})  \circ (\id_{L_\cot}\otimes^\cot
  \Psi^{R_\cot}_{L_\cot,L_\cot} \otimes^\cot \id_{L_\cot})$ in
  $L_\cot\bot^\cot L_\cot$:  
\beq\label{cop-braid-twist}
\Delta_{L_\cot}\circ m_{L_\cot}  = m_{L_\cot \bot^\cot L_\cot}\circ (\Delta_{L_\cot}\otimes^\cot\Delta_{L_\cot}) \, .
\eeq
On the one hand,
\begin{align*}
\Delta_{L_\cot} \circ m_{L_\cot} &= \varphi_{L,L}^{-1} \circ \Gamma(\Delta_L) \circ  \Gamma(m_L) \circ \varphi_{L,L} =
\varphi_{L,L}^{-1} \circ \Gamma(\Delta_L \circ m_L) \circ \varphi_{L,L}\\
& =  \varphi_{L,L}^{-1} \circ \Gamma \big(m_{L \bot L} \circ (\Delta_L \ot\Delta_L) \big) \circ \varphi_{L,L}
=  \varphi_{L,L}^{-1} \circ \Gamma (m_{L \bot L})\circ \Gamma (\Delta_L \ot\Delta_L) \circ \varphi_{L,L}
\end{align*}
where we have used  that $L$ is a braided bialgebra associated
with $H$ (cf. \eqref{cop-braid}). \\
On the other hand, the maps $\varphi_{-,-}$  satisfy  
$$
 \Gamma (\Delta_L \ot \Delta_L) \circ \varphi_{L,L} =\varphi_{_{L \bot L ,  L \bot L}}\circ \big(\Gamma(\Delta_L) \ot^\cot \Gamma(\Delta_L) \big)
\,  
$$
as it can be verified on generic elements in $L_\cot \bot^\cot L_\cot$
by using that $L$ is an $H$-comodule coalgebra (cf. equation \eqref{com-co}).
Thus,
\begin{align*}
\Delta_{L_\cot} \circ m_{L_\cot} &= \varphi_{L,L}^{-1} \circ \Gamma (m_{L \bot L} )\circ 
\varphi_{_{L \bot L ,  L \bot L}}\circ \big(\Gamma(\Delta_L) \ot^\cot \Gamma(\Delta_L) \big) \\
&=\varphi_{L,L}^{-1}  \circ m_{(L \bot  L)_\cot} \circ \big(\Gamma(\Delta_L) \ot^\cot \Gamma(\Delta_L) \big) \\
& =
m_{L_\cot \bot^\cot L_\cot}\circ (\varphi_{L,L}^{-1} \ot^\cot \varphi_{L,L}^{-1}) \circ \big(\Gamma(\Delta_L) \ot^\cot \Gamma(\Delta_L) \big) \\
& = m_{L_\cot \bot^\cot L_\cot}\circ (\Delta_{L_\cot}\ot^\cot \Delta_{L_\cot}) \, ,
\end{align*}
by using in the third equality that  $\varphi_{-,-}$ are algebra maps
(cf. \eqref{al-map}). 

If in addition $L$ has an antipode $S_L$ (by assumption an $H$-comodule map), 
its image under $\Gamma$, $S_{L_\cot}:=\Gamma(S_L)$, $\ell\mapsto
S_{L_\cot}(\ell)=S_L(\ell)$, is an $H_\cot$-comodule map. We show it is an antipode for 
the twisted bialgebra $L_\cot$.  One easily verifies the equality 
 $$
 \varphi_{L,L} \circ \big(\Gamma(\id_{L}) \ot^\cot \Gamma( S_{L}) \big) \circ \varphi_{L,L}^{-1}= \Gamma(\id_{L} \otimes S_{L})
 $$ of $H_\cot$-equivariant maps. Then
\begin{align*}
m_{L_\cot}  \circ (\id_{L_\cot} \ot S_{L_\cot})\circ \Delta_{L_\cot} 
& =
\Gamma(m_L) \circ \varphi_{L,L} \circ \big(\Gamma(\id_{L}) \ot \Gamma( S_{L}) \big) \circ \varphi_{L,L}^{-1} \circ \Gamma(\Delta_L)
\\
 &=
\Gamma(m_L) \circ \Gamma(\id_{L} \otimes S_{L})  \circ \Gamma(\Delta_L)
\\
&=
\Gamma \big(m_L \circ (\id_{L} \otimes S_{L}) \circ \Delta_L \big)
\\
&=
\Gamma \big( \eta_L\circ \varepsilon_L \big)
=
\Gamma( \eta_L) \circ \Gamma(\varepsilon_L)
\\
&=  \eta_{L_\cot}  \circ \varepsilon_{L_\cot}  \; .
\end{align*}
Analogously one shows that 
$m_{L_\cot}  \circ (S_{L_\cot} \ot^\cot \id_{L_\cot} )\circ \Delta_{L_\cot} =\eta_{L_\cot}  \circ \varepsilon_{L_\cot} .$
\end{proof}

\begin{ex}
Let $\bh$ be the braided Hopf algebra associated with a coquasitriangular
Hopf algebra $(H,R)$ (cf. Example \ref{ex:bh}). Given a $2$-cocycle
$\cot$ on $H$, by Proposition \ref{prop:td-bb} we have the braided Hopf algebra $\bh_\cot=(\Gamma(\bh), \hpr_\cot, \Delta_{\bh_\cot}, \varepsilon_{\bh_\cot},\eta_{\bh_\cot}, S_{\bh_\cot}, \Ad)$ associated  with the twisted Hopf algebra $\hg$. It is given by the $\hg$-comodule coalgebra  $\bh_\cot$ of Example \ref{Hcomco} endowed with the product 
\beq\label{prod-hgt0}
h \hpr_\cot k  = 
\zero{h} \hpr \zero{k} ~ \coin{\one{h}}{\one{k}}=
\two{h} \hpr \two{k} ~ \coin{S(\one{h})\three{h}}{S(\one{k})\three{k}} \, .
\eeq
This is the twist deformation of the product
$h \hpr k := \two{h}\two{k} \ru{S(\one{h})\three{h}}{S(\one{k})}$ in $\underline{H}$ 
defined by \eqref{hpr}. In terms of the product in $H$, the product $h \hpr_\cot k $ 
is written as
\beq\label{prod-hgt}
h \hpr_\cot k  = \three{h}\three{k} \, \ru{S(\two{h})\four{h}}{S(\two{k})} \coin{S(\one{h})\five{h}}{S(\one{k})\four{k}}~.
\eeq
\qed\end{ex}

In addition to the braided Hopf algebra $\bh_\cot$ there is also the
braided Hopf algebra $\underline{\hg}$ of the coquasitriangular Hopf
algebra $(H_\cot, R_\cot)$ (cf. Example \ref{ex:bh}). 
The product in $\underline{\hg}$ is as in $\eqref{hpr}$: 
for all $h,k \in \underline{\hg}$, one has
\beq\label{lhsQ0}
h \,\underline{\cdot_\cot}\, k := \two{h}\mt \two{k} 
\rug{S_\cot(\one{h})\mt \three{h}}{S_\cot(\one{k})}~.
\eeq
Recalling the product and antipode in $H_\cot$:
$h \mt k= \co{\one{h}}{\one{k}} \,\two{h}\two{k}\,
\coin{\three{h}}{\three{k}}$, 
and  $S_\cot= u_\cot *S*\bar{u}_\cot$, 
we can rewrite $h \,\underline{\cdot_\cot}\, k$ as
\begin{align}\label{lhsQ}
h \,\underline{\cdot_\cot}\, k
& =
u_\cot(\one{h}) \, u_\cot(\one{k}) 
\bar{u}_\cot(\seven{h})\bar{u}_\cot(\five{k}) \co{\eight{h}}{\six{k} } \, \nine{h} \seven{k} \nn 
\\ & \qquad
\coin{\ten{h}}{\eight{k} }  \co{S(\six{h})}{  h_{\scriptscriptstyle{(11)}} }   
\co{S(\four{k})}   {S(\five{h})h_{\scriptscriptstyle{(12)}}} \nn
\\ &  \qquad
\ru{S(\four{h})h_{\scriptscriptstyle{(13)}}}{S(\three{k})}  
\coin{S(\three{h})h_{\scriptscriptstyle{(14)}}}{S(\two{k})}  
\coin{S(\two{h})}{h_{\scriptscriptstyle{(15)}}}   \, .
\end{align}
The  braided Hopf algebras $\underline{\hg}$ and $\bh_\cot$ are isomorphic:
\begin{thm}\label{isoQb}
The $\bbK$-linear map 
$\Q : \underline{\hg} \longrightarrow \bh_\cot$ 
in \eqref{mapQ} with inverse in \eqref{mapQinv} 
is an isomorphism of braided Hopf algebras associated with $\hg$.
\end{thm}

\begin{proof}
We know from Theorem \ref{thm:Qmap} that the map $\Q$ is an
isomorphism of $\hg$-comodule coalgebras. We are left to show that $\Q
$ is an algebra morphism. It maps the unit of
$\underline{H_\cot}$ to the unit of $\bh_\cot$.
In $\underline{H}_\cot$ the product is given by formula \eqref{prod-hgt}.
Thus we have 
\begin{align*}
\Q(h) \hpr_\cot \Q(k) 
&=   u_\cot(\one{h}) \, u_\cot(\one{k}) \,\three{h} \hpr_\cot \three{k} \,  \coin{S(\two{h})}{\four{h}}
 \, \coin{S(\two{k})}{\four{k}}
\\
& =  u_\cot(\one{h}) \, u_\cot(\one{k}) \, \five{h} \five{k} \, \ru{S(\four{h})\six{h}}{S(\four{k})} \, \coin{S(\two{h})}{\eight{h}}
\\
& \qquad \qquad \coin{S(\three{h})\seven{h}}{S(\three{k})\six{k}} 
 \, \coin{S(\two{k})}{\seven{k}} \; 
\\
& =  u_\cot(\one{h}) \, u_\cot(\one{k}) \, \six{h} \five{k} \, \ru{S(\five{h})\seven{h}}{S(\four{k})} 
\, \coin{S(\two{h})}{\ten{h}} \\
& \qquad \qquad
\coin{S(\four{h})\eight{h}S(\three{k})}{\six{k}} 
\, \coin{S(\three{h}) \nine{h} }{S(\two{k})} \, ,
\end{align*}
where for the last equality we  used the cocycle property \eqref{ii}.
Next we use the coquasitriangularity of $H$  to rewrite
$$
\ru{S(\five{h})h_{\scriptscriptstyle{(7)}}}{S(k_{\scriptscriptstyle{(4)}})} 
S(\four{h})h_{\scriptscriptstyle{(8)}}S(\three{k})
= S(k_{\scriptscriptstyle{(4)}})S(\five{h})h_{\scriptscriptstyle{(7)}}
\ru{S(\four{h})h_{\scriptscriptstyle{(8)}}}{S(\three{k})} 
$$
and obtain
\begin{align*}
\Q(h) \hpr_\cot \Q(k) 
&=  u_\cot(\one{h}) \, u_\cot(\one{k}) \, \six{h} \five{k} \,  \ru{S(\four{h})h_{\scriptscriptstyle{(8)}}}{S(\three{k})} 
\, \coin{S(\two{h})}{\ten{h}} \\
& \qquad \qquad
\coin{S(k_{\scriptscriptstyle{(4)}})S(\five{h})h_{\scriptscriptstyle{(7)}}}{\six{k}} 
\, \coin{S(\three{h}) \nine{h} }{S(\two{k})} \, .
\end{align*}
On the other hand  
\begin{multline*}
\!\!\!\!\!\Q(hk) = \three{h}\three{k} \, u_\cot(\one{h}\one{k}) \, \coin{S(\two{k}) S(\two{h})}{\four{h}\four{k}}
\\ ~~~~~~~~~~= \coin{\one{h}}{\one{k}} u_\cot (\two{h})u_\cot (\two{k}) \coin{S(\three{k})}{S(\three{h})}  
\five{h}\five{k} \,  \coin{S(\four{k}) S(\four{h})}{\six{h}\six{k}}
\end{multline*}
where we used
$u_\cot (hk)= \coin{\one{h}}{\one{k}} u_\cot (\two{h})u_\cot (\two{k}) \coin{S(\three{k})}{S(\three{h})}$
that follows from the basic properties of a 2-cocycle.
Then, using the explicit formula for the product  
$h \,\underline{\cdot_\cot}\, k$ given in \eqref{lhsQ}, we have
\begin{align*}
\Q(h \,\underline{\cdot_\cot}\, k) 
& =
u_\cot(\one{h}) \, u_\cot(\one{k}) 
\coin{S(\five{k})}{S(\seven{h})}  
\nine{h}\seven{k} \,  \coin{S(\six{k}) S(\eight{h})}{\ten{h}\eight{k}}
\\ & 
\qquad \coin{h_{\scriptscriptstyle{(11)}} }{\nine{k} }   
\co{S(\six{h})}{  h_{\scriptscriptstyle{(12)}} }   
\co{S(\four{k})}   {S(\five{h})h_{\scriptscriptstyle{(13)}}}
\\ & 
\qquad \ru{S(\four{h})h_{\scriptscriptstyle{(14)}}}{S(\three{k})}  \coin{S(\three{h})h_{\scriptscriptstyle{(15)}}}{S(\two{k})}  
\coin{S(\two{h})}{h_{\scriptscriptstyle{(16)}}}   
\\ & =
u_\cot(\one{h}) \, u_\cot(\one{k}) 
\seven{h}\six{k} \,  \coin{S(\five{k}) S(\six{h})}{\eight{h}\seven{k}}
\\ & 
\qquad \coin{h_{\scriptscriptstyle{(9)}} }{\eight{k} }   
\co{S(\four{k})S(\five{h})}   {h_{\scriptscriptstyle{(10)}}}
\ru{S(\four{h})h_{\scriptscriptstyle{(11)}}}{S(\three{k})}  
\\ & 
\qquad \coin{S(\three{h})h_{\scriptscriptstyle{(12)}}}{S(\two{k})}  
\coin{S(\two{h})}{h_{\scriptscriptstyle{(13)}}}  \, ,
\end{align*}
where to obtain  the last equality we  used the cocycle condition \eqref{lcocycle} on the product
$\coin{S(\five{k})}{S(\seven{h})} \co{S(\six{h})}{  h_{\scriptscriptstyle{(12)}} }   
\co{S(\four{k})}   {S(\five{h})h_{\scriptscriptstyle{(13)}}}$.
Using once again this condition on the product
$\coin{S(\five{k}) S(\six{h})}{\eight{h}\seven{k}}\coin{h_{\scriptscriptstyle{(9)}} }{\eight{k} }   
\co{S(\four{k})S(\five{h})}{h_{\scriptscriptstyle{(10)}}}$ we finally obtain
\begin{align*}
\Q(h \,\underline{\cdot_\cot}\, k) & =
u_\cot(\one{h}) \, u_\cot(\one{k}) 
\six{h}\five{k} \,  \coin{S(\four{k})S(\five{h})\seven{h}}{\six{k}}
\ru{S(\four{h})h_{\scriptscriptstyle{(8)}}}{S(\three{k})}  \\
& 
\qquad \coin{S(\three{h})h_{\scriptscriptstyle{(9)}}}{S(\two{k})}  
\coin{S(\two{h})}{h_{\scriptscriptstyle{(10)}}}   \, .
\end{align*}
Thus, $\Q(h) \hpr_\cot \Q(k) = \Q(h \,\underline{\cdot_\cot}\, k)$. 
\end{proof}

\subsection{Twisting Hopf--Galois extensions}\label{sec:HGtwist}~\\[.5em]
The deformation by $2$-cocycles of Hopf--Galois extensions was addressed in \cite{ppca} for a general Hopf algebra $H$.  
When $H$ is coquasitriangular one has an additional algebra structure. 

Let $(H,R)$ be a coquasitriangular  Hopf algebra, and $A\in \qcr$
  a quasi-commutative $H$-comodule algebra. Consider the
algebra extension $B=A^{coH}\subseteq A$. Let $\gamma$ be
a $2$-cocycle on $H$, and consider  $ A_\cot\in  \qcrc$ and the corresponding algebra extension
$B_\cot=A_\cot^{co\hg}\subseteq A_\cot$. Since the coactions $\delta^A: A\to A\ot H$ and
$\delta^{A_\cot}:A_\cot\to A_\cot\ot \hg$ coincide, $B_\cot=B$ as
$\bbK$-modules; they also coincide as algebras since  
$B$ carries a trivial $H$-coaction so that $m_{B_\cot}=m_B$ (see \eqref{rmod-twist}). 
From Theorem \ref{prop:canam} both canonical maps 
\beq\label{aagacan}
\chi: A\bot_B A \to A\bot\underline{H} \qquad \mbox{and} \qquad 
\chi_\cot : A_\cot \bot_B^\cot A_\cot \, \to \,A_\cot \bot^\cot \underline{\hg}
\eeq
are comodule algebra maps.  
In the context of coquasitriangular Hopf algebras and
quasi-commutative comodule algebras, Theorem 3.6 of \cite{ppca} can
be sharpened: \begin{thm}\label{Thm:diagr-can-al}
Let $(H,R)$ be a coquasitriangular Hopf algebra, and let $\gamma$ be
a $2$-cocycle on $H$. Let $A\in \qcr$ with  
$B=A^{coH} \subseteq A$ and twist deformation $B=A_\cot^{co\hg}\subseteq
A_\cot\in \qcrc$.
Then the following diagram of morphisms in $\A^{\hg}$ 
\begin{flalign}\label{diagr-can-al}
\xymatrix{
\ar[dd]_-{\varphi_{A,A}} A_\cot \bot_B^\cot A_\cot \ar[rr]^-{\chi_\cot} && A_\cot \bot^\cot \underline{\hg}\ar[d]^-{\id\otimes^\cot \Q}\\
&& A_\cot \bot^\cot \bh_\cot \ar[d]^-{\varphi_{A,\bh}}\\
(A \bot_B A)_\cot \ar[rr]^-{\Gamma(\chi)} && (A \bot \bh)_\cot
}
\end{flalign}
is commutative. If $H$ is cotriangular  the diagram consists of morphisms in $\qcrc$. 
\end{thm}
\begin{proof}
In \cite[Thm.3.6]{ppca} the commutativity of the diagram was shown for morphisms in ${}_{A_\cot}\M_{A_\cot}{}^{\hg}$, and hence in  $\M{}^{\hg}$. 
Since all maps in the diagram have been shown to be algebra maps (see
Theorem \ref{prop:canam}, Proposition \ref{prop:funct-alg} and Theorem \ref{isoQb}), the diagram is indeed in $\A^{\hg}$. When $H$ is cotriangular, so is $\hg$; then $\bh\in \A_\bc^H$ 
and also $\bh_\cot \simeq  \underline{\hg} \in \qcrc$ (from Example
\ref{ex:bh}). 
Thus,  due to Corollary \ref{moncatqc}, the diagram is in $\qcrc$.  
\end{proof}
We remark that since all vertical arrows in the diagram \eqref{diagr-can-al} are isomorphisms, the commutativity of \eqref{diagr-can-al}
implies that the extension $B \subseteq A_\cot$ is an 
$\hg$-Hopf--Galois extension if and only if the starting extension 
$B \subseteq A$ is such for the Hopf algebra $H$ (see Corollary 3.7 in \cite{ppca}).

\subsection{Twisting gauge groups}~\\[.5em]
In the hypothesis of Theorem \ref{Thm:diagr-can-al},
let $B =A^{co H}\subseteq A$ be Hopf--Galois.
The gauge group 
\[
\G_{A_\cot}:=\Hom_{\A ^{\hg}}\big(\,\underline{\hg}\,, A_\cot\big) 
\]
of the Hopf--Galois extension 
$B=A_\cot^{co\hg}\subseteq A_\cot\in  \qcrc$ is isomorphic to the gauge group  
$\G_{A}=\Hom_{\A ^{H}}\big(\,\bh\,, A\big)$ of the initial one:

\begin{prop}\label{thm:gamma-ad}
Let $(H,R)$ be a coquasitriangular Hopf algebra, $A\in \qcr$ and let $\gamma$ be
a $2$-cocycle on $H$. The isomorphism 
$\Gamma: \Hom_{\A ^{H}}\big(\,\bh\,, A\big)
\rightarrow \Hom_{\A^{\hg}}\big(\,{\bh_\cot}\,, A_\cot\big)$, induced by the functor $\Gamma :\A^H  \to \A^{\hg}$ in equation \eqref{functGammaA}, 
when
composed with the pull-back 
$
\Q^*:\Hom_{\A^{\hg}}\big(\,\bh_\cot \,, A_\cot\big) \to \Hom_{\A^{\hg}}\big(\,\underline{\hg}\,, A_\cot\big) 
$ 
of the map  $\Q:
\underline{\hg} \longrightarrow \bh_\cot $ in
\eqref{mapQ}, gives the  group isomorphism 
\begin{align*}
\Gamma_\Q:=\Q^*\circ \Gamma:
\Hom_{\A ^{H}}\big(\,\bh\,, A\big) \, 
\stackrel{\simeq\,}{\longrightarrow} \, 
\Hom_{\A ^{\hg}}\big(\,\underline{\hg}\,, A_\cot\big) \, .
\end{align*}
Explicitly, for all $h \in \underline{\hg}$, one has:
\[
\Gamma_\Q(\f) : h \mapsto \f(\Q(h))=\f(\three{h}) \,
u_\cot(\one{h}) \, \coin{S(\two{h})}{\four{h}}\,.
\]
\end{prop}
\begin{proof}
The map $\Gamma_\Q$ is invertible since it
is  composition of invertible maps:
$$
\Hom_{\A ^{H}}\big(\,\bh\,, A\big)
\stackrel{\Gamma}{\longrightarrow}\Hom_{\A
  ^{\hg}}\big(\,{\bh_\cot}\,, A_\cot\big) 
\stackrel{\Q^*}{\longrightarrow}\Hom_{\A ^{\hg}}\big(\,\underline{\hg}\,, A_\cot\big) \, .
$$
We are only left to verify that $\Gamma_\Q$ preserves the group product, that is that 
$$
\Gamma_\Q (\f * \g) = \Gamma_\Q (\f) *_\cot \Gamma_\Q ( \g)  \, ,
$$
for each $\f,\g: \bh \ra A $, where $*$ denotes the convolution
product in $\Hom_{\A ^{H}}\big(\,\bh\,, A\big)$ and $*_\cot$ denotes
the convolution product in $\Hom_{\A ^{\hg}}\big(\,\underline{\hg}\,,
A_\cot\big)$.

On the one hand, for each $h \in \underline{\hg}$,
\begin{align*}
\Gamma_\Q (\f * \g) (h) &=\Gamma (\f * \g) (\Q(h)) =  
(\f * \g) (\Q(h)) = 
 \f (\one{\Q(h)})_{\:\!} \g (\two{\Q(h)})~,
\end{align*}
with the Sweedler notation $\Delta_\bh(h)=\Delta_H(h)=\one{h}\ot \two{h}$.
By considering definition \eqref{rmod-twist}, in order to
express the product in $A$ in terms of that in $A_\cot$ we further have
\begin{align*}
\Gamma_\Q (\f * \g) (h) &
=\zero{(\f  (\Q(\one{h})))} \mtco \zero{(\g (\Q(\two{h})))}\,
\co{\one{(\f  (\Q(\one{h})))} }{\one{(\g (\Q(\two{h})))}}
\\
&
=  \f  (\zero{\Q(\one{h})}) \mtco  \g (\zero{\Q(\two{h})})\,\,
\co{\one{\Q(\one{h})} }{\one{\Q(\two{h})}}\, ,
\end{align*}
where we used that $\f,\g$ are $H$-comodule maps.
Recalling the definition \eqref{cc-twist} of the twisted coproduct 
$\Delta_{\bh_\cot}$ in ${\bh_\cot}$,  the above expression simplifies to
$$
\Gamma_\Q (\f * \g) (h) 
=  \f  (\Q(h)_{\scriptscriptstyle{[1]}}) \mtco  \g (\Q(h)_{\scriptscriptstyle{[2]}} )
$$
with the Sweedler notation $\Delta_{\bh_\cot} (h)=
h_{\scriptscriptstyle{[1]}} \ot h_{\scriptscriptstyle{[2]}}$ for the
components of the  coproduct $\Delta_{\bh_\cot}$.

On the other hand
\begin{align*}
\left(\Gamma_\Q (\f) *_\cot \Gamma_\Q ( \g)\right)  (h)
&= \f  (\Q(\one{h})) \mtco \g (\Q(\two{h}))
\end{align*}
with the Sweedler notation for  the coproduct in
$\underline{\hg}$ which equals that in $H$.
The identity 
$\Gamma_\Q (\f * \g) = \Gamma_\Q (\f) *_\cot \Gamma_\Q ( \g)  $
then follows recalling that $\Q$ is a coalgebra map.
\end{proof}
It is instructive to recover this isomorphism  considering the
  gauge groups of vertical automorphisms (cf. Proposition 
 \ref{prop:Finv}). The gauge groups $\Aut{A}=\Hom_{_B\A ^H}(A, A)$ and 
 $\Aut{A_\cot}=\Hom_{_B \A^{\hg}}(A_\cot, A_\cot)$ are isomorphic via the functor $\Gamma$:
 
\begin{prop}\label{prop:gamma-aut}
The restriction of the functor $\Gamma :\A^H  \to \A^{\hg} $
in \eqref{functGammaA} to ${_B\A ^H}$  gives a group isomorphism 
\begin{eqnarray*}
\Gamma:\Aut{A}  \stackrel{\simeq}{\longrightarrow}& \Aut{A_\cot} \, , 
\quad
\F:A \ra A  \longmapsto& \F_\cot:=\Gamma(\F):  A_\cot \ra A_\cot \; .
\end{eqnarray*}
\end{prop}
\begin{proof}
Since the functor $\Gamma$ is the identity on morphisms, the linear maps $\F$ and $\F_\cot$ coincide. Clearly, $\Gamma$ preserves the group law which is given by the composition of maps.
\end{proof}
The group isomorphisms in Propositions \ref{thm:gamma-ad} and \ref{prop:gamma-aut} above are related via the isomorphism of Proposition \ref{prop:theta}:
\begin{prop} 
The group isomorphisms ${\underline{\theta}_A}$ and
  ${\underline{\theta}_{A_\cot}}$  given as in Proposition
  \ref{prop:theta},  and the isomorphisms $\Gamma_\Q$ of Proposition 
  \ref{thm:gamma-ad} and  $\Gamma$ of Proposition \ref{prop:gamma-aut}
give the following commutative diagram
\begin{flalign*}
\xymatrix{
\ar[d]_-{\Gamma_\Q} \G_{A}  \ar[rr]^-{\underline{\theta}_A} && \Aut{A} \ar[d]^-{\Gamma}~
\\
\G_{A_\cot} \ar[rr]^-{\underline{\theta}_{A_\cot}} && \Aut{A_\cot}
~.}
\end{flalign*}
\end{prop}
\begin{proof}
Let $\f \in \G_{A}$, we have to show that 
$$(\Gamma \circ \underline{\theta}_{A}) \f :A_\cot \ra A_\cot \, , \quad a \mapsto \zero{a}\f(\one{a})
$$
coincides with
$$
({\underline{\theta}_{A_\cot}} \circ {\Gamma_\Q} )\f :A_\cot \ra A_\cot \, , \quad a \mapsto
\zero{a}\mtco \Gamma_\Q(\f)(\one{a}) =
\zero{a}\mtco
\f(\three{a}) \, u_\cot(\one{a}) \, \coin{S(\two{a})}{\four{a}}~.
$$
Observe that 
$$\zero{a}\f(\one{a})= \zero{a} \mtco \zero{(\f(\two{a}))} \co{\one{a}}{\one{(\f(\two{a}))}}
=
\zero{a} \mtco \f(\three{a}) \co{\one{a}}{S(\two{a}) \four{a}}
$$
where for the first equality we  used (the inverse of) formula
\eqref{rmod-twist} to express the product in $A$ in terms of  the
product $\mtco$ in $A_\cot$ and for the second one $H$-equivariance of
$\f: \bh\to A$.
The equality 
$(\Gamma \circ \underline{\theta}_{A}) \f=  ({\underline{\theta}_{A_\cot}} \circ {\Gamma_\Q} )\f
$
 then follows from the identity
$$
u_\cot(\one{h})\bar\cot(S(\two{h})\otimes k)=\cot(\one{h}\otimes
S(\two{h})k)~, 
$$
for all $h,k\in H$. This is shown by using the cocycle condition and the definition of $u_\cot$ in \eqref{uxS},  (see \cite[Lem.3.2]{ppca}).
\end{proof}

\begin{ex}\label{ex:Hcleft} {\it Noncommutative bundle over a point.} 
Let us consider a commutative Hopf algebra $H$ with trivial $R$-form, 
and the trivial Hopf--Galois extension $\bbK\subseteq H$ with $H$-coaction 
given by the coproduct $\Delta$
(and cleaving map $j=\id_H: H\to H$). 

Then, let $\cot$ be a
2-cocycle on $H$. 
The commutative Hopf algebra $(H, R=\varepsilon\ot \varepsilon)$ is twist deformed to the cotriangular
Hopf algebra $(H_\cot , R_\cot=\cot_{21} *\bar\cot)$.  The total space algebra
$(A=H,\cdot,\Delta)\in \qcr$ is deformed as an $H$-comodule algebra to
$(A_\cot=H_{\bullet_\cot}, \bullet_\cot, \Delta)\in \qcrc$, and we obtain the Hopf--Galois extension
$\bbK\subseteq H_{\bullet_\cot}\in \qcrc$. Notice that this is a cleft
extension with cleaving map $j=\id_H: H_\cot\to H_{\bullet_\cot}$, but
in general needs not be a trivial extension because $ H_\cot$ and
$H_{\bullet_\cot}$ are in general not isomorphic as $H_\cot$-comodule
algebras (see next examples). The gauge group of this
noncommutative cleft extension $\bbK\subseteq H_{\bullet_\cot} $ is isomorphic to the gauge group of the trivial extension $\bbK\subseteq H$.
\qed\end{ex}
\begin{ex}\label{ex:torotwisted} {\it Noncommutative torus bundle over a point.} 
Let us consider the Hopf--Galois extension
$ \IC = \O (\mathbb{T}^n_\theta)^{co\O (\mathbb{T}^n)} \subseteq  \O (\mathbb{T}^n_\theta)$ 
with total space the noncommutative torus $\mathbb{T}^n_\theta$ with generators 
$t_j, \, t^*_j$ satisfying $t_it^*_i=t_i^*t_i=1$, $t_j t_k = e^{i \pi
  \theta_{jk}}t_k t_j$ and $t_j t_k^* = e^{i \pi \theta_{kj}}t_k^*
t_j$, for $\theta_{jk}=-\theta_{kj} \in \IR$, and with structure
group the Hopf algebra $\O (\mathbb{T}^n)$. 

As in Example \ref{ex:toroc}, the $\O (\mathbb{T}^n)$-comodule
map and the algebra map properties of a gauge
transformation $\F: \O (\mathbb{T}^n_\theta)\to \O
(\mathbb{T}^n_\theta)$ imply that the latter is determined by 
$t_j \mapsto F(t_j)=\lambda_j t_j$ and $t^*_j \mapsto
F(t^*_j)=\lambda^*_j t^*_j$, with complex numbers of modulus one, $| \lambda_j |^2 =1$. This shows
that, independently from the noncommutativity of the generators, the gauge 
transformations are parametrized by $\lambda_j \in S^1$.  
Hence the gauge group is isomorphic to the $n$-dimensional torus $\mathbb{T}^n$,
the same of the commutative Hopf--Galois extension
$ \IC = \O (\mathbb{T}^n)^{co\O (\mathbb{T}^n)} \subseteq  \O
(\mathbb{T}^n)$. 

This result is consistent with the use of Proposition \ref{prop:gamma-aut} for the Hopf--Galois extension $ \IC  \subseteq  \O (\mathbb{T}^n_\theta)$ seen
as a  twist deformation of  $\IC \subseteq \O (\mathbb{T}^n)$.
The 2-cocycle  $\gamma$ on $\O (\mathbb{T}^n)$ is determined by its value on the generators,
\beq
\label{427cocycleTn}
\co{t_j}{t_k}= \exp(i \pi \,\theta_{jk}) ~~,~~\theta_{jk}=-\theta_{kj} \in \IR
\eeq
and  defined on the whole algebra by requiring 
$\co{xy}{z}=\co{x}{\one{z}}\co{y}{\two{z}}$ and 
 $\co{x}{yz}=\co{\one{x}}{z}\co{\two{x}}{y}$, for all $x,y,z, \in
 \mathcal{O} (\mathbb{T}^n)$. Being the Hopf algebra $H=\O (\mathbb{T}^n)$ cocommutative, one now obtains $H_\cot=H$.
 \qed\end{ex}
\begin{rem}
When comparing the result of the previous example with Remark \ref{homotopy}, we see that non isomorphic Hopf--Galois extensions 
$\IC  \subseteq  \O (\mathbb{T}^n_\theta)$ (obtained from non co-homologous twists) have isomorphic gauge groups. 
 This is a general feature occuring when starting with a cocommutative Hopf algebra $H$, which is hence transparent to the twist 
so that $H_\cot=H$ (cf. equation \eqref{hopf-twist}).
 \end{rem}  

\begin{ex} \label{ex:sotwisted} {\it Noncommutative $SO(2n)$-bundle over a point.} We specialise Example \ref{ex:Hcleft} to 
$H=\mathcal{O}(SO(2n, \IR))$, 
the algebra of coordinate functions on $SO(2n,
\IR)$.
Let $\mathcal{O}(M(2n, \IR))$ be the  commutative
$*$-algebra over $\mathbb{C}$ with generators 
$a_{ij}, b_{ij}, a_{ij}^*, b_{ij}^*$, $i,j=1, \dots n$.     It is a bialgebra with coproduct and counit  given in matrix notation as
$$
\Delta(M)=M \overset{.}{\otimes} M \quad , \quad \varepsilon(M)=\mathbb{I}_{2n}, \quad \mbox{ for  }\quad M=(M_{IJ}):=\begin{pmatrix}
(a_{ij}) & (  b_{ij})
\\
(b_{ij}^*) & (a_{ij}^*)
\end{pmatrix} \, . 
$$
Here $\overset{.}{\otimes}$ is the combination of tensor product and matrix multiplication, $\mathbb{I}_{2n}$ is the identity matrix and capital indices $I,J$ run from $1$ to $2n$.
The Hopf coordinate algebra $SO(2n, \IR)$ is the quotient 
$\mathcal{O}(SO(2n, \IR))= \mathcal{O}(M(2n, \IR)){/{I_Q}} $
with $I_Q$ the bialgebra ideal defined by   
$$
I_Q= \langle\, M^t Q M -Q \; ; \; M Q M^t-Q\; ; \;
\det(M)-1
\,\rangle \; 
, \quad Q:=\begin{pmatrix}
0 &\mathbb{I}_n 
\\
 \mathbb{I}_n & 0
\end{pmatrix} =Q^t=Q^{-1}~.
$$
The $*$-structure in $\mathcal{O}(M(2n, \IR))$ is 
$M^*=QMQ^{-1}$ so that  $I_Q$ is a $*$-ideal.  The  $*$-bialgebra $\mathcal{O}(SO(2n, \IR))$ is a $*$-Hopf algebra with antipode 
$S(M):= Q M^tQ^{-1}$.  

The algebra $A:=\mathcal{O}(SO(2n, \IR))$ is an $\mathcal{O}(SO(2n, \IR))$-comodule algebra with coaction the coproduct $\Delta$. The corresponding Hopf--Galois extension $\IC \subseteq \mathcal{O}(SO(2n, \IR))$ is trivial and has gauge group  
\beq\label{gaso}
\G_A\simeq  \left( \{ \alpha: \mathcal{O}(SO(2n, \IR)) \ra \IC \mbox{ algebra maps} \} , *\right) ,
\eeq
the set of characters of $\mathcal{O}(SO(2n, \IR))$ with  group multiplication the convolution product.

Next we consider a 2-cocycle $\gamma$ on  a maximal torus  in $\mathcal{O}(SO(2n, \IR))$.
Let $\mathcal{O} (\mathbb{T}^n)$ be  the  commutative $*$-Hopf algebra of functions on the $n$-torus as considered in Example~\ref{ex:toroc}.
It is a quotient Hopf algebra, a ``subgroup'' of $\mathcal{O}(SO(2n, \IR))$,
with projection
$$\pi:
  M \mapsto \mathrm{diag}(T_I) := \mathrm{diag}(t_1, \dots t_n , t_1^* , \dots t_n^*) .
$$ 
The 2-cocycle $\gamma$ on $\mathcal{O} (\mathbb{T}^n)$ given in \eqref{427cocycleTn} lifts
by pullback to a 2-cocycle on
$\mathcal{O}(SO(2n, \IR)) $, that we
still denote by $\gamma$,
$$\co{M_{IJ}}{M_{KL}}:=\co{\pi(M_{IJ})}{\pi(M_{KL})}=\co{T_{I}}{T_{K}} \delta_{IJ} \delta_{KL} .$$ 
The twisted  Hopf algebra $H_\cot=\mathcal{O}(SO(2n, \IR))_\cot$ has product \eqref{hopf-twist}:   
$$
M_{IJ} \mt M_{KL}= \co{T_I}{T_K} M_{IJ} M_{KL}\coin{T_J}{T_L} .
$$
Since $\co{T_I}{T_K} 
= \coin{T_K}{T_I}$,  
the generators in $\mathcal{O}(SO(2n, \IR))_\cot$ obey relations
$$
M_{IJ} \mt M_{KL}= \big(\co{T_I}{T_K} \coin{T_J}{T_L} \big)^2 M_{KL} \mt M_{IJ}.
$$
Explicitly, for
$\lambda_{ij}= \exp(2i\pi \theta_{ij})$, these
read 
\begin{eqnarray*}
a_{ij} \mt a_{kl} = \lambda_{ik}\lambda_{lj} ~a_{kl}\mt a_{ij} &,  & 
a_{ij} \mt b^*_{kl} = \lambda_{ki}\lambda_{lj} ~b^*_{kl}\mt a_{ij} \; , \nn
\\
a_{ij} \mt b_{kl} = \lambda_{ik}\lambda_{jl} ~b_{kl}\mt a_{ij} &,& 
a_{ij} \mt a^*_{kl} = \lambda_{ki}\lambda_{jl} ~a^*_{kl}\mt a_{ij}  \; , \nn
\\
b_{ij} \mt b_{kl} = \lambda_{ik}\lambda_{lj} ~b_{kl}\mt b_{ij} &,& 
b_{ij} \mt b^*_{kl} = \lambda_{ki}\lambda_{jl} ~b^*_{kl}\mt b_{ij} \; ,
\end{eqnarray*}
together with their $*$-conjugated.  
Moreover, 
$$
M^t \mt Q \mt M =Q~,~~
 M \mt Q \mt M^t=Q~,~~
{\det}_\cot(M)=1
$$
with quantum determinant 
$$
{\det}_\cot(M)=\sum_{\sigma \in {\mathcal P}_{2n}}(-1)^{|\sigma|\,}\big(\!\prod_{\mbox{${}^{~\;I<J}_{\,\sigma_I>\sigma_J}$}} \!\lambda_{\sigma_I\sigma_J}\big)\,
M_{1\sigma_1} \!\cdot_\gamma\ldots  \cdot_\gamma
M_{2n\,\sigma_{2n}} \; .
$$
The twisted  Hopf algebras $\mathcal{O}(SO(n, \IR))_\cot$ were studied in
\cite{resh, schir,cdv} (see also \cite[\S 4.1]{ppca}).

The comodule algebra $(A=\mathcal{O}(SO(2n, \IR)), \Delta)$ is deformed to 
a comodule algebra $(A_\cot=\mathcal{O}({SO}(2n, \IR))_{\mtco}, \Delta)$ with product
\eqref{rmod-twist}. On the generators one has
$$
{M}_{IJ} \mtco {M}_{KL}=  {M}_{IJ} {M}_{KL}\coin{T_J}{T_L} , \quad
$$
and hence
they have commutation relations
$$
{M}_{IJ} \mtco {M}_{KL}=  \left(\coin{T_J}{T_L} \right)^2 {M}_{KL} \mtco {M}_{IJ}.
$$
Explicitly, 
\begin{eqnarray*}
{a}_{ij} \mtco {a}_{kl} = \lambda_{lj} ~{a}_{kl}\mtco {a}_{ij} &,  & 
{a}_{ij} \mtco {b}^*_{kl} = \lambda_{lj} ~{b}^*_{kl}\mtco {a}_{ij}  
\\
{a}_{ij} \mtco {b}_{kl} = \lambda_{jl} ~{b}_{kl}\mtco {a}_{ij} &,& 
{a}_{ij} \mtco {a}^*_{kl} = \lambda_{jl} ~{a}^*_{kl}\mtco {a}_{ij}  
\\
{b}_{ij} \mtco {b}_{kl} = \lambda_{lj} ~{b}_{kl}\mtco {b}_{ij} &,& 
{b}_{ij} \mtco {b}^*_{kl} = \lambda_{jl} ~{b}^*_{kl}\mtco {b}_{ij} 
\end{eqnarray*}
with their $*$-conjugated.   
The $\mathcal{O}(SO(2n, \IR))_\cot$-Hopf--Galois extension $\IC \subseteq \mathcal{O}(SO(2n, \IR))_{\mtco}$ is cleft (but no longer trivial) and has gauge group $\G_{A_\cot}$ isomorphic to $\G_{A}$ in \eqref{gaso}. 
\qed
\end{ex}
\begin{rem}\label{rem:sotwisted}
We stress that the gauge group $\G_{A_\cot}$ is not  the group of characters of 
the braided Hopf algebra $\underline{\hg}$ associated with the Hopf algebra $H_\cot=\mathcal{O}(SO(2n, \IR))_\cot$ (see comment after Lemma \ref{lem:ban}). Indeed $\underline{\hg}\simeq \bh_\cot$ is genuine noncommutative.  
The generators of the algebra  $\bh_\cot$ have product \eqref{prod-hgt0}:
$$
{M}_{IJ} \hpr_\cot {M}_{KL}= {M}_{IJ} {M}_{KL}~ \coin{S(T_I)T_J}{S(T_K)T_L} , \quad
$$
where we used that the product $\hpr$ coincides with that in $H=\mathcal{O}(SO(2n, \IR))$ since $R$ is trivial. 
By using the properties of the abelian cocycle $\gamma$ this product leads to commutation relations
$$
{M}_{IJ} \hpr_\cot {M}_{KL}= \big(\coin{T_I}{T_K} \co{T_J}{T_K} \co{T_I}{T_L} \coin{T_J}{T_L} \big)^2 {M}_{KL} \hpr_\cot {M}_{IJ} \, .
$$
We see that in general the algebra $\underline{\hg}\simeq \bh_\cot$ is noncommutative, with less characters than the commutative algebra $H=\mathcal{O}(SO(n, \IR))$.
\end{rem}
 
 \begin{ex}  {\it Noncommutative principal bundles over affine varieties.} 
For a principal $G$-bundle, $\pi: P \ra P/G$, with $G$  a
semisimple affine algebraic group and $P$, $P/G$ affine varieties, as  
in Example \ref{affinecase} we consider the
  $\O(G)$-Hopf-Galois extension $\O(P/G)\subseteq \O(P)$. 
Given a 2-cocycle $\gamma$ on $\O(G)$, the gauge group of the twisted $\O(G)_\cot$-Hopf-Galois extension $\O(P/G)\subseteq
\O(P)_\cot$, is isomorphic to the gauge group of $\pi: P \ra P/G$.
\qed
\end{ex}

 \subsection{Tensoring Hopf--Galois extensions}~\\[.5em]
The fiber product of a $G$-principal bundle $P\to M$ with a $G'$-principal
bundle $P'\to M$ gives a $G\times G'$-principal
bundle $P\times_{_{\!M\!}} P'\to M$. The corresponding gauge group is the product
of the initial gauge groups. In view of the next examples, we
consider the analogue of this fiber product construction for Hopf--Galois extensions. 
Let $H$ and $K$ be Hopf algebras and consider an $H$-Hopf--Galois extension $B =A^{coH}\subseteq A$
and a $K$-Hopf--Galois extension $B ={A'}^{\, co K}\subseteq A'$ of an algebra $B$. 
Assuming that $B$ is in the center of both $A$ and $A'$, the balanced tensor product 
$A \ot_B A'$ inherits an algebra structure from the tensor product algebra $A \ot A'$.
It  is a  comodule algebra for  the Hopf algebra $H \ot K$ (with usual tensor product algebra and coalgebra structures) with coaction 
$$\delta^{A \ot_B A'} : A \ot_B \!A' \,\ra A \ot_B\! A' \ot H \ot K
\; , \quad
a \ot_B a' \mapsto \zero{a} \ot_B \zero{a'} \ot \one{a} \ot \one{a'} \, .
$$
Then $B=(A \ot_{B\!} A')^{co (H\ot K)}\subseteq A \ot_B A'$ is an $(H\ot K)$-Hopf--Galois
extension; indeed the corresponding canonical map has inverse
$$ {\chi^{-1}}_{|_{1 \ot 1 \ot H \ot K}}: h \ot k \mapsto \tuno{h} \ot_B
\tuno{k} \ot_B  \tdue{h} \ot_B \tdue{k} \in (A \ot_B \! A') \ot_B (A \ot_B\! A')~.
$$ 
Here $\tau_H(h)=\tuno{h} \ot_B   \tdue{h}$ and $\tau_K(k)=\tuno{k}
\ot_B   \tdue{k}$  denote the translation maps of the Hopf--Galois
extensions  $B =A^{coH}\subseteq A$ and $B=A'^{\, co K}\subseteq A'$.

If $(H,R)$ and $(K, R')$ are coquasitriangular
then $(H\ot K, R'')$ is  coquasitriangular with $R''=(R\ot R')\circ
(\id\ot\mbox{flip}\ot\id)$, where $\mbox{flip} : K \ot H \to H\ot K$, $ k\ot h \mapsto h\ot k$.
Moreover, if $A\in \qcr$ and
$A'\in \qcrp$, then $A\ot_B\! A'\in\qcrpp$. 

 \begin{prop}\label{prop:iso}
Let $(H,R)$ and $(K,R')$ be coquasitriangular Hopf algebras. 
Let $B =A^{coH}\subseteq A\in \qcr$ and $B=A'^{\, co K}\subseteq A'\in
\qcrp$ be Hopf--Galois extensions, with the additional assumption that  
$A$ is flat as a right $B$-module and $A'$ is flat as a left $B$-module.
Then the gauge group of the tensor product  $(H \ot K)$-Hopf--Galois
extension  $B= (A \ot_{B\!} A')^{co (H\ot K)}\subseteq A \ot_B A'$ is
isomorphic to the direct product of the gauge groups of the Hopf--Galois extensions $B=A^{coH}\subseteq A$ and $B=A'^{\,
  co K}\subseteq A'$:
\beq\label{iso}
\G_{A \ot_B A'} \,\simeq \;\G_A \times \G_{A'}~.
\eeq
\end{prop}
\begin{proof}
We consider gauge transformations as vertical automorphisms and show there is a group isomorphism
$$
\Aut{A\otimes_B A'}  \simeq  \Aut{A}\times \Aut{A'} 
$$ 
implemented by the map 
$$
\textsc{F} \mapsto 
\left\{ 
\begin{array}{l}
\F^{\phantom{\prime}}_{\textsc{F}}: a \mapsto \F^{\phantom{\prime}}_{\textsc{F}}(a):=\textsc{F}(a \ot_B 1_{A'}) 
\\[.4em]
\F'_{\textsc{F}} : a' \mapsto \F'_{\textsc{F}} (a'):=\textsc{F}(1_A \ot_B a') 
\end{array} \right. \, .
$$
In order to show that the image of the algebra map $\F^{\phantom{\prime}}_{\textsc{F}}$ is in 
$A \simeq  A\otimes_B B$ we first observe that the short exact
sequence defining $A'^{co K}$,
$$
0\longrightarrow B=A'^{co K} \xrightarrow{i} A' \xrightarrow{\delta^{A'} \:\! -\;\id_{A'} \ot\;\! \eta_K\, } 
\mbox{Im} (\delta^{A'} - \id_{A'} \ot \eta_K) \subset A'\ot K \longrightarrow 0~,
$$ 
and  $B$-flatness of $A$ imply the exactness of the sequence
\begin{multline*}
0\longrightarrow A\simeq A\ot_{B\!} B
\xrightarrow{\id_A\otimes_B\, i} 
A\ot_B A' \xrightarrow{\id_A\otimes_B(\delta^{A'}\:\!-\;\id_{A'}\ot \,\eta_K \,)} 
\\ \to \mbox{Im}(\id_A\otimes_B(\delta^{A'}\:\!-\;\id_{A'}\ot \,\eta_K \,)) \subset
A \ot_B A' \ot K \longrightarrow 0~.
\end{multline*}
Then $\textsc{F}(a \ot_B 1_{A'})\in A\ot_B B \simeq A$ follows by
showing that $\textsc{F}(a \ot_B 1_{A'})$ is in the kernel of
$\id_A\otimes_B(\delta^{A'}\:\!-\;\id_{A'}\ot \,\eta_K)$, an easy consequence of the $(H\ot K)$-equivariance
of $\textsc{F}$ and the identity $(\id_{A\ot_{B \:\!} A'}\ot \varepsilon_H\ot
\id_K)\delta^{A\ot_{B\!\!}\; A'}=\id_A\ot_B\delta^{A'}$. Moreover,
 $H$-equivariance of $\F^{\phantom{\prime}}_{\textsc{F}}$ 
 follows from the identity $(\id_{A\ot_{B \:\!} A'}\ot \id_H\ot
 \varepsilon_K)
\delta^{A\ot_{B\!\!}\; A'}=(\id_A\ot_B\mbox{flip})\circ (\delta^{A}\ot_B\id_{A'})$.

Similarly one finds that the algebra map $\F'_{\textsc{F}}$ is a $K$-equivariant
map $A'\to B \otimes_B {A'}\simeq A'$.

The map ${\textsc{F}} \mapsto
(\F^{\phantom{\prime}}_{\textsc{F}},\F'_{\textsc{F}})$ is an
isomorphism with inverse $(\F^{\phantom{\prime}},\F')\mapsto {\textsc{F}}_{
(\F^{\phantom{\prime}},\F')}:=\F\otimes_B\F'$.
This is a left inverse:
$ {\textsc{F}}\to (\F^{\phantom{\prime}}_{\textsc{F}},\F'
_{\textsc{F}})\to \F^{\phantom{\prime}}_{\textsc{F}}\ot_B\F'
_{\textsc{F}}={\textsc{F}}$ because $\F^{\phantom{\prime}}_{\textsc{F}}(a) \ot_B\F' _{\textsc{F}}(a') = (\F^{\phantom{\prime}}_{\textsc{F}}(a) \ot_B 1)(1 \ot_B \F' _{\textsc{F}}(a'))=
{\textsc{F}}(a\ot_B 1_{A'}) {\textsc{F}}(1_{A}\ot_B a')={\textsc{F}}(a\ot_B a')$. 
It is easily seen to be a right inverse as well.
$(\F^{\phantom{\prime}},\F')\mapsto \F^{\phantom{\prime}}\ot_B \F' \mapsto
 (\F^{\phantom{\prime}}_{\F^{\phantom{\prime}}\ot_B\F'},\F' _{\F^{\phantom{\prime}}\ot_B\F'})=
(\F^{\phantom{\prime}},\F')$.
\end{proof}

\begin{ex}\label{ex:7hge}
Let $H= \O(SU(2))$ be the Hopf coordinate algebra of the
affine algebraic group $SU(2)$ and let $\O(S^n)$ denote the coordinate algebra
of the classical sphere $S^n$.
It is well-known that $A= \O(S^7)$  is a Hopf--Galois extension of   $B=\O(S^4)$ for the coaction of $H$ dual to the classical principal action of $SU(2)$
on $S^7$ defining the $SU(2)$-Hopf bundle over $S^4$. The module $A=
\O(S^7)$ is flat over $B=\O(S^4)$ (it is actually faithfully flat). Let  $K=\O(U(1))$ be the Hopf algebra of coordinate functions on
$U(1)$ and consider the trivial Hopf--Galois extension $B\subseteq B\ot
K$. Again $B\ot K$ is flat as a $B$-module.
 In accordance with the theory above, we consider the Hopf--Galois extension for the Hopf algebra 
$ \O(SU(2)) \ot  \O(U(1))\simeq  \O(SU(2) \times  U(1))$ with comodule
algebra $\O(S^7)\ot_B B\ot  \O(S^1) \simeq \O(S^7)\ot  \O(S^1)\simeq \O(S^7\times  S^1) $
(with obvious coaction of $\O(U(1))$ on $\O(S^1)$) and subalgebra of
coinvariant elements again $B=\O(S^4)$. 
 
By Proposition \ref{prop:iso}, the commutative Hopf Galois extension 
$\O(S^4)\subseteq \O(S^7)\ot\O(S^1)$ has gauge group ${\G}_{\O(S^7)}\times {\G}_{\O(S^4)\otimes \O(U(1))}$.
Here ${\G}_{\O(S^7)}$ is the gauge group of the 
$SU(2)$-principal bundle $S^7\to S^4$, while
${\G}_{\O(S^4)\otimes \O(U(1))}=\{\f: \O(U(1)) \ra \O(S^4) \mbox{ algebra maps} \}$ is the gauge group of the trivial 
$U(1)$-principal bundle over $S^4$.

We have been working with affine varieties {(noncommutative polynomial algebras of coordinates given by generators and relations)}.
Completion in the smooth category gives the gauge group of the principal  $SU(2) \times U(1)$-bundle
$S^7\times S^1\to S^4$. Being the latter the product of the
gauge group of the $SU(2)$-Hopf bundle  and the group of
$U(1)$ valued functions on $S^4$, the
smooth completion of the gauge group of the commutative Hopf--Galois extension
$\O(S^4)\subseteq \O(S^7)\ot \O(S^1)$ is this classical gauge group.
\qed\end{ex}

\begin{ex}
We present a twisted version of the previous example. To this aim we
first twist the Hopf $*$-algebra  $\O(SU(2)) \ot  \O(U(1))$. 
The commutative Hopf $*$-algebra  $\O(SU(2))$  of coordinate functions on $SU(2)$ is 
generated by the entries of the matrix $\tilde{u} = \begin{pmatrix} a & b \\ c & d\end{pmatrix}$, with $c=-b^*$,  $d=a^*$ and
$aa^*+bb^*=1$. The commutative Hopf $*$-algebra $\O(U(1))$ of coordinate functions on $U(1)$
has generators  $w, {w^*}$,  with  $w{w^*}=w^* w=1$.
Consider then the tensor product Hopf $*$-algebra $\O(SU(2))\ot \O(U(1))$ and the Hopf $*$-ideal $I=\langle b \ot 1, c \ot 1\rangle$. 
The quotient Hopf algebra $\O(SU(2))\ot \O(U(1))/I$ is easily seen to be isomorphic to the Hopf $*$-algebra 
$\O (\mathbb{T}^2)$  of functions on the $2$-torus:
$$
\O(SU(2))\ot \O(U(1))/I \simeq \O (\mathbb{T}^2) \; , \quad (a \ot 1) \mapsto t_1 \; , \quad (1 \ot w) \mapsto t_2 \;.
$$
The 2-cocycle $\gamma$ on $\O (\mathbb{T}^2)$ is the one given in \eqref{427cocycleTn}; we choose the 
convention
\beq
\co{t_1}{t_2} = \co{t_2}{t_1}^{-1} = \exp( i \pi \tfrac{1}{2} \theta) \; .
\eeq
Via the quotient map $\pi:  \O(SU(2))\ot \O(U(1)) \ra \O(SU(2))\ot \O(U(1))/I \simeq \O (\mathbb{T}^2)$ it lifts to a $2$-cocycle $\gamma$ on $\O(SU(2))\ot \O(U(1))$ defined,
for all $\ell,\ell' \in \O(SU(2))\ot \O(U(1))$, by (see \cite[Lem.4.1]{ppca})
$$
\co{\ell}{\ell'} := \co{\pi(\ell)}{\pi(\ell')} .
$$ 
With this 2-cocycle we deform the Hopf algebra $\O(SU(2))\ot \O(U(1))$
to the new Hopf algebra $\left( \O(SU(2))\ot \O(U(1))\right)_\gamma$
with product given in \eqref{hopf-twist}. Working out the commutation
relations, one finds this Hopf algebra to be generated by the elements
\beq\label{gen}
a:=a \ot 1 \; ,  \quad b:=b \ot 1 \; ,   \quad w:= 1 \ot w \; , \;   \mbox{(together with } a^* , ~ b^* , ~{w^*})
\eeq
where $a$  is central   (and so is $a^*$), modulo the further
commutation relations 
\beq \label{crA}
b \mt b^* = b^* \mt b \; , \quad w \mt {w^*} = {w^*} \mt w \; , \quad
b \mt  w= q  ~ w \mt b \; , \quad b  \mt  {w^*} = \bar{q}  ~  {w^*} \mt b \; ,
\eeq 
with $q:= \exp( 2 \pi i\,\theta)$, and
modulo the relations
\beq\label{relA}
a \mt a^* + b \mt b^* =1 \quad, \quad
w \mt  {w^*}= 1 ~.
\eeq
Counit and coproduct are undeformed and given by
$\varepsilon(a)=1,~\varepsilon(b)=0,~\varepsilon(w)=1$ and 
\beq\label{cop0}
\Delta(a)=a \ot a- b \ot b^* \; , \quad \Delta(b)=a \ot b+ b \ot a^* \; , \quad \Delta(w)=w \ot w \;.
\eeq
The antipode, $S_\cot:= u_\cot *S*\bar{u}_\cot$, where
$u_\cot$ and $\bar{u}_\cot$ are in \eqref{uxS}, on the generators reads
\beq\label{anti0}
S_\cot(a)= S(a)=a^*, \qquad S_\cot(b)= S(b)=-b,  \qquad S_\cot(w)= S(w)={w^*}.
\eeq

We next deform the Hopf--Galois extension of Example \ref{ex:7hge}
 to a new Hopf--Galois extension over $B=\O(S^4)$ with Hopf algebra just
 $\big(\O(SU(2)) \ot  \O(U(1))\big)_\gamma$. 
The total space is the noncommutative algebra $ \big(\O(S^7)\ot
\O(S^1)\big)_\gamma$, deformation of the algebra $\O(S^7)\ot  \O(S^1)$
as in \eqref{rmod-twist}. Explicitly, since $\O(S^7)\ot  \O(S^1)$ is
the commutative $*$-algebra
generated by  $z_j, z^*_j$, for $j=1,2,3,4$, and $w, w^*$, then  $\big(\O(S^7)\ot  \O(S^1)\big)_\gamma$ is the noncommutative $*$-algebra  
generated by 
$z_j, z^*_j$, and $w, w^*$, modulo the commutation relations
\beq\label{crS}
z_j \mtco z_k= z_k \mtco z_j 
\; ; \quad
z_j \mtco w = \exp(-i \pi \theta)~ w \mtco z_j
\; ; \quad
z^*_j \mtco w = \exp(i \pi \theta)~ w \mtco z^*_j \; , 
\eeq
for each $j,k=1,\dots, 4$ (and their $*$-conjugates), and the relations
\beq\label{relS}
\sum_{j=1}^4 z^*_j \mtco z_j= 1 \quad ; \quad 
{w}^* \mtco w=1 ~.
\eeq
Thus the subalgebras  $\O(S^7)$ and $\O(S^1)$ remain
commutative, and the noncommutativity can therefore be ascribed
just to the tensor product in $\big(\O(S^7)\ot \O(S^1)\big)_\gamma
\simeq \O(S^7\times S^1)_\gamma$,
suggesting the notation $\O(S^7)\ot_\gamma\O(S^1)$. 
Indeed this algebra is of the same type as those 
obtained in \cite{dvl17} as noncommutative products of spheres, there denoted $\O(S^7\times_\gamma S^1)$.

By Proposition \ref{thm:gamma-ad} we conclude that the gauge group of
the Hopf--Galois extension $\O(S^4)\subseteq \big(\O(S^7)\ot
\O(S^1)\big)_\gamma$ is isomorphic to that of the commutative
Hopf--Galois extension $\O(S^4)\subseteq \O(S^7)\ot\O(S^1)$ described in
the previous example (that is, the gauge group of the $SU(2)$-Hopf bundle times the group of $U(1)$ valued functions on $S^4$).
\qed\end{ex}

\begin{ex}\label{ex-lungo}
{\it Reduction of the ``structure group''.}
Given any odd dimensional sphere $S^{2n-1}$, the cartesian product $S^{2n-1}\times S^1$ carries a diagonal action of $\IZ_2$, with the non trivial generator of the latter sending a point on a sphere to its antipodal point. 
The quotient $(S^{2n-1} \times S^1) / \IZ_2$ is a copy of 
$S^{2n-1} \times S^1$ : if $z_j, z^*_j $, $j=1, \dots, n$, are coordinates on $S^{2n-1}$ with 
$\sum_{j=1}^n z^*_j z_j= 1$ and $w$ is the coordinate of $S^1$ with $w^* w=1$,  coordinates for the quotient are given by 
$x_j =z_j  w$ and $y = (w^*)^2$. 

On the other hand, with the group structures of $S^3 = SU(2)$ and $S^1 = U(1)$, the quotient group $(SU(2)\times U(1)) / \IZ_2$ is isomorphic to the group $U(2)$. 
Consider   the principal $SU(2) \times
S^1$ bundle $S^7 \times S^1\to S^4$. The subgroup $\IZ_2$ of
$SU(2)\times U(1)$ acts on $S^{7}\times S^1$ as above by flipping
antipodal points; then the quotient
leads to an $(SU(2)\times U(1)) / \IZ_2 \simeq U(2)$
bundle with total space $(S^7\times S^1) / \IZ_2 \simeq S^7\times S^1$ and base space still $S^4$.
\qed\end{ex}

We next present an Hopf--Galois description of
this construction that also applies to the noncommutative Hopf--Galois
extension  $\O(S^4)\subseteq \big(\O(S^7)\ot  \O(S^1)\big)_\gamma$ and
leads to an $\O(U_q(2))$-Hopf--Galois extension, with corresponding gauge group.

\begin{ex}{\it Twisting Example \ref{ex-lungo}}. 
The $*$-algebras $\big(\O(SU(2))\otimes \O(U(1))\big)_\gamma$ and $\big(\O(S^7)\ot
\O(S^1)\big)_\gamma$ are both $\IZ_2$-graded with the generators that are
 odd, while the commutation relations, determinant and radius relations
\eqref{crA}-\eqref{relA} and \eqref{crS}-\eqref{relS} are even; the $*$-involutions are grade preserving.

Firstly, consider the $\IZ_2$-invariant $*$-algebra $\big(\O(S^7)\ot \O(S^1)\big)_\gamma ^{\IZ_2}\subseteq \big(\O(S^7)\ot
\O(S^1)\big)_\gamma $. It is easy seen to be generated by the elements $x_j:=z_j \mtco w$ and
$y:={w^*}\mtco {w^*}$, and their $*$-conjugates $x_j^*$ and $y^*$, modulo 
commutation relations induced by \eqref{crS}:
$$
x_j \mtco x_k\,=\, x_k\mtco x_j~,~~ x_j \mtco y\,=\,q \:y \mtco x_j~,~~ x_j^* \mtco y\,=\,{q^{-1}}\: y \mtco x^*_j~,
$$ 
(with their $*$-conjugates) and radius  relations $\sum_j x_j^* \mtco x_j=1\,,\; y \mtco  y^*=1$ induced by \eqref{relS}.
 Here $q=exp(2\pi i\,\theta)$. We see that $\big(\O(S^7)\ot
\O(S^1)\big)_\gamma ^{\IZ_2}$ is again a deformation of the algebra of
coordinates of the affine variety $S^7\times S^1$ as in \eqref{crS}, but with different
noncommutativity parameter $q=exp(2\pi i\,\theta)$ rather than $exp(-i\pi \,\theta)$.   

Next, the even $*$-subalgebra $\big(\O(SU(2))\otimes \O(U(1))\big) ^{\IZ_2}_\gamma$ 
is also a Hopf $*$-subalgebra of $\big(\O(SU(2))\otimes \O(U(1))\big)_\gamma$ since 
the coproduct and the antipode are grade preserving and hence restrict to the even subalgebra.
This Hopf $*$-subalgebra is easily seen to be the algebra generated by the elements
\beq\label{gen-inv}
\alpha:=a \mt w \; , 
~~\beta:=b \mt w \; ,~~
D:={w}\mt w\; , ~
\eeq
and their $*$-conjugated elements $\alpha^*={w^*} \mt a^*\,,~ \beta^*={w^*}\mt b^* \, , ~D^{*}={w}^{*}\mt {w}^{*}$,
modulo the commutation relations induced by \eqref{crA}, which are worked out to be
\begin{align}\label{rel-inv}
& \alpha \mt \beta = q^{-1} \beta \mt \alpha \, , \quad \alpha \mt \beta^* = q \beta^* \mt \alpha \, , \nn \\
& D \mt \alpha = \alpha \mt D \, , \quad D \mt \alpha^* = \alpha^* \mt D \, ,  \nn \\
& D\mt \beta = q^{-2} \beta \mt D \, , \quad D\mt \beta^* = q^{2} \beta^* \mt D  
\end{align}
(together with their $*$-conjugates) and the relations
\beq \label{id-inv}
\alpha\mt \alpha^*+\beta\mt\beta^*=1\;,\quad D\mt D^{*}=1~.
\eeq
When restricting the coproduct in \eqref{cop0} to the
subalgebra one obtains: 
\beq\label{copz2}
\Delta(\alpha)=\alpha \ot \alpha - q \beta \ot (\beta^* \mt D)\; , \quad \Delta(\beta)=\alpha \ot \beta + \beta \ot (\alpha^* \mt D) \; , \quad \Delta(D)=D \ot D \;.
\eeq
Similarly, restricting the counit one has $\varepsilon(\alpha)=1,
\varepsilon(\beta)=0, \varepsilon(D)=1$, and finally for the antipode
\eqref{anti0} one finds
\beq\label{antiz2}
S(\alpha)=\alpha^*,  \quad S(\beta) = -q (\beta \mt D^*),  \quad S(D)={D^*}.
\eeq 

Paralleling the classical result for the group $(SU(2) \times U(1)) /
\IZ_2 \simeq U(2)$, the Hopf $*$-algebra $\big(\O(SU(2))\otimes \O(U(1))\big) ^{\IZ_2}_\gamma$ (with the structures in \eqref{copz2} and \eqref{antiz2}) is isomorphic to the cotriangular quantum group $\O(U_{q}(2))$
(the multiparametric quantum group $U_{q,r}(2)$ with $r=1$, see e.g. \cite{schir, resh}). 
We recall that $\O(U_q(2))$ is generated by the matrix entries,  
 $
 u:= \begin{pmatrix} \alpha & \beta \\ \gamma & \delta\end{pmatrix},
 $
and by the inverse $D^{-1}$ of the quantum determinant $D=\alpha \delta-q^{-1} \beta\gamma$. These generators 
satisfy the FRT commutation relations  
 $\R^{ji}_{kl} u_{km}u_{ln}= u_{ik} u_{jl} \R^{lk}_{mn}$, 
with matrix $\R=\mathrm{diag}(1,q^{-1} ,q,1)$.
Explicitly, we have
\begin{eqnarray}
\alpha \beta = q^{-1} \beta \alpha \,\, ; \quad 
\alpha \gamma = q \gamma \alpha \,\, ; \quad  
\beta \delta = q \delta \beta \,\, ; \quad 
\gamma \delta = q^{-1} \delta \gamma \,\, ; \quad 
\beta \gamma = q^2 \gamma \beta \,\, ; \quad 
\alpha \delta = \delta \alpha \, \nonumber\\[.2em]
\alpha {D}^{-1}= {D}^{-1} \alpha \; ; \quad \beta {D}^{-1} = q^{-2} {D}^{-1} \beta
\; ; \quad \gamma {D}^{-1} = {q}^2 {D}^{-1} \gamma
\; ; \quad \delta {D}^{-1} =  {D}^{-1} \delta \, .~~~\!\;\!~~~~~~~\label{FRTU2}
\end{eqnarray}
The costructures are 
$\Delta(u)=u \overset{.}{\otimes} u$, $\Delta(D^{-1})=D^{-1}\otimes D^{-1}$, and 
$\varepsilon(u)=  \mathbb{I}_2$, $\varepsilon(D^{-1})= 1$,
while the antipode is
\beq\label{antipU2}
 S(u)= {D}^{-1} \begin{pmatrix}
\delta & -q^{-1} \beta
\\ -q \gamma & \alpha
\end{pmatrix} ~,\qquad S(D^{-1}) = D \, . 
\eeq 
The $*$-structure defining the real form $\O(U_{q}(2))$ requires the
deformation parameter to be a phase,  we set $q=exp(2\pi i
\theta)$ as in \eqref{crA}. The $*$-structure is then given
by 
\beq\label{*A}
\begin{pmatrix} \alpha^* & \beta^* \\ \gamma^* & \delta^*\end{pmatrix}
=
{D}^{-1} \begin{pmatrix}
\delta & -q \gamma
\\ -q^{-1} \beta & \alpha
\end{pmatrix}~, \qquad(D^{-1})^*=D~.
\eeq
The defining relations \eqref{rel-inv} and
\eqref{id-inv} are the same as the FRT commutation relations
\eqref{FRTU2} and
the relations $D=\alpha\delta-q^{-1}\beta\gamma$, $DD^{-1}=1$ 
with  $\delta=D\alpha^*$ and $\gamma=-q^{-1}D\beta^*$
as given in \eqref{*A}. The costructures and antipode too in
\eqref{copz2} and \eqref{antiz2} are  those of $\O(U_q(2))$,
thus showing the isomorphism $\big(\O(SU(2))\otimes
\O(U(1))\big)_\gamma^{\IZ_2}\simeq\O(U_q(2))$ as Hopf $*$-algebras. We
have obtained the quantum group $\O(U_q(2))$ from the ``quantum double cover'' 
$\big(\O(SU(2))\otimes \O(U(1))\big)_\gamma$.\footnote{The coproduct \eqref{copz2} and
  antipode \eqref{antiz2} also show a semidirect structure of the Hopf
  $*$-algebra $\O(U_q(2))$
that corresponds to the semidirect product $SU(2) \rtimes U(1) \simeq U(2)
\simeq (SU(2) \times U(1)) / \IZ_2 
$, obtained by decomposing $U(2)$ matrices 
as {\tiny{$
 \begin{pmatrix} \alpha & \beta \\ \gamma & \delta\end{pmatrix}\,=\,
 \begin{pmatrix} \alpha & \beta \\ -D\beta^* & D\alpha^*\end{pmatrix}=
 \begin{pmatrix} 1 & 0 \\ 0 & D\end{pmatrix}
 \begin{pmatrix} \alpha & \beta \\ -\beta^* & \alpha^*\end{pmatrix}\,.$}}}
Finally, the coaction 
$$
\big(\O(S^7)\ot \O(S^1)\big)_\gamma \to
\big(\O(S^7)\ot \O(S^1)\big)_\gamma \otimes \big(\O(SU(2))\times \O(U(1))\big)_\gamma
$$ 
is even and therefore,  restricted to $\big(\O(S^7)\ot
\O(S^1)\big)_\gamma^{\IZ_2} $, defines an $\O(U_q(2))$-coaction 
$$
\big(\O(S^7)\ot
\O(S^1)\big)_\gamma^{\IZ_2} \to \big(\O(S^7)\ot
\O(S^1)\big)_\gamma^{\IZ_2} \ot \O(U_q(2)).
$$
It then follows that the subalgebra  $\O(S^4)\subseteq\big(\O(S^7)\ot \O(S^1)\big)_\gamma$  
of $\big(\O(SU(2))\ot\O(U(1))\big)_\gamma$-coinvariants, being even, 
coincides with the subalgebra  $\O(S^4)\subseteq\big(\O(S^7)\ot\O(S^1)\big)_\gamma^{\IZ_2}  $  of 
$\O(U_q(2))$-coinvariants. 
Furthermore, since the canonical map of the initial Hopf--Galois extension 
$\O(S^4)\subseteq\big(\O(S^7)\ot\O(S^1)\big)_\gamma$ is even, 
the extension $\O(S^4)\subseteq\big(\O(S^7)\ot\O(S^1)\big)_\gamma^{\IZ_2}$ is Hopf--Galois as well.

We further observe that the $\O(U_q(2))$-Hopf--Galois extension is a 2-cocycle deformation 
of the commutative  $\O(U(2))$-Hopf--Galois extension: the 2-cocycle on $\O(U(2))\simeq\big(\O(SU(2))\otimes \O(U(1))\big)
^{\IZ_2}$ is the restriction of the one on $\O(SU(2))\otimes \O(U(1))$. The corresponding
deformation of $\big(\O(S^7)\ot\O(S^1)\big)^{\IZ_2}$ via the $\O(U(2))$-coaction is then the deformation of $\big(\O(S^7)\ot\O(S^1)\big)^{\IZ_2}\subseteq \O(S^7)\ot\O(S^1)$ via the $\O(SU(2))\otimes \O(U(1))$ coaction. 
From Proposition  \ref{thm:gamma-ad} we conclude that the
gauge group of this noncommutative $\O(U_q(2))$-Hopf--Galois extension is undeformed.
\qed\end{ex}

\addtocontents{toc}{\protect\setcounter{tocdepth}{1}}

\appendix
\section{The canonical map as a morphism of relative Hopf modules}\label{app:hge}

As mentioned in \S \ref{sec:hge}, for a generic Hopf algebra $H$ (that is, not necessarily coquasitriangular), the canonical map \eqref{canonical}  of a Hopf--Galois extension $B\subseteq A$,
\begin{align*}
\can = (m \ot \id) \circ (\id \otimes_B \delta^A ) : A \otimes_B A \longrightarrow A \ot H ~  ,   
\quad  a' \ot_B a \longmapsto a' \zero{a} \ot \one{a} ,
\end{align*} 
 was shown in \cite[\S 2]{ppca} to be a
  a morphism in the category ${_A{\mathcal  M}_A}^{H}$ of relative Hopf modules.

It was proved in \cite[\S 1.1]{Sch90} that 
$\can$ is a morphism in ${{\mathcal  M}_A}^{H}$ when 
$A \otimes_B A$ and $A \ot H$ are seen as objects in ${{\mathcal  M}_A}^{H}$ with right $A$-module structures  
\beq \label{m2}
(a \ot_B a')a'':=  a \ot_B a'a'' \quad \mbox{and} \quad (a \ot h) a'= a \zero{a'} \ot h \one{a'}
\eeq
and right $H$-coactions: for all $a,a',a'' \in A$, $h \in H$, 
\beq \label{m3}
 a \ot_B a' \mapsto a \ot \zero{a'}  \ot \one{a'} \quad \mbox{and} \quad 
 a \ot h \mapsto a \ot \one{h} \ot \two{h}  \, .
 \eeq

Moreover,  $\can$ is a morphism in ${_A{\mathcal  M}}^{H}$ when 
$A \otimes_B A$ and $A \ot H$ are considered to be objects in ${_A{\mathcal  M}}^{H}$ with left $A$-module structures given by left multiplication on the first factors and right $H$-coactions:
for all $a,a' \in A$, $h \in H$,
\beq\label{m1}
a \ot_B a' \mapsto \zero{a} \ot a' \ot \one{a} \quad \mbox{and} \quad 
a \ot h \mapsto \zero{a} \ot \two{h} \ot \one{a}S(\one{h}) \, .
\eeq

In \cite[\S 2]{ppca} both $A \otimes_B A$ and $A \ot H$ 
were shown to be objects in  ${_A{\mathcal M}_A}^{H}$, with $\can$ a morphism in the category ${_A{\mathcal  M}_A}^{H}$ of relative Hopf modules. As  already recalled in  \S \ref{sec:hge}, 
the left 
$A$-module structures are the left multiplication on the first factors and the 
right $A$-actions as in \eqref{m2}.
The tensor product $A\ot_B A$ carries the 
right $H$-coaction in \eqref{AAcoact}: 
\beq\label{app-AAcoact}
\delta^{A\otimes_B A}: A\otimes_B A\to A\otimes_B A\otimes H, \quad a\otimes_B a'  \mapsto  \zero{a}\otimes_B \zero{a'} \otimes  \one{a}\one{a'} \, 
\eeq 
for all $a,a'\in A$.   
The right $H$-coaction on $A \ot H$ 
is given by \eqref{AHcoact}: for all $ a\in A,~ h \in H$,
\beq\label{app-AHcoact}
\delta^{A\otimes H}(a\otimes h) = \zero{a}\otimes \two{h} \otimes \one{a}\,S(\one{h})\, \three{h} \in A \ot H \ot H \, .
\eeq 

These two approaches can be related: the coactions \eqref{app-AAcoact} and \eqref{app-AHcoact} can be obtained as the compositions of the coactions \eqref{m1} and \eqref{m3} in the sense of the following lemma. 
\begin{lem} 
Let $(V,\delta_1,\delta_2)$ be a relative Hopf module in $\M^{H,H}$, that is $V$ is a $\bbK$-module endowed with two coactions $\delta_1: v \mapsto \zero{v} \ot \one{v}$ and $\delta_2:v \mapsto \uzero{v} \ot \uone{v}$ of an Hopf algebra $H$ compatible in the sense that 
\beq
(\delta_1 \ot \id_H)\circ \delta_2 = (\id_V \ot \textup{flip}) \circ (\delta_2 \ot \id_H)\circ \delta_1
\eeq
that is  for all $v \in V$ 
\beq \label{MHH}
\zero{(\uzero{v})} \ot \one{(\uzero{v})} \ot  \uone{v}=
\uzero{(\zero{v})} \ot \one{v} \ot \uone{(\zero{v})}  \, .
\eeq
Then the compositions
\beq\label{comp-coactions}
\delta_2 \circ \delta_1:=(\id_V \ot m_H)\circ  (\delta_2 \ot \id_H)\circ \delta_1 
: \quad v \mapsto  \uzero{(\zero{v})}  \ot \uone{(\zero{v})} \one{v} 
\eeq
\beq\label{comp-coactions2}
\delta_1 \circ \delta_2:=(\id_V \ot m_H)\circ  (\delta_1 \ot \id_H)\circ \delta_2 
: \quad v \mapsto  \zero{(\uzero{v})} \ot \one{(\uzero{v})}  \uone{v} 
\eeq
define new coactions on $V$.
\end{lem}
\begin{proof}
Notice that condition \eqref{MHH} is symmetric for the exchange $\delta_1 \leftrightarrow \delta_2$ and so it is enough to prove the result for (say) $\delta:=\delta_2 \circ \delta_1$. 
It is easy to verify that $(\id_V \ot \varepsilon)\circ \delta =\id_V$.
We have to show  that 
$(\id_V \ot \Delta)\circ \delta = (\delta \ot \id_H)\circ \delta$.  Let $v \in V$, 
then
\begin{align*}
(\id_V \ot \Delta)\circ \delta (v) & = 
\uzero{(\zero{v})}  \ot \Delta(\uone{(\zero{v})} \one{v}  )
\\
 & = 
\uzero{(\zero{v})}  \ot \one{(\uone{(\zero{v})})} \one{v}   \ot \two{(\uone{(\zero{v})})} \two{v}
\\
 & = 
\uzero{(\zero{v})}  \ot \uone{(\zero{v})} \one{v}   \ot \utwo{(\zero{v})} \two{v}
\end{align*}
where we  used the fact that $\delta_2$ is a comodule map:
$$\uzero{v}  \ot \one{(\uone{v})}   \ot \two{(\uone{v})}  =
\uzero{(\uzero{v})}  \ot \uone{(\uzero{v})}   \ot \uone{v}  =:
\uzero{v}  \ot \uone{v}   \ot \utwo{v}  \, ,
$$
for $v \in V$. On the other hand
\begin{align*}
(\delta \ot \id_H)\circ \delta (v) &=
\delta(\uzero{(\zero{v})})  \ot \uone{(\zero{v})} \one{v} 
\\
&=
\uzero{(\zero{(\uzero{(\zero{v})})})}  \ot \uone{(\zero{(\uzero{(\zero{v})})})} \one{(\uzero{(\zero{v})})}  \ot \uone{(\zero{v})} \one{v} 
\\
&=
\zero{(\uzero{(\uzero{(\zero{v})})})}  \ot \uone{(\uzero{(\zero{v})})} \one{(\uzero{(\uzero{(\zero{v})})})}  \ot \uone{(\zero{v})} \one{v} 
\\
&=
\zero{(\uzero{(\zero{v})})}  \ot \uone{(\zero{v})} \one{(\uzero{ (\zero{v})})}  \ot \utwo{(\zero{v})} \one{v} 
\end{align*}
where the third equality follows from the condition \eqref{MHH} on $\uzero{(\zero{v})}$ and the last one from the fact that $\delta_2$ is a comodule map.
By using once again condition \eqref{MHH} on $\zero{v}$ we obtain
\begin{align*}
(\delta \ot \id_H)\circ \delta (v) &=
\uzero{(\zero{(\zero{v})})}  \ot  \uone{(\zero{ (\zero{v})})} \one{(\zero{v})} \ot \utwo{(\zero{v})} \one{v} 
\\
&=
\uzero{(\zero{v})}  \ot  \uone{(\zero{v})} \one{v} \ot \utwo{(\zero{v})} \two{v} 
\end{align*}
where the last equality uses that $\delta_1$ is a comodule map. By comparison with the formula obtained before for  $(\id_V \ot \Delta)\circ \delta (v)$ we can conclude that $\delta$ is a comodule map.
\end{proof}

On the $H$ comodule $A \ot_B A$, with  $\delta_1$ the coaction in \eqref{m3} and $\delta_2$ the coaction in \eqref{m1}, the
condition \eqref{MHH} is satisfied and  their composition $\delta_2 \circ \delta_1$, defined as in \eqref{comp-coactions}, is just the coaction  in \eqref{app-AAcoact}. On the other hand, on the $H$ comodule $A \ot H$ with coactions $\delta_1$ as in \eqref{m3} and $\delta_2$ as in \eqref{m1},  condition \eqref{MHH} is satisfied as well and  their composition $\delta_2 \circ \delta_1$ is the tensor product  coaction  in \eqref{app-AHcoact}.

\subsection{The translation map as a morphism of relative Hopf modules}\label{app:t}

The condition for the canonical map $\chi$ to be a morphism of
relative Hopf modules can  equivalently be expressed in terms of its
inverse. In particular,  it allows to infer the following properties
of  the translation map $\chi^{-1}{|_{_{1 \bot \underline{H}}}}= \tau:  H \ra A \ot_B A$ (cf. \cite[Prop.3.6]{brz-tr}).

The condition that $\can^{-1}$ is an $H$-comodule map with respect to the $H$-coactions in \eqref{m3} is  $\chi^{-1} \circ(\id \ot \Delta)= (\id \ot_B \delta^A)\circ \chi^{-1}$. This identity, restricted to $1 \ot H$,  implies the following property of the translation map:
$$
(\id \ot_B \delta^A)\circ \tau =(\tau \ot \id) \circ\Delta ~.
$$
That is,  $\tau:  H \ra A \ot_B A$ is an $H$-comodule map for $H$ with coaction given by $\Delta$ (that is the coaction in \eqref{m3} restricted to $1 \ot H$) and $A \ot_B A$ with coaction as in $\eqref{m3}$.

Similarly, the condition for  $\chi^{-1}$ to be a morphism in ${{\mathcal  M}}^{H}$ with respect to the $H$-coactions in \eqref{m1} gives
$$
[(\id \ot \mbox{flip}) \circ (\delta^A \ot_B \id)] \circ \tau=(  \tau  \ot  \id) \circ [(  \id  \ot  S) \circ \mbox{flip} \circ \Delta]  ~.
$$
Finally, being $\can_{|_{1 \ot H}}$ a comodule morphism with respect to the $H$-coactions \eqref{app-AAcoact} 
and \eqref{app-AHcoact}, one obtains that
$$
\delta^{A\otimes_B A} \circ \tau = (\tau \ot \id) \circ \Ad ~.
$$

\noindent
\textbf{Acknowledgments.}~\\[.5em]
We thank Rita Fioresi, Catherine Meusburger and  Alexander Schenkel for
 fruitful discussions. 
All authors are members of COST Action MP1405 \textit{QSPACE}. 
PA acknowledges support and hospitality from ICTP - 
INFN during his scientific visit in Trieste, and partial support from INFN, CSN4, Iniziativa
Specifica GSS and from INdAM-GNFM. 
This research has a financial support from Universit\`a del Piemonte Orientale. GL acknowledges partial support from INFN, Iniziativa
Specifica GAST and from INdAM-GNSAGA.
Part of the work of CP was carried out at  Universit\'e Catholique de Louvain (IRMP) in Louvain-la-neuve. CP  gratefully acknowledges support  by the Belgian Scientific Policy under IAP grant DYGEST, 
as well as support by  COST (Action MP1405)
 and INFN Torino during her scientific visits in
Torino.

\end{document}